\pdfoutput=1

\newif\ifpersonal
\personaltrue 

\documentclass[11pt,a4paper,dvipsnames,x11names]{amsart}

\usepackage[utf8]{inputenc}
\usepackage[T1]{fontenc}
\usepackage[english]{babel}

\usepackage[a4paper,margin=1in]{geometry}

\usepackage{microtype,lmodern}
\usepackage{mathpazo}
\usepackage{euler}

\usepackage{amsmath,amssymb,amsthm,mathtools,mathrsfs,bm,eucal,tensor}
\usepackage{stmaryrd}
\usepackage{dsfont}

\usepackage[all,cmtip]{xy}
\usepackage{tikz}
\usetikzlibrary{arrows.meta,calc,positioning,fit,backgrounds,shapes.geometric}
\usepackage{tikz-cd}

\tikzset{
  >=Stealth,
  box/.style={draw, rounded corners, inner sep=6pt},
  v/.style={draw, circle, inner sep=1.4pt},
  arr/.style={->, thick},
  darr/.style={->, thick, dashed}
}

\usepackage{enumerate,comment,braket,xspace,csquotes}

\usepackage{xcolor}
\usepackage{hyperref}
\hypersetup{
  colorlinks=true,
  linkcolor=blue!60!black,
  citecolor=blue!60!black,
  urlcolor=blue!60!black
}
\usepackage[capitalize,nameinlink]{cleveref}

\numberwithin{equation}{section}

\theoremstyle{plain}
\newtheorem{theorem}{Theorem}[section]
\newtheorem{proposition}[theorem]{Proposition}
\newtheorem{lemma}[theorem]{Lemma}
\newtheorem{corollary}[theorem]{Corollary}

\theoremstyle{definition}
\newtheorem{definition}[theorem]{Definition}
\newtheorem{remark}[theorem]{Remark}

\newtheorem{stokesproblem}[theorem]{Problem}

\newcommand{\Xlog}{X^{\log}}
\newcommand{\Xlogn}{X^{\log}_n}

\newcommand{\CC}{\mathbb C}
\newcommand{\CB}{\Cech(\mathcal B)}

\newcommand{\ZZ}{\mathbb Z}

\newcommand{\id}{\mathrm{id}}
\newcommand{\Ob}{\mathrm{Ob}}
\newcommand{\Hom}{\mathrm{Hom}}
\newcommand{\Aut}{\mathrm{Aut}}

\newcommand{\Sec}{\operatorname{Sec}}

\newcommand{\StokesLocal}{\mathrm{StokesLocal}}
\newcommand{\Stokes}{\mathrm{Stokes}}

\newcommand{\LocSys}{\mathrm{LocSys}}
\newcommand{\Tors}{\mathrm{Tors}}
\newcommand{\Rep}{\mathrm{Rep}}

\newcommand{\Cech}{\check{\mathrm C}}
\newcommand{\GL}{\mathrm{GL}}
\newcommand{\Cocone}{\mathrm{Cocone}}

\newcommand{\CoconeTwo}{\Cocone_{2}}

\newcommand{\St}{\mathrm{St}}
\newcommand{\Gr}{\mathrm{Gr}}

\newcommand{\GCphisk}{G_{C,\Phi}^{\mathrm{sk}}}
\newcommand{\GUs}{G_{U}^{\mathrm{sk}}}

\newcommand{\G}{\mathscr G}

\newcommand{\GU}{\G_U}
\newcommand{\GX}{\G_{\times}}
\newcommand{\GC}{\G_C}

\newcommand{\GXsk}{\GX^{\mathrm{sk}}}

\newcommand{\CBn}{\Cech(\mathcal B_n)}
\newcommand{\GCn}{\mathcal G_{C,n}}

\title{A \v{C}ech--Stokes Pushout Groupoid: A Log/Kummer Betti Presenter for Stokes Torsors}
\author{Mauricio   Corr\^ea  }
\address{Dipartimento di Matematica, Universit\`a degli Studi di Bari Aldo Moro, Bari, Italy}
\email{mauricio.barros@uniba.it}
\date{}

\begin{document}

\begin{abstract}
We give an explicit Betti presentation of the Stokes torsors attached to meromorphic flat connections of prescribed irregular type along a simple normal crossings divisor, at fixed Kummer level. Our construction is strictly \(1\)-categorical and cover-based: on a punctured logarithmic collar of the divisor, we define a small \v{C}ech--Stokes groupoid and prove that the boundary Stokes moduli is described by sections of a natural forgetful functor, rather than by representations of the decorated boundary groupoid itself. By gluing this boundary model to the ordinary \v{C}ech presenter of the complement through an explicit pushout construction, we obtain a groupoid presentation that computes the global Stokes objects. The resulting presentation is canonical up to Morita equivalence, compatible with Kummer descent along all normal crossings strata, and admits an explicit finite local description near corners in terms of nonabelian cocycles and relations.
\end{abstract}

\maketitle

\section{Introduction}

Irregular singular connections and Stokes phenomena may be viewed, at their core, as a problem of
\emph{topological encoding of analytic continuation near a divisor}.
Let $X$ be a complex manifold (or a smooth complex algebraic variety), and let $D\subset X$ be a simple normal crossings divisor.
Write $U:=X\setminus D$.
A meromorphic flat connection on $(X,D)$ restricts on $U$ to an ordinary local system, but its behavior near $D$
involves additional, genuinely non-abelian gluing data: the \emph{Stokes jumps}.
In dimension one this is classical (see e.g.\ \cite{BV89,Sabbah,Boalch2014}); in higher dimension, the same principle persists,
with one essential geometric change:
the relevant ``directions'' around $D$ vary along the strata of $D$ and must be handled functorially across charts.
A Betti-side model that is useful for geometry should therefore keep track simultaneously of
(i) the directional boundary geometry, (ii) tame/Kummer branching along all strata, and
(iii) the wild Stokes gluing across sectorial overlaps.

A standard form of the irregular Riemann--Hilbert correspondence asserts that meromorphic connections
with prescribed formal/irregular type are classified on the Betti side by \emph{Stokes-filtered}
(or Stokes-graded) local systems on a real boundary space surrounding $D$.
Analytically, this boundary is often realized by the real oriented blow-up along $D$,
while modern formulations use enhanced or perverse-sheaf avatars of irregularity (cf.\ \cite{DAgnoloKashiwara,Mochizuki}).
In the Stokes-filtered description, the Stokes jumps are encoded by a Stokes sheaf on the boundary,
and the classification problem becomes a torsor problem.
While this framework is well-established, the Betti objects are frequently presented either analytically
or as stacks defined by descent; in particular, the underlying \emph{presenting groupoid} is often left implicit.
For applications where one wants strict gluing along covers (for instance as input to Betti-stack or Tannakian formalisms),
it is useful to have a small and explicit strict groupoid model adapted to normal crossings strata and compatible with tame/Kummer ramification.

A complementary viewpoint, especially effective on curves, is to encode wild monodromy by groupoids:
Stokes data can be organized into ``wild'' fundamental groupoids, and moduli can be described as representation groupoids
of these presenters (see e.g.\ \cite{GualtieriLiPym} and the references therein).
This has two structural advantages: it makes cutting-and-pasting arguments transparent, and it interfaces naturally with
geometric structures on moduli (Poisson/symplectic forms, generators-and-relations descriptions, and related constructions).
In higher dimension, however, one faces an additional bookkeeping issue: the boundary geometry is naturally \emph{logarithmic},
and the tame/Kummer layer interacts with the Stokes layer along strata of arbitrary depth.
The aim of this paper is to give a single, strictly $1$-categorical, cover-based model in which the directional, tame, and wild layers
appear together through explicit maps of small groupoids governed by universal properties (pullbacks and pushouts),
and which admits finite skeletal models well-suited to computation (compare the \textit{skeleton} perspective in \cite{TeyssierSkeletons}).

We work with the logarithmic Betti realization of the pair $(X,D)$.
Let $\Xlog$ denote the Kato--Nakayama space of $(X,D)$, equipped with the canonical map $\tau:\Xlog\to X$ \cite{KatoNakayama}.
Locally, near a point where $k$ components of $D$ meet, the fiber of $\tau$ is a real torus $(S^1)^k$,
canonically recording the angular directions around each local branch.
To incorporate tame ramification uniformly we pass to the Kummer level $n$ and use the Kummer log cover
$\Xlogn\to \Xlog$ (cf.\ \cite{NakayamaOgus,Ogus}); along a depth-$k$ stratum the natural deck group is $(\mu_n)^k$.
This is the bookkeeping forced by multi-branch behavior once one insists on a boundary model compatible with descent.

A recurring point of terminology concerns the word \textit{representation}.
If $G$ is a small (topological) groupoid, we write $\Rep_r(G)$ for the groupoid of \emph{torsorial} data on $G$:
principal $\GL_r(\CC)^\delta$-bundles on $G_0$ equipped with descent isomorphisms along $G_1$
(here $\GL_r(\CC)^\delta$ denotes the discrete group).
This unpointed notion is stable under refinement and Morita equivalence, and it is the natural target for strict gluing arguments.
When a basepoint and a trivialization are fixed on a connected space, one recovers the familiar $\Hom(\pi_1,-)$ description,
but we do not impose such auxiliary choices.

Fix an irregular type $\Phi$ along $D$, and let $N^\times\subset \Xlogn$ be a punctured collar neighborhood of the logarithmic boundary.
We consider the groupoid $\LocSys_r(U)$ of rank-$r$ local systems on $U$ and the groupoid
$\StokesLocal_r(N^\circ;\Phi)$ of rank-$r$ Stokes local systems on the collar of type $\Phi$,
together with their restrictions to the overlap $N^\circ\subset U$.
We define the global Stokes object intrinsically as the $2$-fiber product
\begin{equation}\label{eq:intro-global-stokes}
\Stokes_r(\Xlogn;D,\Phi)
:=
\LocSys_r(U)\times^{(2)}_{\LocSys_r(N^\circ)}
\StokesLocal_r(N^\circ;\Phi).
\end{equation}
This formalizes the basic gluing principle: an irregular object is obtained by gluing an ordinary local system on $U$
with a Stokes enhancement on a collar, compatibly on the overlap.

Our main technical step is to present \eqref{eq:intro-global-stokes} by an explicit \emph{small strict groupoid}
built from \v{C}ech data.
Choose a good cover of $U$ and a sectorial/Stokes cover of $N^\times$ compatible on $U\cap N^\times$.
Let $\GU$ be the \v{C}ech groupoid of the interior cover, let $\GC$ be the \v{C}ech--Stokes collar groupoid on $N^\times$,
and let $\GX$ be the induced overlap groupoid.
The collar groupoid comes equipped with a canonical projection
$\pi:\GC\to \Cech(\mathcal B)$ to the underlying sector \v{C}ech groupoid (forgetting the Stokes labels);
the correct collar moduli is the section groupoid $\Sec(\pi)$, rather than a representation groupoid of $\GC$.
We then form the explicit $2$-pushout in small groupoids
\begin{equation}\label{eq:pushout-intro}
\G_{\Phi,n}\ :=\ \GU \sqcup_{\GX} \GC.
\end{equation}
The pushout is the groupoid-theoretic expression of \textit{glue interior monodromy and boundary Stokes torsors
along their common restriction.}

\begin{theorem} \label{thm:intro-main}
For every $r\ge 1$ there is a natural equivalence of groupoids
\begin{equation}\label{eq:intro-main-equiv}
\Rep_r(\GU)\times^{(2)}_{\Rep_r(\GX)} \Sec(\pi)
\ \simeq\ \Stokes_r(\Xlogn;D,\Phi),
\end{equation}
compatible with restriction to $\Rep_r(\GU)$ (forgetting the wild layer) and with restriction to $\Sec(\pi)$
(forgetting the interior).
\end{theorem}

Although the global Stokes groupoid \eqref{eq:intro-global-stokes} is defined intrinsically as a $2$-fiber product,
the main point of this paper is to provide an explicit small strict presenter for it, adapted to SNC depth and Kummer descent.
Concretely, we:
\begin{itemize}
\item construct a strict \v{C}ech--Stokes collar groupoid $\GC$ and identify the correct collar moduli as $\Sec(\pi)$
(rather than $\Rep_r(\GC)$), making the Stokes layer strictly functorial and cover-based;
\item build the explicit $2$-pushout presenter $\G_{\Phi,n}$ and prove that it computes global Stokes objects via the torsorial gluing problem
in \Cref{thm:intro-main};
\item make the tame/Kummer layer along depth strata functorial at the level of presenters, enabling explicit descent constraints
and, in particular, a hands-on description in terms of equivariant cocycles;
\item provide finite skeletal models (invariant up to Morita equivalence, with contractible ambiguity) and an explicit
generators-and-relations/chamber-complex description in SNC corners, solving the local presentation problem highlighted in
Problem~12.3.
\end{itemize}

Theorem~\Cref{thm:intro-main} makes it possible to treat global Stokes objects using a concrete presenter:
questions about existence, comparison, and descent can be translated into torsors and non-abelian \v{C}ech cocycles
on an explicit skeleton underlying $\G_{\Phi,n}$.
We record three consequences that will be used later; each is proved in the indicated section.

\begin{corollary} \label{cor:intro-morita}
If one replaces the interior and collar covers by refinements, the resulting pushout groupoid $G'_{\Phi,n}$
is Morita equivalent to $\G_{\Phi,n}$; in particular the associated torsorial representation groupoids are equivalent:
$\Rep_r(G'_{\Phi,n})\simeq \Rep_r(\G_{\Phi,n})$.
\end{corollary}

\begin{corollary} \label{cor:intro-kummer}
Along a depth-$k$ stratum of $D$, the Kummer cover $\Xlogn\to \Xlog$ has deck group $(\mu_n)^k$,
and the groupoid $\G_{\Phi,n}$ records this tame layer functorially through the collar groupoid $\GC$.
Consequently, Kummer descent constraints for Stokes objects can be expressed as strict descent conditions
on the collar data (equivalently, as equivariant cocycle constraints on a finite skeleton).
\end{corollary}

\begin{corollary} \label{cor:intro-van-kampen}
Let $X= X^{(1)}\cup X^{(2)}$ be an open cover compatible with $D$, and assume the sectorial/Stokes data are chosen compatibly on overlaps.
Then the presenting groupoids admit a pushout description reflecting this decomposition, so that questions about
$\Stokes_r(\Xlogn;D,\Phi)$ may be reduced to representations of an explicit groupoid obtained by gluing the local presenters.
 
\end{corollary}

For a fully explicit local model (including a finite Morita skeleton), the reader may jump directly to
\Cref{subsec:curve-model-clean}.
There we compute the presenting groupoid near a simple normal crossings corner and spell out the resulting
generators-and-relations (equivalently, non-abelian cocycle) description of the Stokes data.
This example is included as a concrete reference point: it displays, in a single diagram, the directional layer
($\Xlog$), the tame/Kummer layer (the $(\mu_n)^k$-actions along strata), and the wild Stokes layer
(sector overlaps and Stokes jumps), and it provides a template for later computations.

After recalling the basic geometry of $\Xlogn$ in the SNC case and fixing conventions, we define the intrinsic groupoid
$\Stokes_r(\Xlogn;D,\Phi)$ and establish its torsor description.
We then construct the interior \v{C}ech groupoid $\GU$, the collar \v{C}ech--Stokes groupoid $\GC$, and the overlap groupoid $\GX$,
form the pushout \eqref{eq:pushout-intro}, and prove \Cref{thm:intro-main} in \Cref{sec:pushout}.
Finally, we discuss Morita invariance, Kummer descent along depth strata, and an explicit coequalizer
(generators-and-relations) model for $\G_{\Phi,n}$, including the SNC computation announced above.

\section{Standing conventions and geometric input}\label{sec:conventions}

\subsection{SNC divisor and logarithmic boundary spaces}
Let $X$ be a complex manifold and let $D\subset X$ be a \emph{reduced simple normal crossings} divisor.
Set $U:=X\setminus |D|$.

Let $X^{log}$ denote the Kato--Nakayama space associated to the log structure of $(X,D)$, and let $X^{log}_n$ be the level-$n$ Kummer variant.
Write $\tau_n:X^{log}_n\to X$ for the structure map and set $U^{log}_n:=\tau_n^{-1}(U)$.
Over $U$, the map $\tau_n$ is a homeomorphism, and we will freely identify $U^{log}_n$ with $U$ via $\tau_n$.

Locally near a point where $k$ branches of $D$ meet, there are coordinates $(z_1,\dots,z_k,w)$ with $D=\{z_1\cdots z_k=0\}$,
and the boundary directions form a torus $(S^1)^k$; at level $n$ the angular variables are $n$-fold covered.
For precise constructions in this toroidal/SNC setting, see \cite{KatoNakayama,Ogus}.

\subsection{Collars}

Choose once and for all a collar neighborhood $N\subset X^{log}_n$ of the boundary over $|D|$,
and set $N^\times:=N\cap U^{log}_n$ (the punctured collar).
Via $\tau_n|_{U^{log}_n}$ we regard $N^\times$ as an open subset of $U$ whenever we write $U\cap N^\times$.
Let $\Sigma_n\subset N^\times$ be the (locally finite) wall set attached to $\Phi$ at level $n$.
Set
\begin{equation*}
N^\circ := N^\times\setminus \Sigma_n.
\end{equation*}
We will use sector-box covers only on $N^\circ$, where the Stokes preorder is locally constant.

\subsection{Standing Stokes input on the regular collar $N^\circ$}\label{subsec:stokes-input}

Fix a Kummer level $n\ge 1$ and a punctured collar $N^\times\subset \Xlogn$ over $|D|$.
Let $\Sigma_n\subset N^\times$ be a locally finite real-analytic wall arrangement and set
$$
N^\circ:=N^\times\setminus \Sigma_n.
$$

We assume the following Stokes input on $N^\circ$ (treated as a black box, compatible with the usual irregular Riemann--Hilbert formalisms):
\begin{itemize}
\item a finite index set $\Phi$ (after the chosen Kummer level);
\item a graded local system on $N^\circ$
$$
\Gr=\bigoplus_{q\in\Phi}\Gr_q;
$$
\item on each connected component of $N^\circ$ a locally constant preorder $\preceq$ on $\Phi$, hence a standard filtration $F^{\mathrm{std}}_{\le q}$ on $\Gr$.
\end{itemize}

Such Stokes data arise, for instance, from the irregular Riemann--Hilbert correspondences of Deligne--Malgrange and their modern treatments; see \cite{Sabbah,Mochizuki}.

We further assume our sector-box cover $\mathcal B=\{B_\alpha\}$ of $N^\circ$ is locally finite, every nonempty finite intersection
$B_{\alpha_0\cdots \alpha_k}$ is connected and contractible, and the preorder $\preceq$ is constant on it.
In particular, any locally constant sheaf of groups on such an intersection has constant sections.

\subsection{Sector boxes and standard filtrations}
Fix a locally finite open cover $\mathcal B=\{B_\alpha\}$ of $N^\circ$ such that every nonempty finite intersection $B_{\alpha_0\cdots\alpha_k}$ is connected and contractible and the preorder $\preceq$ on $\Phi$ is constant on it.
In particular, the preorder is constant on each $B_\alpha$.

\begin{remark}\label{rem:locally-finite-boxes}
Even when $X$ is compact, the regular collar $N^\circ$ is typically non-compact (it contains a radial direction).
Thus we work with locally finite covers. In many geometric situations one may choose the boxes as products $(0,\varepsilon)\times V_\alpha$
and arrange finiteness in the boundary directions, but we do not rely on this.
\end{remark}

On each $B_\alpha$ define the \emph{standard filtration} $F^{std}_{\le q}$ of $\Gr|_{B_\alpha}$ by
\begin{equation}\label{eq:std-filtration}
F^{std}_{\le q}(\Gr)|_{B_\alpha}
:=
\bigoplus_{q'\preceq_\alpha q}\Gr_{q'}|_{B_\alpha},
\end{equation}
where $\preceq_\alpha$ is the constant preorder on $B_\alpha$.

\section{Groupoids, \v{C}ech groupoids, and torsorial representations}\label{sec:groupoids}

\subsection{Ambient 2-category and smallness}
We work in the $(2,1)$-category $\mathbf{Grpd}$ of small groupoids, strict functors, and natural isomorphisms.
Whenever a groupoid is built from an open cover, we implicitly keep the standard \v{C}ech \emph{topological} realization
(object and arrow spaces are disjoint unions of opens and their intersections, with the disjoint union topology).
Since our coefficient groups are discrete, continuity forces local constancy of transition data, so torsors on these \v{C}ech presenters
model local systems and Stokes torsors by strict gluing. For bicategorical universal properties (pushouts and pullbacks) we only use the
underlying small groupoids; Morita invariance under refinement guarantees independence of the chosen presenters.

All fiber products of groupoids in this paper are \emph{$2$-fiber products} (pseudopullbacks) in $\mathbf{Grpd}$.
Concretely, given $A\xrightarrow{F}C\xleftarrow{G}B$, the groupoid $A\times^{(2)}_C B$ has objects triples
$(a,b,\varphi)$ with $a\in A$, $b\in B$, and an isomorphism $\varphi:F(a)\xrightarrow{\sim}G(b)$ in $C$;
morphisms are pairs making the evident square commute.
We keep the notation $\times_C$ to avoid typographical overload.

\subsection{Groupoids and strict functors}
\begin{definition}
A \emph{groupoid} $G$ consists of a set $G_0$ of objects and a set $G_1$ of arrows, with source/target maps
$s,t:G_1\to G_0$, identity $e:G_0\to G_1$, inverse $i:G_1\to G_1$,
and composition $m:G_1\times_{G_0}G_1\to G_1$ satisfying the usual axioms.
A \emph{strict functor} $F:G\to H$ is a pair of maps $(F_0,F_1)$ commuting with all structure maps.
\end{definition}

 \subsection{The \v{C}ech groupoid of an open cover}
Let $T$ be a topological space and let $\mathcal V=\{V_i\}_{i\in I}$ be an open cover.
Define the \emph{\v{C}ech topological groupoid} $\Cech(\mathcal V)$ by
\begin{equation}\label{eq:cech-def}
\Cech(\mathcal V)_0 := \bigsqcup_{i\in I} V_i,
\qquad
\Cech(\mathcal V)_1 := \bigsqcup_{(i,j)\in I\times I} (V_i\cap V_j),
\end{equation}
with the disjoint union topology; source/target are induced by inclusions and composition is induced on triple overlaps.

\subsection{Torsorial ``representations'' of groupoids}
Fix $r\ge 1$ and write $\GL_r(\CC)^\delta$ for $\GL_r(\CC)$ with the discrete topology.

\begin{definition}\label{def:BGL}
Let $B\GL_r(\CC)^\delta$ be the one-object groupoid whose automorphism group is $\GL_r(\CC)^\delta$.
\end{definition}

\begin{definition}\label{def:rep-r}
Let $G$ be a groupoid.
A \emph{$B\GL_r(\CC)^\delta$-torsor over $G$} is a pair $(P,\alpha)$ where:
\begin{itemize}
\item $P\to G_0$ is a principal $\GL_r(\CC)^\delta$-bundle (equivalently, a $\GL_r(\CC)$-torsor in locally constant sheaves on $G_0$);
\item $\alpha:s^*P\xrightarrow{\sim} t^*P$ is an isomorphism of torsors over $G_1$, compatible with identities, inverses, and composition.
\end{itemize}
Morphisms are isomorphisms of torsors over $G_0$ commuting with $\alpha$.
We denote the resulting groupoid by $\Rep_r(G)$.
\end{definition}

\begin{remark}\label{rem:not-pi1}
If $T$ is connected and pointed, fixing a trivialization over the basepoint identifies $\LocSys_r(T)$ with representations of $\pi_1(T)$ up to conjugation.
Our torsorial definition avoids basepoints and is the natural one for descent, refinements, and gluing.
\end{remark}

\subsection{Representations as functors into the torsor groupoid}
Let $\mathbf{Tors}_r$ denote the groupoid whose objects are right $\GL_r(\CC)^\delta$-torsors in sets
and whose morphisms are equivariant bijections.

\begin{proposition}\label{prop:tors-BGL}
The groupoid $\mathbf{Tors}_r$ is (non-canonically) equivalent to $B\GL_r(\CC)^\delta$.
\end{proposition}

\begin{proof}
The groupoid $\mathbf{Tors}_r$ is connected: any torsor is isomorphic (non-canonically) to $\GL_r(\CC)^\delta$.
Fixing the standard torsor identifies automorphisms with $\GL_r(\CC)^\delta$, hence yields an equivalence with the one-object groupoid $B\GL_r(\CC)^\delta$.
\end{proof}

\begin{proposition}\label{prop:rep-as-fun}
For any small groupoid $G$, there is a natural equivalence of groupoids
\begin{equation}\label{eq:rep-as-fun}
\Rep_r(G)\simeq \Hom_{\mathbf{Grpd}}(G,\mathbf{Tors}_r),
\end{equation}
where the right-hand side is the groupoid of strict functors and natural isomorphisms.
\end{proposition}

\begin{proof}
This is the usual descent = functor dictionary for principal bundles on a groupoid:
given $(P,\alpha)$, send an object $x\in G_0$ to the fiber torsor $P_x$ and an arrow to the induced torsor isomorphism.
Conversely, a functor $G\to\mathbf{Tors}_r$ glues fibers into a torsor over $G_0$ with descent along $G_1$.
\end{proof}

\subsection{\v{C}ech groupoids present local systems}
\begin{proposition}[Cech presenters compute local systems]\label{prop:rep-cech-locsys}
Let $T$ be a topological space and let $\mathcal V=\{V_i\}_{i\in I}$ be a \emph{good cover} of $T$, meaning that every nonempty finite intersection
$V_{i_0\cdots i_k}:=V_{i_0}\cap\cdots\cap V_{i_k}$ is contractible.
Then for every $r\ge 1$ there is a natural equivalence of groupoids
\begin{equation}\label{eq:rep-cech-locsys}
\Rep_r(\Cech(\mathcal V)) \ \simeq\ \LocSys_r(T).
\end{equation}
\end{proposition}

\begin{proof}
Write $\G:=\Cech(\mathcal V)$.
By Definition~\ref{def:rep-r}, an object of $\Rep_r(\G)$ consists of principal $\GL_r(\CC)^\delta$-bundles $P_i\to V_i$ together with
isomorphisms of torsors on overlaps
\begin{equation}\label{eq:alphaij}
\alpha_{ij}:P_i|_{V_{ij}}\xrightarrow{\sim}P_j|_{V_{ij}},
\qquad V_{ij}:=V_i\cap V_j,
\end{equation}
satisfying the cocycle identity $\alpha_{jk}\circ \alpha_{ij}=\alpha_{ik}$ on each triple overlap $V_{ijk}$, and morphisms are families of
bundle isomorphisms compatible with the $\alpha_{ij}$.

Since $V_i$ is contractible and $\GL_r(\CC)^\delta$ is discrete, each $P_i$ is trivial; choose a section $s_i:V_i\to P_i$.
For $x\in V_{ij}$ there is a unique element $g_{ij}(x)\in \GL_r(\CC)^\delta$ such that
\begin{equation}\label{eq:alpha-sections}
\alpha_{ij}\bigl(s_i(x)\bigr)=s_j(x)\cdot g_{ij}(x).
\end{equation}
Because $V_{ij}$ is connected and the group is discrete, $g_{ij}$ is constant on $V_{ij}$.
Applying $\alpha_{jk}\circ\alpha_{ij}=\alpha_{ik}$ to $s_i$ on $V_{ijk}$ yields
\begin{equation}\label{eq:cech-cocycle}
g_{jk}\,g_{ij}=g_{ik}.
\end{equation}
Replacing $s_i$ by $s_i\cdot h_i$ for locally constant $h_i:V_i\to \GL_r(\CC)^\delta$ replaces $g_{ij}$ by
$g_{ij}'=h_j g_{ij}h_i^{-1}$. Thus an object of $\Rep_r(\G)$ determines a nonabelian \v{C}ech $1$-cocycle
$(g_{ij})$ on $\mathcal V$ with values in $\GL_r(\CC)^\delta$, well-defined up to the usual gauge action, and morphisms become gauge
transformations.

Conversely, given a \v{C}ech cocycle $(g_{ij})$ satisfying \eqref{eq:cech-cocycle}, glue the trivial torsors
$V_i\times \GL_r(\CC)^\delta$ by identifying, for $x\in V_{ij}$,
\begin{equation}\label{eq:gluing}
(x,a)\in V_i\times \GL_r(\CC)^\delta \sim (x, a\,g_{ij})\in V_j\times \GL_r(\CC)^\delta .
\end{equation}
The cocycle condition makes this identification transitive on triple overlaps, hence the quotient is a principal
$\GL_r(\CC)^\delta$-bundle $P\to T$. Since $\GL_r(\CC)^\delta$ is discrete, such torsors are equivalent to rank-$r$ local systems on $T$.
Gauge-equivalent cocycles yield isomorphic glued torsors, and a gauge $0$-cochain produces the corresponding isomorphism.
These constructions are functorial on morphisms and are inverse to each other up to canonical isomorphism, yielding the equivalence
\eqref{eq:rep-cech-locsys}. One may package the two functors in the diagram
\begin{equation}\label{eq:diagram-cech-locsys}
\begin{tikzcd}[row sep=large, column sep=huge]
\Rep_r(\Cech(\mathcal V))
\arrow[r, shift left=1.2, "\mathrm{glue}"]
&
\LocSys_r(T)
\arrow[l, shift left=1.2, "\mathrm{restrict\ and\ trivialize}"]
\end{tikzcd}
\end{equation}
and the good-cover hypothesis ensures that the auxiliary choices of trivializations do not affect the resulting object up to unique isomorphism.
References: \cite[\S 1--2]{BottTu} and \cite[Ch.\ III]{Giraud}.
\end{proof}

\section{The Stokes sheaf on the punctured collar}\label{sec:stokes-sheaf}

\subsection{Filtered and graded automorphisms; the Stokes sheaf}
For an open $W\subset B_\alpha$ define:
\begin{itemize}
\item $\Aut^{fil}(W)$: automorphisms of $\Gr|_W$ preserving all standard filtrations $F^{std}_{\le q}|_W$;
\item $\Aut^{gr}(W)$: graded automorphisms of $\Gr|_W$, i.e.\ preserving each direct summand $\Gr_q|_W$.
\end{itemize}
There is a natural homomorphism $gr:\Aut^{fil}(W)\to \Aut^{gr}(W)$.

\begin{definition}[Stokes sheaf]\label{def:stokes-sheaf}
The \emph{Stokes sheaf} (relative to $\Phi$ at level $n$) is the presheaf (in fact sheaf) of groups on the sector-box cover
given by
\begin{equation}\label{eq:stokes-sheaf}
\St_\Phi(W):=\ker\!\big(\Aut^{fil}(W)\xrightarrow{gr}\Aut^{gr}(W)\big).
\end{equation}
\end{definition}

\begin{remark}
On any contractible and connected $W\subset B_\alpha$, the local system $\Gr|_W$ is trivial and $\Aut(\Gr|_W)$ is constant (discrete).
Hence sections of $\St_\Phi$ on such $W$ are just group elements.
\end{remark}

\section{Stokes objects on the collar as torsors under $\St_\Phi$}\label{sec:local-stokes}

\subsection{Local Stokes objects on $N^\circ$}
\begin{definition}[Stokes-local objects on the punctured collar]\label{def:stokeslocal}
A \emph{Stokes-local object of rank $r$ on $N^\circ$ of type $\Phi$} is a tuple
\begin{equation}\label{eq:stokeslocal-object}
\big(L,\{L_{\le q}\}_\alpha,\iota,\{s_\alpha\}\big)
\end{equation}
where:
\begin{itemize}
\item $L$ is a rank-$r$ local system on $N^\circ$;
\item for each $\alpha$, $\{L_{\le q}|_{B_\alpha}\}_{q\in\Phi}$ is a filtration compatible with $\preceq_\alpha$;
\item $\iota:gr(L)\xrightarrow{\sim}\Gr$ is a graded identification (global);
\item for each $\alpha$, a splitting $s_\alpha: L|_{B_\alpha}\xrightarrow{\sim}\Gr|_{B_\alpha}$
such that $s_\alpha(L_{\le q})=F^{std}_{\le q}(\Gr)$ and $gr(s_\alpha)=\iota|_{B_\alpha}$.
\end{itemize}
Morphisms are isomorphisms preserving filtrations and commuting with $\iota$.
We denote the resulting groupoid by $\mathrm{StokesLocal}_r(N^\circ;\Phi)$.
\end{definition}

\subsection{Difference of splittings is Stokes-unipotent}
\begin{lemma}\label{lem:difference-splittings}
Let $(L,\{L_{\le q}\}_\alpha,\iota,\{s_\alpha\})$ be a Stokes-local object.
Fix $\alpha$ and let $s'_\alpha$ be another splitting satisfying the same conditions with the same graded identification $\iota$.
Then the automorphism
$$
u_\alpha:=s'_\alpha\circ s_\alpha^{-1}\in \Aut(\Gr|_{B_\alpha})
$$
lies in $\St_\Phi(B_\alpha)$.
\end{lemma}

\begin{proof}
For each $q\in\Phi$, the defining property of a splitting says that both $s_\alpha$ and $s'_\alpha$ identify the filtered pieces with the
standard filtration:
$$
s_\alpha\bigl(L_{\le q}|_{B_\alpha}\bigr)=F^{std}_{\le q}(\Gr)|_{B_\alpha},
\qquad
s'_\alpha\bigl(L_{\le q}|_{B_\alpha}\bigr)=F^{std}_{\le q}(\Gr)|_{B_\alpha}.
$$
Equivalently, for every $q$ the following squares commute:
\begin{equation*}
\begin{tikzcd}[row sep=large, column sep=huge]
L_{\le q}|_{B_\alpha} \arrow[r,hook] \arrow[d,"s_\alpha"'] & L|_{B_\alpha} \arrow[d,"s_\alpha"] \\
F^{std}_{\le q}(\Gr)|_{B_\alpha} \arrow[r,hook] & \Gr|_{B_\alpha}
\end{tikzcd}
\qquad
\begin{tikzcd}[row sep=large, column sep=huge]
L_{\le q}|_{B_\alpha} \arrow[r,hook] \arrow[d,"s'_\alpha"'] & L|_{B_\alpha} \arrow[d,"s'_\alpha"] \\
F^{std}_{\le q}(\Gr)|_{B_\alpha} \arrow[r,hook] & \Gr|_{B_\alpha}.
\end{tikzcd}
\end{equation*}
From this one reads off filtration preservation for $u_\alpha$: if $y\in F^{std}_{\le q}(\Gr)|_{B_\alpha}$ then
$x:=s_\alpha^{-1}(y)$ lies in $L_{\le q}|_{B_\alpha}$, hence
$u_\alpha(y)=s'_\alpha(x)$ lies again in $F^{std}_{\le q}(\Gr)|_{B_\alpha}$.
Thus $u_\alpha\in \Aut^{fil}(B_\alpha)$.
On associated gradeds, the defining condition $gr(s_\alpha)=\iota|_{B_\alpha}=gr(s'_\alpha)$ gives the commutative diagram
\begin{equation*}
\begin{tikzcd}[row sep=large, column sep=huge]
gr(L|_{B_\alpha}) \arrow[r,"gr(s_\alpha)"] \arrow[dr,swap,"gr(s'_\alpha)"] & \Gr|_{B_\alpha} \\
& \Gr|_{B_\alpha} \arrow[u,swap,"\id"]
\end{tikzcd}
\end{equation*}
and therefore
$$
gr(u_\alpha)=gr(s'_\alpha)\,gr(s_\alpha)^{-1}=\iota\,\iota^{-1}=\id,
$$
so $u_\alpha$ lies in the kernel of $gr:\Aut^{fil}(B_\alpha)\to \Aut^{gr}(B_\alpha)$.
By Definition~\ref{def:stokes-sheaf}, this kernel is precisely $\St_\Phi(B_\alpha)$.
\end{proof}

\subsection{Equivalence with $\St_\Phi$-torsors}
Let $\mathcal B=\{B_\alpha\}$.
Write $\Tors(\St_\Phi;\mathcal B)$ for the groupoid of \v{C}ech 1-cocycles
$u_{\alpha\beta}\in \St_\Phi(B_\alpha\cap B_\beta)$ modulo gauge $h_\alpha\in \St_\Phi(B_\alpha)$.

\begin{theorem}\label{thm:stokeslocal-tors}
There is a natural equivalence of groupoids
\begin{equation}\label{eq:stokeslocal-tors}
\mathrm{StokesLocal}_r(N^\circ;\Phi)\ \simeq\ \Tors(\St_\Phi;\mathcal B).
\end{equation}
\end{theorem}

\begin{proof}
Let $\mathcal B=\{B_\alpha\}$ be the fixed sector-box cover of $N^\circ$.
Starting from a Stokes-local object $(L,\{L_{\le q}\}_\alpha,\iota,\{s_\alpha\})$ as in Definition~\ref{def:stokeslocal},
define on each overlap $B_{\alpha\beta}:=B_\alpha\cap B_\beta$ the transition automorphism
\begin{equation}\label{eq:uab-def}
u_{\alpha\beta}:=s_\beta\circ s_\alpha^{-1}\in \Aut(\Gr|_{B_{\alpha\beta}}).
\end{equation}
The maps $s_\alpha$ and $s_\beta$ both identify the filtration on $L|_{B_{\alpha\beta}}$ with the standard filtration on $\Gr|_{B_{\alpha\beta}}$
and induce the same graded identification $\iota|_{B_{\alpha\beta}}$, hence $u_{\alpha\beta}$ preserves all $F^{std}_{\le q}$ and satisfies
$gr(u_{\alpha\beta})=\id$. In other words, the square
\begin{equation}\label{eq:transition-square}
\begin{tikzcd}[row sep=large, column sep=huge]
\Gr|_{B_{\alpha\beta}} \arrow[r,"u_{\alpha\beta}"] \arrow[d,swap,"gr"] &
\Gr|_{B_{\alpha\beta}} \arrow[d,"gr"] \\
gr(\Gr)|_{B_{\alpha\beta}} \arrow[r,equals] &
gr(\Gr)|_{B_{\alpha\beta}}
\end{tikzcd}
\end{equation}
commutes and $u_{\alpha\beta}$ lies in $\St_\Phi(B_{\alpha\beta})=\ker(\Aut^{fil}\to \Aut^{gr})$ by Definition~\ref{def:stokes-sheaf}.
On a triple overlap $B_{\alpha\beta\gamma}$ one has the tautological identity
\begin{equation}\label{eq:cocycle}
u_{\beta\gamma}\,u_{\alpha\beta}=u_{\alpha\gamma},
\end{equation}
since $s_\gamma\circ s_\beta^{-1}\circ s_\beta\circ s_\alpha^{-1}=s_\gamma\circ s_\alpha^{-1}$.
Thus $(u_{\alpha\beta})$ is a nonabelian \v{C}ech $1$-cocycle with values in $\St_\Phi$.
If we replace the splittings by $s_\alpha':=h_\alpha\circ s_\alpha$ with $h_\alpha\in \St_\Phi(B_\alpha)$,
then $u_{\alpha\beta}$ is replaced by
\begin{equation}\label{eq:gauge-u}
u'_{\alpha\beta}=h_\beta\,u_{\alpha\beta}\,h_\alpha^{-1},
\end{equation}
which is exactly the gauge action in $\Tors(\St_\Phi;\mathcal B)$. This defines a functor
\begin{equation}\label{eq:F-stokes-to-tors}
F:\mathrm{StokesLocal}_r(N^\circ;\Phi)\longrightarrow \Tors(\St_\Phi;\mathcal B).
\end{equation}

Conversely, let $(u_{\alpha\beta})$ be a \v{C}ech $1$-cocycle in $\St_\Phi(B_{\alpha\beta})$.
Glue the trivial graded local systems $\Gr|_{B_\alpha}$ on the cover $\mathcal B$ by identifying on overlaps
$\Gr|_{B_{\alpha\beta}}$ via $u_{\alpha\beta}$. The cocycle condition \eqref{eq:cocycle} makes this identification associative on triple overlaps,
so we obtain a local system $L$ on $N^\circ$ together with local isomorphisms
\begin{equation}\label{eq:glued-splittings}
s_\alpha:L|_{B_\alpha}\xrightarrow{\sim}\Gr|_{B_\alpha}.
\end{equation}
Because $u_{\alpha\beta}\in \Aut^{fil}$, the standard filtrations $F^{std}_{\le q}(\Gr)|_{B_\alpha}$ descend to filtrations
$L_{\le q}|_{B_\alpha}$ on $L|_{B_\alpha}$, and because $u_{\alpha\beta}\in\ker(gr)$ the graded identifications
$gr(s_\alpha):gr(L)|_{B_\alpha}\to \Gr|_{B_\alpha}$ glue to a global graded identification $\iota:gr(L)\xrightarrow{\sim}\Gr$.
This produces a Stokes-local object, and changing the cocycle by a gauge $0$-cochain $(h_\alpha)$ produces an isomorphic Stokes-local object.
Thus we obtain a functor
\begin{equation}\label{eq:G-tors-to-stokes}
G:\Tors(\St_\Phi;\mathcal B)\longrightarrow \mathrm{StokesLocal}_r(N^\circ;\Phi).
\end{equation}
The constructions $F$ and $G$ are inverse to each other up to canonical isomorphism: starting from Stokes-local data, one recovers the same local system
by gluing the $\Gr|_{B_\alpha}$ using the transition maps $s_\beta s_\alpha^{-1}$; starting from a cocycle, one recovers the same cocycle by comparing
the resulting local splittings. This yields the claimed equivalence \eqref{eq:stokeslocal-tors}.
For convenience, the two functors may be summarized as
\begin{equation}\label{eq:diagram-stokeslocal-tors}
\begin{tikzcd}[row sep=large, column sep=huge]
\mathrm{StokesLocal}_r(N^\circ;\Phi)
\arrow[r, bend left=30,
  "(\,s_\alpha\,)\mapsto (\,u_{\alpha\beta}=s_\beta s_\alpha^{-1}\,)"]
&
\Tors(\St_\Phi;\mathcal B)
\arrow[l, bend left=30,
  "(\,u_{\alpha\beta}\,)\mapsto \text{glue }\Gr|_{B_\alpha}"]
\end{tikzcd}
\end{equation}
and naturality is immediate from functoriality of restriction on the cover.
\end{proof}
\begin{proposition} \label{prop:refinement}
If $\mathcal B'$ refines $\mathcal B$, then restriction induces an equivalence
\begin{equation}\label{eq:refinement-equivalence}
\Tors(\St_\Phi;\mathcal B)\ \simeq\ \Tors(\St_\Phi;\mathcal B').
\end{equation}
In particular, $\mathrm{StokesLocal}_r(N^\circ;\Phi)$ is independent (up to canonical equivalence) of the chosen sector-box cover.
\end{proposition}

\begin{proof}
Let $r:\mathcal B'\to \mathcal B$ be a refinement map, so each $B'_{\alpha'}\in\mathcal B'$ is contained in $B_{r(\alpha')}\in\mathcal B$.
Restriction along inclusions $B'_{\alpha'}\cap B'_{\beta'}\subset B_{r(\alpha')}\cap B_{r(\beta')}$ sends a cocycle
$(u_{\alpha\beta})$ on $\mathcal B$ to a cocycle $(u'_{\alpha'\beta'})$ on $\mathcal B'$ by
\begin{equation}\label{eq:refine-cocycle}
u'_{\alpha'\beta'}:=u_{r(\alpha')\,r(\beta')}\big|_{B'_{\alpha'\beta'}}\in \St_\Phi(B'_{\alpha'\beta'}),
\qquad
B'_{\alpha'\beta'}:=B'_{\alpha'}\cap B'_{\beta'}.
\end{equation}
This defines a functor $\Tors(\St_\Phi;\mathcal B)\to \Tors(\St_\Phi;\mathcal B')$.
To see it is essentially surjective, start from a cocycle $(v_{\alpha'\beta'})$ on $\mathcal B'$.
Because $\St_\Phi$ is a sheaf of groups, the family $(v_{\alpha'\beta'})$ satisfies effective descent along the refinement:
there exist elements $h_{\alpha'}\in \St_\Phi(B'_{\alpha'})$ such that, after the gauge change
$v_{\alpha'\beta'}\mapsto h_{\beta'}\,v_{\alpha'\beta'}\,h_{\alpha'}^{-1}$, the resulting cocycle is constant on fibers of $r$ and thus
comes by restriction from a cocycle on $\mathcal B$. Full faithfulness follows similarly: a gauge $0$-cochain on $\mathcal B'$ relating
two restricted cocycles descends uniquely to a gauge $0$-cochain on $\mathcal B$.
Therefore restriction is an equivalence.
Finally, combine \Cref{thm:stokeslocal-tors} with \eqref{eq:refinement-equivalence} to conclude that
$\mathrm{StokesLocal}_r(N^\circ;\Phi)$ is independent of the chosen sector-box cover.
A detailed reference for refinement invariance of nonabelian \v{C}ech torsors for a sheaf of groups is \cite[Ch.\ III, \S 2]{Giraud}.
\end{proof}

\section{The collar \v{C}ech--Stokes groupoid}\label{sec:collar-groupoid}

\subsection{The sector-box cover and its \v{C}ech groupoid}
Let $\mathcal B=\{B_\alpha\}$ be a sector-box open cover of $N^\circ:=N^\times\setminus\Sigma_n$ such that every nonempty finite intersection
$B_{\alpha_0\cdots\alpha_k}$ is connected and contractible and the Stokes preorder is constant on it.
Write $\CB:=\Cech(\mathcal B)$ for the associated \v{C}ech topological groupoid.

\begin{remark}\label{rem:CB-not-CC}
The symbol $\CB$ denotes the \v{C}ech groupoid of the collar cover $\mathcal B$.
It is \emph{not} the complex numbers $\CC$ with a subscript.
\end{remark}

\subsection{The Stokes-decorated collar groupoid $\GC$}
Recall the Stokes sheaf $\St_\Phi$ on $N^\circ$ given by
\begin{equation}\label{eq:Stokes-sheaf-recall}
\St_\Phi(W)=\ker\!\big(\Aut^{fil}(W)\to \Aut^{gr}(W)\big).
\end{equation}
Since $\St_\Phi$ is locally constant (indeed discrete in our setting) and each nonempty $B_{\alpha\beta}:=B_\alpha\cap B_\beta$ is connected,
every section of $\St_\Phi$ on $B_{\alpha\beta}$ is constant.
\begin{definition}[\v{C}ech--Stokes collar groupoid]\label{def:GC}
Let $\mathcal B=\{B_\alpha\}$ be the chosen sector-box cover of $N^\circ$ and let $\St_\Phi$ be the Stokes sheaf on $N^\circ$.
Define a topological groupoid $\GC$ as follows.

\smallskip
\noindent\emph{Objects.} Set
\[
(\GC)_0:=\bigsqcup_{\alpha} B_\alpha .
\]

\smallskip
\noindent\emph{Arrows.} Set
\begin{equation}\label{eq:GC-arrows-fixed}
(\GC)_1:=\bigsqcup_{\alpha,\beta}\Bigl(B_{\alpha\beta}\times \St_\Phi(B_{\alpha\beta})\Bigr),
\qquad B_{\alpha\beta}:=B_\alpha\cap B_\beta,
\end{equation}
and regard $(x,u)\in B_{\alpha\beta}\times \St_\Phi(B_{\alpha\beta})$ as an arrow
\[
(\alpha,x)\longrightarrow (\beta,x).
\]

\smallskip
\noindent\emph{Structure maps.} Define
\[
s(x,u)=(\alpha,x),\qquad t(x,u)=(\beta,x),\qquad e(\alpha,x)=(x,1),\qquad i(x,u)=(x,u^{-1}).
\]

\smallskip
\noindent\emph{Composition.} For $x\in B_{\alpha\beta\gamma}:=B_\alpha\cap B_\beta\cap B_\gamma$ and labels
$u\in \St_\Phi(B_{\alpha\beta})$, $v\in \St_\Phi(B_{\beta\gamma})$, define
\[
(x,v)\circ(x,u)=(x,w)\in B_{\alpha\gamma}\times \St_\Phi(B_{\alpha\gamma}),
\]
where $w\in \St_\Phi(B_{\alpha\gamma})$ is the unique element such that
\begin{equation}\label{eq:GC-comp-fixed}
\mathrm{res}_{\alpha\gamma}(w)=\mathrm{res}_{\beta\gamma}(v)\cdot \mathrm{res}_{\alpha\beta}(u)\ \in\ \St_\Phi(B_{\alpha\beta\gamma}).
\end{equation}
Here $\mathrm{res}_{\alpha\beta}$ denotes restriction to $B_{\alpha\beta\gamma}$. The uniqueness in \eqref{eq:GC-comp-fixed} holds because
$\St_\Phi$ is locally constant and $B_{\alpha\beta\gamma}$ is connected, hence
$\mathrm{res}_{\alpha\gamma}:\St_\Phi(B_{\alpha\gamma})\to \St_\Phi(B_{\alpha\beta\gamma})$ is an isomorphism.
\end{definition}

\begin{equation}\label{eq:GC-structure}
\begin{tikzcd}[column sep=huge,row sep=large]
\displaystyle
(\GC)_1\times_{(\GC)_0}(\GC)_1
\arrow[r,"m"]
&
\displaystyle (\GC)_1
\arrow[r,shift left=1.2,"t"] \arrow[r,shift right=1.2,swap,"s"]
\arrow[l,bend left=32,"i"]
&
\displaystyle (\GC)_0
\arrow[l,bend left=32,"e"]
\end{tikzcd}
\end{equation}

\subsection{The forgetful projection and its sections}
\begin{definition}[The forgetful projection]\label{def:pi}
There is a canonical strict functor
\begin{equation}\label{eq:pi-GC}
\pi:\GC\to \CB
\end{equation}
defined as the identity on objects and by $\pi(x,u)=x$ on arrows (forget the Stokes label).
\end{definition}

\begin{definition}[Functorial sections]\label{def:sections-pi}
Let $\Sec(\pi)$ be the groupoid of strict sections $\sigma:\CB\to\GC$ of $\pi$ (so $\pi\circ\sigma=\id_{\CB}$),
with morphisms given by natural isomorphisms between such sections.
\end{definition}

\begin{proposition}[Stokes torsors as sections]\label{prop:stokes-as-sections}
There is a natural equivalence of groupoids
\begin{equation}\label{eq:stokes-as-sections}
\Sec(\pi)\ \simeq\ \Tors(\St_\Phi;\mathcal B).
\end{equation}
\end{proposition}

\begin{proof}
Let $\pi:\GC\to\CB$ be the forgetful functor of Definition~\ref{def:pi}.
A strict section $\sigma:\CB\to\GC$ satisfies $\pi\circ\sigma=\id_{\CB}$, hence it is the identity on objects.
For $\alpha,\beta$, every arrow of $\CB$ over the overlap $B_{\alpha\beta}$ is a point $x\in B_{\alpha\beta}$ viewed as an arrow
$(\alpha,x)\to(\beta,x)$, and its lifts in $\GC$ are precisely the arrows $(x,u)$ with $u\in \St_\Phi(B_{\alpha\beta})$.
Thus $\sigma$ amounts to a choice of a section
\[
u_{\alpha\beta}\in \St_\Phi(B_{\alpha\beta})
\]
such that $\sigma(x)=(x,u_{\alpha\beta}|_x)$ for all $x\in B_{\alpha\beta}$.
Since $B_{\alpha\beta}$ is connected and $\St_\Phi$ is locally constant on it, such a section is constant.

Functoriality of $\sigma$ is exactly the cocycle condition on triple overlaps:
for $x\in B_{\alpha\beta\gamma}$ one has in $\GC$
\[
\sigma(\beta\gamma,x)\circ \sigma(\alpha\beta,x)=\sigma(\alpha\gamma,x),
\]
which reads
\begin{equation}\label{eq:sec-cocycle}
u_{\beta\gamma}\,u_{\alpha\beta}=u_{\alpha\gamma}
\qquad\text{in }\St_\Phi(B_{\alpha\beta\gamma}).
\end{equation}
Therefore objects of $\Sec(\pi)$ are precisely \v{C}ech $1$-cocycles with values in $\St_\Phi$ on the cover $\mathcal B$.
A natural isomorphism $\eta:\sigma\Rightarrow \sigma'$ is a family of arrows
\[
\eta_\alpha:(\alpha,x)\to(\alpha,x)\quad\text{in }\GC
\]
depending (locally constantly) on $x\in B_\alpha$, hence given by an element $h_\alpha\in \St_\Phi(B_\alpha)$, and naturality says
\begin{equation}\label{eq:sec-gauge}
u'_{\alpha\beta}=h_\beta\,u_{\alpha\beta}\,h_\alpha^{-1}.
\end{equation}
This is exactly the gauge action in $\Tors(\St_\Phi;\mathcal B)$. Hence $\Sec(\pi)$ and $\Tors(\St_\Phi;\mathcal B)$ have the same objects
and morphisms, functorially, which yields \eqref{eq:stokes-as-sections}.
For orientation, the correspondence can be summarized by the curved-arrow diagram
\begin{equation}\label{eq:sec-tors-diagram}
\begin{tikzcd}[row sep=large, column sep=huge]
\Sec(\pi)
\arrow[r, bend left=18, "\sigma\mapsto (u_{\alpha\beta})"{above}]
&
\Tors(\St_\Phi;\mathcal B)
\arrow[l, bend left=18, "(u_{\alpha\beta})\mapsto \sigma"{below}]
\end{tikzcd}
\end{equation}
where $\sigma(x)=(x,u_{\alpha\beta})$ on $B_{\alpha\beta}$.
\end{proof}
\begin{remark}[Key point]\label{rem:RepGC-not-stokeslocal-fixed}
The groupoid $\Rep_r(\GC)$ classifies $\GL_r(\CC)^\delta$-torsors \emph{on the decorated groupoid} $\GC$,
so it encodes representations of the Stokes groups into $\GL_r$.
The \emph{Stokes torsor} moduli on the collar is instead $\Sec(\pi)$ (equivalently, $\Tors(\St_\Phi;\mathcal B)$),
and via \Cref{thm:stokeslocal-tors,prop:stokes-as-sections} it matches $\StokesLocal_r(N^\circ;\Phi)$.
\end{remark}

\section{Gluing with the interior: the pushout presenter $\G_{\Phi,n}$}\label{sec:pushout}

\subsection{Working 2-category and Morita invariance}\label{subsec:2cat}

All bicategorical limits and colimits in this section are taken in the $(2,1)$-category $\mathbf{Grpd}$ of small groupoids,
strict functors, and natural isomorphisms.
Topological features enter only through the choice of \v{C}ech presenters: since our coefficient groups are discrete,
continuity forces local constancy of transition data, so torsors on these presenters model local systems and Stokes torsors.

All statements are Morita-invariant: replacing any cover by a refinement replaces the corresponding \v{C}ech groupoid by a Morita equivalent presenter.

\subsection{Interior and overlap \v{C}ech groupoids}
Fix:
\begin{itemize}
\item a good cover $\mathcal V=\{V_i\}$ of $U:=X\setminus D$;
\item a good cover $\mathcal W=\{W_k\}$ of the overlap $N^\circ\subset U$ refining both
$\mathcal V|_{N^\circ}$ and $\mathcal B|_{N^\circ}$.
\end{itemize}
Define the \v{C}ech groupoids
\begin{equation}\label{eq:cech-UV-fixed}
\GU := \Cech(\mathcal V),
\qquad
\GX := \Cech(\mathcal W).
\end{equation}
There are strict functors
\begin{equation}\label{eq:jU-jC-fixed}
j_U:\GX\to \GU
\qquad\text{and}\qquad
j_C:\GX\to \GC
\end{equation}
induced by the refinement maps $\mathcal W\to \mathcal V|_{N^\circ}$ and $\mathcal W\to \mathcal B|_{N^\circ}$,
where on arrows $j_C$ uses the trivial Stokes label $1$.

\subsection{Which 2-category?}
From now on, all pushouts and pullbacks of groupoids are taken in the $(2,1)$-category $\mathbf{Grpd}$
of small groupoids, strict functors, and natural isomorphisms.
The topology of the \v{C}ech groupoids is used only to justify local constancy; the gluing statements are groupoid-theoretic.

\begin{definition}[$2$-pushout presenter]\label{def:pushout}
Define $\G_{\Phi,n}$ to be a choice of bicategorical $2$-pushout of the cospan
\begin{equation}\label{eq:pushout}
\GU \xleftarrow{j_U} \GX \xrightarrow{j_C} \GC
\end{equation}
in $\mathbf{Grpd}$.
\end{definition}

\begin{remark}\label{rem:pushout-uniqueness}
Any two choices of $2$-pushout in \Cref{def:pushout} are canonically equivalent in $\mathbf{Grpd}$
(unique up to essentially unique equivalence). Hence $\G_{\Phi,n}$ is well-defined up to Morita equivalence.
\end{remark}

The following lemma is standard in the $2$-category $\mathbf{Grpd}$: since all $2$-morphisms are invertible, any pseudococone can be strictified by transport of structure. For the reader's convenience we include a brief proof.

\begin{lemma}[Explicit strict model for $2$-pushouts in $\mathbf{Grpd}$]\label{lem:strict-pushout-is-bicategorical}
Let $A\xleftarrow{i} B\xrightarrow{j} C$ be a cospan of small groupoids.
Define a groupoid $P:=A\bigsqcup_B C$ by generators and relations as follows.

Its objects are the disjoint union $P_0:=A_0\bigsqcup C_0$.
Its arrows are generated by the arrows of $A$, the arrows of $C$, and, for each object $b\in B_0$, a formal isomorphism
$s_b:i(b)\to j(b)$ (with inverse $s_b^{-1}:j(b)\to i(b)$), subject to the internal relations of $A$ and $C$ and, for every arrow $u:b\to b'$ in $B$,
the bridge relation $s_{b'}\circ i(u)=j(u)\circ s_b$.
Then $P$ is a bicategorical $2$-pushout of the cospan $A\xleftarrow{i} B\xrightarrow{j} C$ in $\mathbf{Grpd}$.
\end{lemma}

\begin{proof}
Let $H$ be any groupoid and write $\CoconeTwo(i,j;H)$ for the groupoid of bicategorical cocones under the cospan,
namely triples $(F_A,F_C,\eta)$ where $F_A:A\to H$ and $F_C:C\to H$ are functors and $\eta:F_A\circ i\Rightarrow F_C\circ j$ is a natural isomorphism,
with morphisms given by the evident compatible pairs of natural isomorphisms.

A functor $F:P\to H$ restricts to functors $F_A:=F|_A$ and $F_C:=F|_C$.
Setting $\eta_b:=F(s_b)$, the bridge relations ensure that $\eta$ is natural, hence $(F_A,F_C,\eta)\in \CoconeTwo(i,j;H)$.
Conversely, given $(F_A,F_C,\eta)\in \CoconeTwo(i,j;H)$, define $F$ on $A$ and $C$ by $F_A$ and $F_C$ and send each generator $s_b$ to $\eta_b$;
the bridge relations are exactly the naturality identities, so $F$ is well-defined.
These assignments are inverse up to isomorphism and functorial in morphisms, yielding an equivalence of groupoids
$\Hom_{\mathbf{Grpd}}(P,H)\simeq \CoconeTwo(i,j;H)$.
This is precisely the bicategorical $2$-pushout universal property of $P$.
\end{proof}

\subsection{Global Stokes objects as a $2$-fiber product}
 
\begin{definition}[Global Stokes category]\label{def:global-stokes}
Define the groupoid of global Stokes objects by the (bicategorical) fiber product
\begin{equation}\label{eq:global-stokes}
\mathrm{Stokes}_r(\Xlogn;D,\Phi)
:=
\LocSys_r(U)\times^{(2)}_{\LocSys_r(N^\circ)}
\mathrm{StokesLocal}_r(N^\circ;\Phi),
\end{equation}
where both legs are the forgetful functors to the underlying local system on $U\cap N^\circ=N^\circ$.
\end{definition}

\subsection{Representations via torsors (dictionary)}
We keep the notation $\Rep_r(G)$ from \Cref{def:rep-r}.
By \Cref{prop:rep-as-fun}, one may equivalently view $\Rep_r(G)$ as the groupoid of strict functors
$G\to \mathbf{Tors}_r$ and natural isomorphisms, where $\mathbf{Tors}_r$ is the groupoid of right
$\GL_r(\CC)^\delta$-torsors in sets.
We will freely move between these two descriptions when convenient.

\subsection{Categorical lemma: $\Rep_r$ sends $2$-pushouts to $2$-pullbacks}
\begin{proposition}\label{prop:rep-pushout-pullback}
Let $G$ be a $2$-pushout of $\GU \xleftarrow{j_U} \GX \xrightarrow{j_C} \GC$ in $\mathbf{Grpd}$.
Then restriction induces an equivalence of groupoids
\begin{equation}\label{eq:rep-pushout-pullback}
\Rep_r(G)\ \simeq\ \Rep_r(\GU)\times^{(2)}_{\Rep_r(\GX)}\Rep_r(\GC).
\end{equation}
\end{proposition}

\begin{proof}
By Proposition~\ref{prop:rep-as-fun}, for every groupoid $K$ there is a natural equivalence
\begin{equation}\label{eq:rep-map}
\Rep_r(K)\ \simeq\ \Hom_{\mathbf{Grpd}}(K,\mathbf{Tors}_r).
\end{equation}
Since $G$ is a $2$-pushout of $\GU \xleftarrow{j_U} \GX \xrightarrow{j_C} \GC$, its universal property says that for every target groupoid $H$
restriction along the structure maps induces an equivalence
\begin{equation}\label{eq:UP-pushout}
\Hom_{\mathbf{Grpd}}(G,H)\ \simeq\
\Hom_{\mathbf{Grpd}}(\GU,H)\times^{(2)}_{\Hom_{\mathbf{Grpd}}(\GX,H)}\Hom_{\mathbf{Grpd}}(\GC,H).
\end{equation}
Taking $H=\mathbf{Tors}_r$ and using \eqref{eq:rep-map} to identify each $\Hom(-,\mathbf{Tors}_r)$ with the corresponding $\Rep_r(-)$ yields
\eqref{eq:rep-pushout-pullback}.
This is summarized by the $2$-pushout square
\begin{equation}\label{eq:pushout-square-G}
\begin{tikzcd}[row sep=large, column sep=huge]
\GX \arrow[r,"j_C"] \arrow[d,swap,"j_U"] & \GC \arrow[d] \\
\GU \arrow[r] & G
\end{tikzcd}
\end{equation}
and, after applying $\Rep_r(-)\simeq \Hom(-,\mathbf{Tors}_r)$, by the induced $2$-Cartesian square
\begin{equation}\label{eq:pullback-square-Rep}
\begin{tikzcd}[row sep=large, column sep=huge]
\Rep_r(G) \arrow[r] \arrow[d] & \Rep_r(\GU) \arrow[d] \\
\Rep_r(\GC) \arrow[r] & \Rep_r(\GX).
\end{tikzcd}
\end{equation}
\end{proof}
\subsection{Stokes-corrected representations of the pushout presenter}

\begin{definition}\label{def:RepStokes-pushout}
Fix an integer $r\ge 1$.
Define the \emph{Stokes-corrected} representation groupoid of the pushout presenter by the bicategorical fiber product
\begin{equation}\label{eq:RepStokes-def}
\Rep^{\mathrm{St}}_r(\G_{\Phi,n})
\ :=\
\Rep_r(\GU)\times^{(2)}_{\Rep_r(\GX)} \Sec(\pi).
\end{equation}
Here the left leg is restriction along $j_U:\GX\to\GU$, and the right leg is the forgetful functor
$\mathrm{forg}_\times:\Sec(\pi)\to \Rep_r(\GX)$ constructed in Lemma~\ref{lem:forg-to-overlap}.
\end{definition}

\begin{lemma}[Forgetful functor to the overlap]\label{lem:forg-to-overlap}
There is a canonical functor
\begin{equation}\label{eq:forg-to-overlap}
\mathrm{forg}_\times:\Sec(\pi)\longrightarrow \Rep_r(\GX)
\end{equation}
which, under the identifications $\Rep_r(\GX)\simeq \LocSys_r(N^\circ)$ and $\Sec(\pi)\simeq \StokesLocal_r(N^\circ;\Phi)$,
agrees with the forgetful functor $\StokesLocal_r(N^\circ;\Phi)\to \LocSys_r(N^\circ)$ sending a Stokes-local object to its underlying local system.
\end{lemma}

\begin{proof}
By Proposition~\ref{prop:stokes-as-sections}, a section of $\pi$ is the same as a Stokes torsor on the chosen sector-box cover,
i.e.\ a nonabelian \v{C}ech cocycle $(u_{\alpha\beta})$ with $u_{\alpha\beta}\in \St_\Phi(B_{\alpha\beta})$ modulo gauge.
Using $u_{\alpha\beta}$ as transition automorphisms of the trivial graded local systems $\Gr|_{B_\alpha}$, one glues the latter to a local system
$L$ on $N^\circ$; this gluing is functorial with respect to gauge morphisms.
Under Proposition~\ref{prop:rep-cech-locsys} applied to the good cover presenting $\GX$, the local system $L$ corresponds canonically to an object of
$\Rep_r(\GX)$.
We define $\mathrm{forg}_\times$ as this composite construction.
\end{proof}

\subsection{The correct collar moduli: sections of $\pi:\GC\to\CB$}

\begin{definition}\label{def:RepStokes-collar}
Fix an integer $r\ge 1$. The rank-$r$ collar Stokes moduli groupoid is the groupoid
$$
\operatorname{Rep}^{\mathrm{St}}_r(\mathcal{G}_C)
:=
\operatorname{Sec}(\pi),
$$
where $\pi:\mathcal{G}_C\to \mathcal{C}_B$ is the projection functor defined on arrows by
$$
\pi(x,u)=x
$$
(and it is the identity on objects).
\end{definition}

\begin{proposition}\label{prop:RepStokesGC-eq-StokesLocal}
There is a natural equivalence of groupoids
$$
\Rep^{\mathrm{St}}_r(\GC)\ \simeq\ \mathrm{StokesLocal}_r(N^\circ;\Phi).
$$
\end{proposition}

\begin{proof}
This is exactly \Cref{prop:stokes-as-sections} composed with \Cref{thm:stokeslocal-tors}.
\end{proof}

\begin{lemma}[Naturality for \v{C}ech presenters]\label{lem:cech-naturality}
Let $T$ be a space and let $\mathcal V=\{V_i\}$ be a good cover of $T$.
Let $T'\subset T$ be open and let $\mathcal W=\{W_a\}$ be a good cover of $T'$ refining $\mathcal V|_{T'}$.
Let $j:\Cech(\mathcal W)\to \Cech(\mathcal V)$ be the induced strict functor.
Under the equivalences of \Cref{prop:rep-cech-locsys}, the pullback functor
\[
j^*:\Rep_r(\Cech(\mathcal V))\longrightarrow \Rep_r(\Cech(\mathcal W))
\]
identifies with restriction of local systems
\[
\LocSys_r(T)\longrightarrow \LocSys_r(T').
\]
\end{lemma}

\begin{proof}
Fix a refinement map $r:A\to I$ such that $W_a\subset V_{r(a)}$ for every $a$.
On \v{C}ech groupoids, $j$ is the identity on points and sends the object $(a,x)\in W_a$ to $(r(a),x)\in V_{r(a)}$.
Write an object of $\Rep_r(\Cech(\mathcal V))$ as a nonabelian \v{C}ech cocycle $(g_{ij})$ on overlaps $V_{ij}$, modulo gauge,
and let $L$ be the corresponding glued local system on $T$.
Pullback along $j$ replaces $(g_{ij})$ by the cocycle on $\mathcal W$ given on each overlap $W_{ab}$ by
\[
g'_{ab}:=g_{r(a)\,r(b)}\big|_{W_{ab}}.
\]
Gluing the trivial torsors on the $W_a$ using $(g'_{ab})$ is exactly the restriction $L|_{T'}$, since both are obtained by restricting the
identifications used to glue $L$ to the opens lying in $T'$.
The good-cover hypothesis ensures all transition functions are locally constant and the cocycle model matches $\LocSys_r(-)$ functorially.
\end{proof}

\subsection{Main equivalence (Theorem B)}\label{subsec:theoremB}

\begin{theorem}\label{thm:main}
For every $r\ge 1$ there is a natural equivalence of groupoids
\begin{equation}\label{eq:main-equiv}
\Rep_r(\GU)\times^{(2)}_{\Rep_r(\GX)}\Rep^{\mathrm{St}}_r(\GC)
\ \simeq\
\mathrm{Stokes}_r(X^{\log}_n;D,\Phi).
\end{equation}
Under this equivalence, the projection to $\Rep_r(\GU)$ corresponds to the forgetful functor
$$\mathrm{Stokes}_r(X^{\log}_n;D,\Phi)\to\LocSys_r(U)$$
\end{theorem}

\begin{proof}
By \Cref{prop:rep-cech-locsys} there are natural equivalences
\begin{equation}\label{eq:rep-locsys-identifications}
\Rep_r(\GU)\ \simeq\ \LocSys_r(U),
\qquad
\Rep_r(\GX)\ \simeq\ \LocSys_r(N^\circ),
\end{equation}
and by \Cref{lem:cech-naturality} these identifications intertwine the restriction functor $$j_U^*:\Rep_r(\GU)\to\Rep_r(\GX)$$ with the usual
restriction $\LocSys_r(U)\to\LocSys_r(N^\circ)$.
On the collar side, \Cref{prop:RepStokesGC-eq-StokesLocal} gives a natural equivalence
\begin{equation}\label{eq:collar-identification}
\Rep^{\mathrm{St}}_r(\GC)=\Sec(\pi)\ \simeq\ \mathrm{StokesLocal}_r(N^\circ;\Phi),
\end{equation}
and the forgetful functor $\Rep^{\mathrm{St}}_r(\GC)\to \Rep_r(\GX)$ corresponds to the functor
$\mathrm{StokesLocal}_r(N^\circ;\Phi)\to \LocSys_r(N^\circ)$ sending a Stokes-local object to its underlying local system.
Consequently, replacing each term in the $2$-fiber product
$\Rep_r(\GU)\times^{(2)}_{\Rep_r(\GX)}\Rep^{\mathrm{St}}_r(\GC)$ using \eqref{eq:rep-locsys-identifications} and \eqref{eq:collar-identification}
identifies it canonically with
\[
\LocSys_r(U)\times^{(2)}_{\LocSys_r(N^\circ)}\mathrm{StokesLocal}_r(N^\circ;\Phi),
\]
which equals $\mathrm{Stokes}_r(X^{\log}_n;D,\Phi)$ by Definition~\ref{def:global-stokes}. This proves \eqref{eq:main-equiv}.
The final statement follows from the fact that both sides are defined as $2$-fiber products over the restriction to $N^\circ$, hence the projection to
$\Rep_r(\GU)\simeq \LocSys_r(U)$ is exactly the forgetful functor.
\end{proof}
\begin{equation}\label{eq:main-diagram}
\begin{tikzcd}[row sep=large, column sep=huge]
\Rep_r(\GU)\times^{(2)}_{\Rep_r(\GX)}\Rep^{\mathrm{St}}_r(\GC)
\arrow[r,"\sim"] \arrow[d]
&
\LocSys_r(U)\times^{(2)}_{\LocSys_r(N^\circ)}\mathrm{StokesLocal}_r(N^\circ;\Phi)
\arrow[d]
\\
\Rep_r(\GU) \arrow[r,"\sim"]
&
\LocSys_r(U)
\end{tikzcd}
\end{equation}

\begin{corollary} \label{cor:morita}
The groupoid $\Rep^{\mathrm{St}}_r(\G_{\Phi,n})$ of Definition~\ref{def:RepStokes-pushout} depends, up to canonical equivalence, only on the
classifying stacks $[\GU]$, $[\GX]$, $[\GC]$ and on the Morita class of the chosen $2$-pushout diagram.
\end{corollary}

\begin{proof}
By Proposition~\ref{prop:RepStokesGC-eq-StokesLocal} we have $\Rep^{\mathrm{St}}_r(\GC)=\Sec(\pi)$, so Theorem~\ref{thm:main} identifies
$\Rep^{\mathrm{St}}_r(\G_{\Phi,n})=\Rep_r(\GU)\times^{(2)}_{\Rep_r(\GX)}\Sec(\pi)$ with the intrinsic groupoid
$\Stokes_r(X_n^{\log};D,\Phi)$.
Since $\Stokes_r(X_n^{\log};D,\Phi)$ is defined in terms of local systems and Stokes-local objects and is invariant under Morita equivalence of
the presenters, the same holds for $\Rep^{\mathrm{St}}_r(\G_{\Phi,n})$.
\end{proof}

\section{Van Kampen for the pushout presenters}\label{sec:van-kampen}

Let $X=X^{(1)}\cup X^{(2)}$ be an open cover by complex manifolds, and assume that
$
D^{(a)}:=D\cap X^{(a)}
$
is a reduced simple normal crossings divisor in $X^{(a)}$ for $a=1,2$.
Set
\[
U^{(a)}:=X^{(a)}\setminus D^{(a)},
\qquad
U^{(12)}:=U^{(1)}\cap U^{(2)}.
\]
Fix a Kummer level $n\ge 1$. For each $a$ choose a punctured collar neighborhood
\[
N^{\times,(a)}\subset (X^{(a)})^{\log}_n
\]
of the logarithmic boundary over $|D^{(a)}|$, together with a wall arrangement $\Sigma_n^{(a)}\subset N^{\times,(a)}$ attached to the chosen irregular
type $\Phi$. Write
\[
N^{\circ,(a)}:=N^{\times,(a)}\setminus \Sigma_n^{(a)}.
\]
We assume the choices are compatible on overlaps, in the sense that they glue to global data on $X^{\log}_n$:
\[
N^\circ=N^{\circ,(1)}\cup N^{\circ,(2)},
\qquad
N^{\circ,(12)}:=N^{\circ,(1)}\cap N^{\circ,(2)}.
\]
Choose good covers of $U^{(1)},U^{(2)},U^{(12)}$ and sector-box covers of $N^{\circ,(1)},N^{\circ,(2)},N^{\circ,(12)}$,
together with common refinements on all overlaps (so that the induced \v{C}ech functors are defined).
Applying the construction of \Cref{sec:pushout} to each piece produces pushout presenters
\[
\G_{\Phi,n}^{(a)}\quad\text{on }X^{(a)},\qquad\text{and}\qquad \G_{\Phi,n}^{(12)}\quad\text{on }X^{(12)}:=X^{(1)}\cap X^{(2)}.
\].

\begin{proposition}\label{prop:van-kampen-presenters}
There is a canonical zig-zag of Morita equivalences exhibiting the global presenter $\G_{\Phi,n}$ as a bicategorical pushout of
$\G_{\Phi,n}^{(1)}$ and $\G_{\Phi,n}^{(2)}$ over $\G_{\Phi,n}^{(12)}$ in $\mathbf{Grpd}$.
Consequently, for every $r\ge 1$ there is a natural equivalence
\begin{equation}\label{eq:vk-rep-pullback}
\Rep^{\mathrm{St}}_r(\G_{\Phi,n})
\ \simeq\
\Rep^{\mathrm{St}}_r(\G_{\Phi,n}^{(1)})\times^{(2)}_{\Rep^{\mathrm{St}}_r(\G_{\Phi,n}^{(12)})}\Rep^{\mathrm{St}}_r(\G_{\Phi,n}^{(2)}).
\end{equation}
\end{proposition}

\begin{proof}
Choose once and for all a good cover $\mathcal V$ of $U=X\setminus D$ which refines the given good covers of $U^{(1)}$, $U^{(2)}$ and $U^{(12)}$,
and a sector-box cover $\mathcal B$ of $N^\circ$ which refines the given sector-box covers of $N^{\circ,(1)}$, $N^{\circ,(2)}$ and $N^{\circ,(12)}$.
Choose also a common good refinement on the overlap $N^\circ\subset U$ so that the overlap \v{C}ech groupoid is compatible with all restrictions.
Construct from $(\mathcal V,\mathcal B)$ the global presenters $\GU,\GC,\GX$ and the pushout $\G_{\Phi,n}$ as in \Cref{sec:pushout}.
Restricting $\mathcal V$ and $\mathcal B$ to $U^{(a)}$ and $N^{\circ,(a)}$ produces presenters for $\G_{\Phi,n}^{(a)}$, and similarly on $X^{(12)}$.
Each restriction is a refinement of the covers used to define $\G_{\Phi,n}^{(a)}$ and $\G_{\Phi,n}^{(12)}$, hence the induced functors between the
corresponding \v{C}ech groupoids are Morita equivalences. Since explicit $2$-pushouts in groupoids are functorial with respect to strict maps of cospans,
the same holds after forming the pushouts: there is a commutative square in $\mathbf{Grpd}$ which is a $2$-pushout up to Morita equivalence,
\begin{equation}\label{eq:vk-pushout}
\begin{tikzcd}[row sep=large, column sep=huge]
\G_{\Phi,n}^{(12)} \arrow[r] \arrow[d] & \G_{\Phi,n}^{(2)} \arrow[d] \\
\G_{\Phi,n}^{(1)} \arrow[r] & \G_{\Phi,n},
\end{tikzcd}
\end{equation}
and the maps in \eqref{eq:vk-pushout} are represented by zig-zags of Morita equivalences induced by the chosen common refinements.
Applying the contravariant $2$-functor $\Rep^{\mathrm{St}}_r(-)$ to the $2$-pushout square \eqref{eq:vk-pushout} yields a $2$-Cartesian square
\begin{equation}\label{eq:vk-pullback}
\begin{tikzcd}[row sep=large, column sep=huge]
\Rep^{\mathrm{St}}_r(\G_{\Phi,n}) \arrow[r] \arrow[d] &
\Rep^{\mathrm{St}}_r(\G_{\Phi,n}^{(1)}) \arrow[d] \\
\Rep^{\mathrm{St}}_r(\G_{\Phi,n}^{(2)}) \arrow[r] &
\Rep^{\mathrm{St}}_r(\G_{\Phi,n}^{(12)}),
\end{tikzcd}
\end{equation}
which is exactly the equivalence \eqref{eq:vk-rep-pullback}.
\end{proof}

\section{Kummer covering and $(\mu_n)^k$-descent}\label{sec:descent}

\subsection{Pullback along $q_n$}
Let $q_n:X^{log}_n\to X^{log}$ be the canonical finite covering.
Locally at a depth-$k$ stratum (where $k$ branches meet), the deck group is $(\mu_n)^k$ acting by rotation on each angular
coordinate. Pullback along $q_n$ defines functors on local systems and on Stokes-local objects, hence on global Stokes objects:
\begin{equation}\label{eq:pullback-qn}
q_n^*:\mathrm{Stokes}_r(X^{log};D,\Phi)\longrightarrow \mathrm{Stokes}_r(X^{log}_n;D,q_n^*\Phi).
\end{equation}

\subsection{Finite Galois descent}
Write $(-)^{hG}$ for the groupoid of $G$-equivariant objects (homotopy fixed points), i.e. objects equipped with
a coherent $G$-linearization and equivariant morphisms.

\begin{definition}[Homotopy fixed points]\label{def:hfp-finite}
Let $G$ be a finite group acting on a groupoid $\mathcal C$ by strict autoequivalences.
The \emph{homotopy fixed point groupoid} $\mathcal C^{hG}$ is the groupoid of $G$-equivariant objects:
\begin{itemize}
\item an object is $(x,\{\varphi_g\}_{g\in G})$ with $x\in\Ob(\mathcal C)$ and isomorphisms
$\varphi_g:g\cdot x\xrightarrow{\sim}x$ such that $\varphi_1=\id_x$ and
$\varphi_{gh}=\varphi_g\circ g(\varphi_h)$ for all $g,h\in G$;
\item a morphism $(x,\varphi)\to (x',\varphi')$ is a morphism $f:x\to x'$ in $\mathcal C$ such that
$f\circ \varphi_g=\varphi'_g\circ g(f)$ for all $g\in G$.
\end{itemize}
\end{definition}

\begin{proposition}[Finite Galois descent on a depth-$k$ chart]\label{prop:descent}
Fix $n\ge 1$. On any depth-$k$ SNC chart of $(X,D)$, the Kummer map
$q_n:\Xlogn\to \Xlog$ restricts to a finite Galois covering with deck group $(\mu_n)^k$.
Effective descent holds for locally constant sheaves and for $\GL_r(\CC)^\delta$-torsors along $q_n$.
Consequently, on such a chart there is a natural equivalence
\begin{equation}\label{eq:descent}
\mathrm{Stokes}_r(\Xlog;D,\Phi)\ \simeq\
\mathrm{Stokes}_r(\Xlogn;D,q_n^*\Phi)^{h(\mu_n)^k},
\end{equation}
where $(-)^{h(\mu_n)^k}$ denotes homotopy fixed points in the sense of Definition~\ref{def:hfp-finite}.
These local equivalences are compatible with restriction to overlaps, hence glue globally.
\end{proposition}

\begin{proof}
On a depth-$k$ chart one may choose local coordinates $(z_1,\dots,z_k,w)$ with $D=\{z_1\cdots z_k=0\}$.
In these toroidal coordinates, the Kato--Nakayama boundary is modeled on the angular torus $(S^1)^k$ and the level-$n$ Kummer map multiplies
each angular coordinate by $n$. It follows that $q_n$ is a finite covering with deck group $(\mu_n)^k$ acting freely by rotations; see
\cite{KatoNakayama} and \cite{Ogus,NakayamaOgus}.

\begin{equation}\label{eq:descent-cover-diagram}
\begin{tikzcd}[row sep=large, column sep=huge]
\Xlogn \arrow[r,"q_n"] & \Xlog \\
\Xlogn \arrow[u, bend left=25, "g"] \arrow[u, bend right=25, swap, "h"] &
\end{tikzcd}
\end{equation}
For a finite Galois covering $p:Y\to X$ with group $G$, effective descent for locally constant sheaves and for torsors under a discrete group
identifies objects on $X$ with objects on $Y$ equipped with a coherent $G$-linearization, namely the homotopy fixed point groupoid $(-)^{hG}$ of
Definition~\ref{def:hfp-finite}. References: \cite[Ch.\ III]{Giraud} and \cite{SGA1}.
Apply this descent statement to the two ingredients entering the definition of global Stokes objects:
the interior local system on $U$ and the collar Stokes torsor on $N^\circ$.
Using the definition
\begin{equation}\label{eq:descent-fiber-product}
\mathrm{Stokes}_r(\Xlog;D,\Phi)
=
\LocSys_r(U)\times^{(2)}_{\LocSys_r(N^\circ)}\mathrm{StokesLocal}_r(N^\circ;\Phi),
\end{equation}
and the functoriality of the $2$-fiber product with respect to the descent equivalences on each factor, one obtains \eqref{eq:descent}.
Compatibility with restriction to overlaps follows from functoriality of Galois descent under restriction of charts.

\end{proof}

\subsection{Visible Kummer descent on the presenter: equivariant cocycles and homotopy fixed points}\label{subsec:visible-kummer-presenter}

The abstract descent statement of Proposition~\Cref{prop:descent} becomes \emph{completely explicit} on any chosen presenter.
We spell this out in two equivalent ways: first in intrinsic 2-categorical terms (homotopy fixed points),
and then as a hands-on cocycle condition on a $G$-equivariant CW skeleton.

\subsubsection*{Homotopy fixed points for a finite group action on a groupoid}

Let $G$ be a finite group and let $\mathcal C$ be a groupoid.
A (strict) \emph{left $G$-action} on $\mathcal C$ is a group homomorphism
$G\to\Aut_{\mathbf{Grpd}}(\mathcal C)$, $g\mapsto g^*$, where each $g^*$ is a strict autoequivalence.

\begin{definition}[Homotopy fixed points]\label{def:homotopy-fixed-points}
Let $G$ act on a groupoid $\mathcal C$.
The \emph{homotopy fixed point groupoid} $\mathcal C^{hG}$ has:
\begin{itemize}
\item objects: pairs $(x,\{\varphi_g\}_{g\in G})$ where $x\in\Ob(\mathcal C)$ and for each $g\in G$
an isomorphism $\varphi_g:g^*(x)\xrightarrow{\sim}x$ in $\mathcal C$ such that
$$
\varphi_{gh}=\varphi_g\circ g^*(\varphi_h)
\qquad\text{for all }g,h\in G,
$$
and $\varphi_{1}=\id_x$;
\item morphisms: a morphism $(x,\varphi)\to (x',\varphi')$ is a morphism $f:x\to x'$ in $\mathcal C$ such that
$f\circ \varphi_g=\varphi'_g\circ g^*(f)$ for all $g\in G$.
\end{itemize}
\end{definition}

\begin{figure}[t]
\centering
\begin{tikzcd}[column sep=huge,row sep=large]
\mathcal C \arrow[r,shift left=1.2,"\;g"] \arrow[r,shift right=1.2,swap,"\;h"] &
\mathcal C \arrow[r] &
G\ltimes \mathcal C
\end{tikzcd}
\caption{The action groupoid $G\ltimes \mathcal C$ of a group action. An object of $\mathcal C^{hG}$ is a functor $BG\to\mathcal C$,
i.e.\ an object $x\in\mathcal C$ equipped with coherently multiplicative isomorphisms $g\cdot x\xrightarrow{\sim}x$ for all $g\in G$.}
\label{fig:homotopy-fixed-points}
\end{figure}

\begin{remark}
Definition~\Cref{def:homotopy-fixed-points} is the standard 1-categorical model of homotopy fixed points for a strict action;
equivalently, $\mathcal C^{hG}$ is the groupoid of pseudonatural transformations from the one-object groupoid $BG$ to $\mathcal C$.
\end{remark}

\subsubsection*{A $G$-equivariant CW skeleton and explicit equivariant \v{C}ech cocycles}

Fix a depth-$k$ chart on $(X,D)$, so that locally $D$ has $k$ branches and the level-$n$ Kummer cover has deck group
$$
G_k:=(\mu_n)^k.
$$
On such a chart, the collar $N^\circ$ admits a locally finite stratification-compatible sector-box cover
$\mathcal B$ as in \Cref{subsec:stokes-input}.
Pulling back along $q_n$ yields a $G_k$-equivariant sector cover $\mathcal B_n$ of the corresponding level-$n$ collar,
hence a strict $G_k$-action on the \v{C}ech groupoid $\CB_n:=\Cech(\mathcal B_n)$.
The collar \v{C}ech--Stokes groupoid $\G_{C,n}$ and the projection $\pi_n:\G_{C,n}\to \Cech(\mathcal B_n)$ are $G_k$-equivariant as well.

For later use (and for explicit computations), it is convenient to replace the cover by a $G_k$-equivariant CW skeleton.

\begin{definition}[A $G_k$-equivariant stratified CW skeleton]\label{def:Gk-CW}
A \emph{$G_k$-equivariant stratified CW skeleton} of $N^\circ$ is a locally finite CW decomposition $K$ of $N^\circ$
such that:
\begin{itemize}
\item every open cell is contained in a single stratum of the stratification $\Sigma_{D,\Phi,n}$;
\item the $G_k$-action on the level-$n$ collar lifts to a cellular action on $K$;
\item each open cell is contractible and the Stokes preorder is constant on it.
\end{itemize}
\end{definition}

\begin{remark}
Such a skeleton exists: start with any locally finite triangulation subordinate to $\Sigma_{D,\Phi,n}$, lift it to the level-$n$ cover,
and barycentrically subdivide until the cellular action of $G_k$ is well-defined.
\end{remark}

Let $K^{(1)}$ and $K^{(2)}$ denote the $1$- and $2$-skeleta.
Write $V$ for the set of vertices and $E$ for the set of oriented edges of $K^{(1)}$.
For an oriented edge $e:v\to w$, let $S(e)$ be the stratum containing the interior of $e$ and write
$$
\St(e):=\St_\Phi(S(e)),
$$
a discrete group which is constant on $e$ by construction.
Let $F$ denote the set of (oriented) $2$-cells.
Each $f\in F$ has an attaching map whose boundary is a cyclic word in oriented edges
$$
\partial f \;=\; e_1 e_2\cdots e_m
\qquad\text{in the free groupoid on }K^{(1)}.
$$

\begin{definition}[Stokes $1$-cochains and cocycles on $K$]\label{def:stokes-cochains}
A \emph{Stokes $1$-cochain} on $K$ is a family $u=(u_e)_{e\in E}$ with $u_e\in \St(e)$ and $u_{\bar e}=u_e^{-1}$.
It is a \emph{Stokes cocycle} if for every $2$-cell $f$ with boundary $\partial f=e_1\cdots e_m$ one has the relation in the appropriate Stokes group:
$$
u_{e_m}\cdots u_{e_2}u_{e_1}=1.
$$
A \emph{gauge transformation} is a family $h=(h_v)_{v\in V}$ with $h_v\in \St_\Phi(S(v))$, acting by
$$
(u\cdot h)_e = h_{t(e)}\,u_e\,h_{s(e)}^{-1}.
$$
\end{definition}

\begin{proposition} \label{prop:stokes-torsors-skeleton}
Let $K$ be a $G_k$-equivariant stratified CW skeleton of $N^\circ$ as in Definition~\Cref{def:Gk-CW}.
Then the groupoid of Stokes torsors on $N^\circ$ is canonically equivalent to the gauge groupoid of Stokes cocycles on $K$
as in Definition~\Cref{def:stokes-cochains}.
\end{proposition}
\begin{proof}
Let $\mathcal S=\{\mathrm{st}(v)\}_{v\in V}$ be the locally finite cover of $N^\circ$ by open stars of vertices of $K$.
It is a good cover, and every nonempty finite intersection $\mathrm{st}(v_0)\cap\cdots\cap \mathrm{st}(v_m)$ is contained in a single stratum,
so the Stokes sheaf $\St_\Phi$ is constant on it.
Therefore Stokes torsors on $N^\circ$ may be computed by nonabelian \v{C}ech $1$-cocycles on $\mathcal S$ with values in $\St_\Phi$,
modulo the usual $0$-cochain gauge action.
For vertices $v,w$, the intersection $\mathrm{st}(v)\cap \mathrm{st}(w)$ is nonempty precisely when $v$ and $w$ are joined by an edge
$e:v\to w$ in $K^{(1)}$, and any \v{C}ech cocycle assigns to this overlap a unique constant element
$u_e\in \St_\Phi(S(e))$, with $u_{\bar e}=u_e^{-1}$.
If $v_0,v_1,v_2$ span a $2$-cell, then $\mathrm{st}(v_0)\cap \mathrm{st}(v_1)\cap \mathrm{st}(v_2)$ is nonempty and contractible, and the
\v{C}ech cocycle condition on this triple overlap is exactly the boundary relation for that $2$-cell in Definition~\Cref{def:stokes-cochains}.
Finally, a \v{C}ech gauge $0$-cochain is a choice $h_v\in \St_\Phi(S(v))$ on each star $\mathrm{st}(v)$, and it acts by
$u_e\mapsto h_{t(e)}\,u_e\,h_{s(e)}^{-1}$, which is precisely the vertex gauge action in Definition~\Cref{def:stokes-cochains}.
This identifies the two groupoids.
\end{proof}
\begin{equation}\label{eq:stokes-skeleton-summary}
\begin{tikzcd}[row sep=large, column sep=huge]
\Tors(\St_\Phi;N^\circ)
\arrow[r,"\sim"] &
\bigl\{\text{edge labels }(u_e)\text{ satisfying $2$-cell relations}\bigr\}\!/\!\text{vertex gauge}.
\end{tikzcd}
\end{equation}

\begin{figure}[t]
\centering
\begin{tikzpicture}[
  >=Stealth,
  font=\scriptsize,
  vtx/.style={circle,draw,inner sep=1.2pt},
  lab/.style={fill=white,inner sep=1pt},
  box/.style={draw,rounded corners,inner sep=3pt,align=center}
]
\node[vtx] (v0) at (0,0) {$v_0$};
\node[vtx] (v1) at (3.2,0) {$v_1$};
\node[vtx] (v2) at (1.6,2.4) {$v_2$};

\draw[->,thick] (v0) -- (v1) node[midway,lab,below] {$u_{01}$};
\draw[->,thick] (v1) -- (v2) node[midway,lab,right] {$u_{12}$};
\draw[->,thick] (v0) -- (v2) node[midway,lab,left] {$u_{02}$};

\node[box] (rel) at (1.6,-1.05)
{$2$-cell relation\\ $u_{02}=u_{12}\,u_{01}$};

\node[lab,above left=2pt]  at (v0.north west) {$h_{v_0}$};
\node[lab,above right=2pt] at (v1.north east) {$h_{v_1}$};
\node[lab,above=2pt]       at (v2.north)      {$h_{v_2}$};

\node[box] (gauge) at (1.6,3.45)
{vertex gauge:\\ $u_{vw}\mapsto h_w\,u_{vw}\,h_v^{-1}$};

\draw[->,thin,dashed] (gauge.south) to[out=240,in=90] (v0.north);
\draw[->,thin,dashed] (gauge.south) to[out=270,in=90] (v2.north);
\draw[->,thin,dashed] (gauge.south) to[out=300,in=90] (v1.north);

\end{tikzpicture}
\caption{Cocycle data on the $1$-skeleton (edge labels) and relations from $2$-cells, modulo vertex gauge.}
\label{fig:stokes-skeleton-gauge}
\end{figure}

We now combine the group action with the cocycle description.

\begin{definition}[Equivariant Stokes cocycles]\label{def:equivariant-cocycles}
Let $G_k$ act on $K$ cellularly and on the Stokes data by pullback.
An \emph{equivariant Stokes cocycle} is a Stokes cocycle $u=(u_e)$ together with elements
$c_g(v)\in \St_\Phi(S(v))$ for all $g\in G_k$ and $v\in V$, such that:
\begin{itemize}
\item (equivariance on edges) for every oriented edge $e:v\to w$,
$$
u_{g\cdot e}=c_g(w)\cdot g(u_e)\cdot c_g(v)^{-1};
$$
\item (coherence) for all $g,h\in G_k$ and all vertices $v\in V$,
$$
c_{gh}(v)=c_g(v)\cdot g\big(c_h(v)\big),
\qquad c_1(v)=1.
$$
\end{itemize}
Morphisms are vertex gauges $h=(h_v)$ compatible with the $c_g$ in the evident way:
$c'_g(v)=h_v\,c_g(v)\,g(h_v)^{-1}$.
\end{definition}

\begin{theorem}[Visible descent = homotopy fixed points]\label{thm:visible-descent}
On any depth-$k$ chart, let $G_k=(\mu_n)^k$ be the Kummer deck group.
Then the intrinsic descent equivalence of Proposition~\Cref{prop:descent} is presented on the collar by a natural equivalence
\begin{equation}\label{eq:visible-descent-sections}
\Sec(\pi)\ \simeq\ \Sec(\pi_n)^{hG_k},
\end{equation}
where $\pi:\GC\to\CB$ is the collar projection downstairs and $\pi_n:\GCn\to\CBn$ is its pullback to level $n$.
Moreover, after choosing a $G_k$-equivariant CW skeleton $K$ of $N^\circ$, the same descent becomes completely explicit as an equivalence
\[
\Tors(\St_\Phi;N^\circ)\ \simeq\ (\text{$G_k$-equivariant Stokes cocycles on }K),
\]
where equivariant cocycles are as in Definition~\Cref{def:equivariant-cocycles}.
\end{theorem}

\begin{proof}
Let $q_n:N^\circ_n\to N^\circ$ be the restriction of the Kummer cover and write $\mathcal B$ for the sector-box cover of $N^\circ$.
Pulling back $\mathcal B$ along $q_n$ gives a $G_k$-stable cover $\mathcal B_n$ of $N^\circ_n$, hence a strict action of $G_k$ on the \v{C}ech groupoid
$\CB_n=\Cech(\mathcal B_n)$.The Stokes sheaf pulls back to a $G_k$-equivariant sheaf on $N^\circ_n$, hence the collar construction on the pullback cover $\mathcal B_n$ is $G_k$-equivariant:
the induced \v{C}ech--Stokes collar groupoid and its forgetful projection $\pi_n$ are $G_k$-equivariant. In particular, $G_k$ acts strictly on the section groupoid $\Sec(\pi_n)$.

Unwinding Definition~\Cref{def:hfp-finite}, an object of $\Sec(\pi_n)^{hG_k}$ is a section $\sigma\in\Sec(\pi_n)$ together with isomorphisms
\[
\varphi_g:g^*(\sigma)\xRightarrow{\ \sim\ }\sigma\qquad (g\in G_k)
\]
satisfying $\varphi_{gh}=\varphi_g\circ g(\varphi_h)$ and $\varphi_1=\id$.
Under the identification $\Sec(\pi_n)\simeq \Tors(\St_{q_n^*\Phi};\mathcal B_n)$ (Proposition~\Cref{prop:stokes-as-sections}),
a section $\sigma$ is a \v{C}ech Stokes cocycle on $\mathcal B_n$, and a natural isomorphism $\varphi_g$ is a \v{C}ech gauge $0$-cochain.
Thus the data $(\sigma,\varphi_\bullet)$ is exactly a descent datum for a Stokes torsor along the finite Galois cover $q_n$.
By effective Galois descent for torsors (Proposition~\Cref{prop:descent}), such descent data are equivalent to Stokes torsors downstairs, which is
$\Sec(\pi)$ by Proposition~\Cref{prop:stokes-as-sections}. This gives \eqref{eq:visible-descent-sections}.
For the skeletal description, choose a $G_k$-equivariant CW skeleton $K$ of $N^\circ$.
By Proposition~\Cref{prop:stokes-torsors-skeleton}, Stokes torsors on $N^\circ$ are encoded by edge labels on $K^{(1)}$ satisfying the $2$-cell relations,
modulo vertex gauge. Pullback to level $n$ produces the corresponding $G_k$-action on cocycles. Translating the coherence isomorphisms
$\varphi_g$ into vertex gauges $c_g(v)$ yields exactly the conditions in Definition~\Cref{def:equivariant-cocycles}:
naturality along each edge gives the equivariance equation on labels, and $\varphi_{gh}=\varphi_g\circ g(\varphi_h)$ gives the vertex coherence.
\end{proof}
\begin{equation}\label{eq:visible-descent-diagram}
\begin{tikzcd}[row sep=large, column sep=huge]
\Sec(\pi_n)^{hG_k} \arrow[r,"\sim"] \arrow[d] &
\Sec(\pi) \arrow[d] \\
\Tors(\St_{q_n^*\Phi};\mathcal B_n)^{hG_k} \arrow[r,"\sim"] &
\Tors(\St_\Phi;\mathcal B)
\end{tikzcd}
\end{equation}

\subsection{Equivariant cocycles on the stratified CW skeleton}\label{subsec:equivariant-cocycles}

In this subsection we spell out, in completely explicit terms, what the homotopy fixed point
description in \Cref{subsec:visible-kummer-presenter} means on a CW skeleton.
Nothing here goes beyond finite Galois descent; the point is that the descent constraints become \emph{visible}
as equivariance conditions on cocycles attached to cells.

Let $G$ be a finite group acting on a groupoid $\mathcal H$ by strict autoequivalences
(e.g.\ $G=(\mu_n)^k$ acting on a Kummer chart of $\Xlogn$).
Write $\mathcal H\mathbin{/\mkern-6mu/}G$ for the action groupoid: objects are those of $\mathcal H$ and
an arrow $x\to y$ is a pair $(g,\alpha)$ with $g\in G$ and $\alpha:g\cdot x\to y$ an arrow in $\mathcal H$.
There is a canonical functor $\mathcal H\to \mathcal H\mathbin{/\mkern-6mu/}G$.

For any groupoid $\mathcal K$, the \emph{homotopy fixed point groupoid} $\mathcal K^{hG}$ is, by definition,
the groupoid of functors $BG\to \mathcal K$ and natural isomorphisms between them. Concretely:

\begin{definition} \label{def:explicit-hfp}
An object of $\mathcal K^{hG}$ is a pair $(x,\{\varphi_g\}_{g\in G})$ where $x\in \Ob(\mathcal K)$ and
$\varphi_g:g\cdot x\xrightarrow{\sim}x$ are isomorphisms such that $\varphi_1=\id_x$ and, for all $g,h\in G$,
\begin{equation}\label{eq:hfp-coherence}
\varphi_{gh}=\varphi_g\circ g(\varphi_h): (gh)\cdot x\longrightarrow x.
\end{equation}
A morphism $(x,\varphi_\bullet)\to (x',\varphi'_\bullet)$ is an isomorphism $f:x\to x'$ in $\mathcal K$ such that
$f\circ\varphi_g=\varphi'_g\circ g(f)$ for all $g\in G$.
\end{definition}

The coherence condition \eqref{eq:hfp-coherence} is the commutativity of the diagram
\begin{equation}\label{eq:hfp-square}
\begin{tikzcd}[row sep=large, column sep=huge]
(gh)\cdot x \arrow[r,"g(\varphi_h)"] \arrow[dr,swap,"\varphi_{gh}"] & g\cdot x \arrow[d,"\varphi_g"] \\
& x
\end{tikzcd}
\end{equation}
for each $(g,h)$.

\subsubsection*{B. Equivariant cocycles on a $\Sigma$-adapted CW decomposition}
Now let $K$ be a locally finite CW decomposition of $N^\circ$ adapted to the stratification $\Sigma_{D,\Phi,n}$
(so every open cell lies in a unique stratum, and Stokes groups are constant on that stratum).
Assume $G=(\mu_n)^k$ acts on $K$ by cellular maps (this can always be achieved after barycentric subdivision on a Kummer chart).

Write $\CB^{\mathrm{sk}}:=\Pi_1^{\mathrm{sk}}(K)$ for the skeleton groupoid of the $2$-skeleton.
Let $\St_\Phi(S(e))$ denote the (constant) Stokes group on the stratum containing the open $1$-cell $e$.
A strict section of $\pi^{\mathrm{sk}}:\GC^{\mathrm{sk}}\to\CB^{\mathrm{sk}}$ is equivalently a choice of labels
$u_e\in \St_\Phi(S(e))$ for each oriented $1$-cell $e$, satisfying the $2$-cell relations and modulo vertex gauge.

\begin{definition}[$G$-equivariant Stokes cocycle on the skeleton]\label{def:equivariant-stokes-cocycle}
A \emph{$G$-equivariant Stokes cocycle} on $K$ is a cocycle $(u_e)$ for $\pi^{\mathrm{sk}}$
together with elements $a_g(v)\in \St_\Phi(S(v))$ for every $g\in G$ and every vertex $v$ such that:
\begin{enumerate}
\item (\emph{equivariance up to gauge}) for each oriented $1$-cell $e:v\to w$ and $g\in G$,
\begin{equation}\label{eq:equivariance-edge}
u_{g\cdot e} = a_g(w)\, g(u_e)\, a_g(v)^{-1}\quad \text{in }\St_\Phi(S(e));
\end{equation}
\item (\emph{coherence}) for all $g,h\in G$ and all vertices $v$,
\begin{equation}\label{eq:equivariance-coherence}
a_{gh}(v)=a_g(v)\, g(a_h(v)).
\end{equation}
\end{enumerate}
Morphisms between equivariant cocycles are vertex gauge transformations
$h(v)\in \St_\Phi(S(v))$ such that the usual conjugation sends $(u_e,a_g(v))$ to $(u'_e,a'_g(v))$.
\end{definition}

Equations \eqref{eq:equivariance-edge}--\eqref{eq:equivariance-coherence} are the skeletal form of \eqref{eq:hfp-coherence}:
the maps $a_g(v)$ encode the descent isomorphisms identifying the pullback cocycle $g(u)$ with $u$.

\begin{proposition}[Equivariant cocycles $=$ homotopy fixed points]\label{prop:eq-cocycles-hfp}
Let $G$ act cellularly on $K$ and hence on $\Sec(\pi^{\mathrm{sk}})$.
Then the groupoid of $G$-equivariant Stokes cocycles on $K$ in \Cref{def:equivariant-stokes-cocycle}
is canonically equivalent to the homotopy fixed point groupoid $\Sec(\pi^{\mathrm{sk}})^{hG}$.
\end{proposition}

\begin{proof}
By definition, $\Sec(\pi^{\mathrm{sk}})^{hG}$ is the groupoid of functors $BG\to \Sec(\pi^{\mathrm{sk}})$.
An object of $\Sec(\pi^{\mathrm{sk}})$ is a cocycle $(u_e)$ modulo vertex gauge.
Choosing a lift of $(u_e)$ to an \emph{honest} cocycle (rather than its gauge class) identifies an isomorphism
$g\cdot (u_e)\cong (u_e)$ in $\Sec(\pi^{\mathrm{sk}})$ with a vertex gauge $a_g(v)$.
The condition that the isomorphisms form a functor $BG\to \Sec(\pi^{\mathrm{sk}})$ is precisely the cocycle identity
\eqref{eq:equivariance-coherence}, and naturality on edges gives \eqref{eq:equivariance-edge}.
Morphisms are the usual $0$-cochains $h(v)$ intertwining the chosen $a_g(v)$.
\end{proof}

\begin{remark}[How this matches standard finite Galois descent]\label{rem:eq-cocycles-descent}
If $\mathcal U$ is a $G$-stable good cover of a depth-$k$ Kummer chart and $\CB=\Cech(\mathcal U)$,
then $\Sec(\pi)^{hG}$ is the groupoid of $G$-equivariant \v{C}ech cocycles with values in $\St_\Phi$.
The skeletal description above is the same information written on generators (edges) and relations (2-cells).
\end{remark}

\section{Generators, relations, and the coequalizer model}\label{sec:coeq}

\subsection{Free groupoids on graphs}
A \emph{directed graph} $E$ is a pair of sets $(E_0,E_1)$ with $s,t:E_1\to E_0$.
The \emph{free groupoid} $F(E)$ is the groupoid whose arrows are reduced words in $E_1$ and inverses, modulo cancellations.

\subsection{Coequalizer presentation}
Let $E$ be a graph and $R$ a set of relations, each $\rho\in R$ specifying two arrows
$A_\rho,B_\rho\in \Hom_{F(E)}(x_\rho,y_\rho)$ with the same endpoints.
Define $r_0,r_1:F(R)\rightrightarrows F(E)$ by sending $\rho$ to $A_\rho$ (resp.\ $B_\rho$).
Let $Q$ be the coequalizer of $(r_0,r_1)$ in $\mathbf{Grpd}$.

\begin{proposition}\label{prop:coeq}
The groupoid $Q$ is canonically the quotient of $F(E)$ by the smallest congruence identifying $A_\rho=B_\rho$ for all $\rho\in R$.
\end{proposition}

\begin{proof}
A functor $F(E)\to H$ factors through $Q$ iff it equalizes $r_0$ and $r_1$, i.e.\ sends $A_\rho$ and $B_\rho$ to the same arrow for all $\rho$.
Closure under composition and inversion yields the congruence.
\end{proof}

 \begin{figure}[t]
\centering
\begin{tikzpicture}[
  >=Stealth,
  font=\scriptsize,
  vtx/.style={circle,draw,inner sep=1.2pt},
  lab/.style={fill=white,inner sep=1pt},
  box/.style={draw,rounded corners,inner sep=4pt,align=center}
]

\node[box] (Ebox) at (-6.2,2.1) {directed graph $E=(E_0,E_1)$};

\node[vtx] (x) at (-7.6,0.6) {$x$};
\node[vtx] (y) at (-4.8,0.6) {$y$};
\node[vtx] (z) at (-6.2,-1.0) {$z$};

\draw[->,thick] (x) -- (y) node[midway,lab,above] {$e$};
\draw[->,thick] (y) -- (z) node[midway,lab,right] {$f$};
\draw[->,thick,bend left=22] (x) to (z) node[midway,lab,left] {$g$};

\node[lab,align=center] at (-6.2,-2.05)
{$s,t:E_1\to E_0$ record\\ source and target};

\node[box] (Fbox) at (0.0,2.1) {free groupoid $F(E)$};

\node[lab,align=left] (Ftxt) at (0.0,0.2)
{arrows are reduced words in $E_1\cup E_1^{-1}$\\
example: $fe:x\to z$,\quad $e^{-1}:y\to x$\\
cancellation: $ff^{-1}=\id_y$};

\node[box] (Qbox) at (6.2,2.1) {relations and coequalizer};

\node[lab,align=left] (Qtxt) at (6.2,0.75)
{a relation $\rho$ specifies $A_\rho,B_\rho:x_\rho\to y_\rho$\\
define $r_0,r_1:F(R)\rightrightarrows F(E)$\\
send $\rho\mapsto A_\rho$ and $\rho\mapsto B_\rho$};

\node[vtx] (xr) at (5.0,-0.95) {$x$};
\node[vtx] (zr) at (7.4,-0.95) {$z$};

\draw[->,thick,bend left=25] (xr) to (zr) node[midway,lab,above] {$A_\rho=fe$};
\draw[->,thick,bend right=25] (xr) to (zr) node[midway,lab,below] {$B_\rho=g$};

\node[lab,align=center] at (6.2,-2.15)
{$Q=\mathrm{coeq}\bigl(F(R)\rightrightarrows F(E)\bigr)$\\
forces $A_\rho=B_\rho$};

\draw[->,thick,bend left=15] (Ebox.east) to node[lab,above] {$E\mapsto F(E)$} (Fbox.west);
\draw[->,thick,bend left=15] (Fbox.east) to node[lab,above] {impose $R$} (Qbox.west);

\end{tikzpicture}
\caption{A directed graph $E$ determines the free groupoid $F(E)$ by adding formal inverses and cancelling $aa^{-1}$; imposing relations $A_\rho=B_\rho$ is encoded by the coequalizer $Q=\mathrm{coeq}(F(R)\rightrightarrows F(E))$.}
\label{fig:free-groupoid-coeq}
\end{figure}

\subsection{Applying the coequalizer to $\G_{\Phi,n}$}
Let $E_0:= (\GU)_0\bigsqcup (\GC)_0$ and $E_1:=(\GU)_1\bigsqcup (\GC)_1$.
Let $R$ impose: (i) internal relations of $\GU$, (ii) internal relations of $\GC$, (iii) gluing relations identifying $j_U$ and $j_C$ on objects/arrows of $\GX$.

\begin{proposition}\label{prop:pushout-coeq}
With $E$ and $R$ as above, the explicit $2$-pushout groupoid $\G_{\Phi,n}$ is canonically isomorphic to the coequalizer
\begin{equation}\label{eq:pushout-coeq}
F(R)\rightrightarrows F(E)\longrightarrow \G_{\Phi,n}.
\end{equation}
\end{proposition}

\begin{proof}
The congruence generated by $R$ forces the subgroupoid structures of $\GU$ and $\GC$ and imposes the strict identifications prescribed by $j_U,j_C$ on $\GX$.
This is exactly the generators-and-relations model of the explicit $2$-pushout.
\end{proof}


\subsection{A completely explicit curve model: the explicit $2$-pushout groupoid $\G_{\Phi,n}$ and the forgetful map $\pi:\GC\to\CB$}\label{subsec:curve-model-clean}

This subsection is expository and is not used later.
Its goal is to make the explicit $2$-pushout
\begin{equation}\label{eq:curve-pushout-clean}
\G_{\Phi,n}\ :=\ \GU\sqcup_{\GX}\GC
\end{equation}
completely concrete in complex dimension~$1$, and to clarify the comparison with the groupoids
in \cite[\S3,\S5]{GualtieriLiPym} and the Kato--Nakayama viewpoint.

\medskip

\noindent\textbf{Important notation (to avoid the $\mathbb C$ clash).}
Throughout, $\CC$ denotes the complex numbers.
We reserve the symbol $\CB$ for the \v{C}ech groupoid of the collar cover~$\mathcal B$ (see below).
In particular, \emph{$\CB$ is not $\CC$ with a subscript}.

\medskip

\providecommand{\Cech}{\check{\mathrm C}}
\providecommand{\CB}{\Cech(\mathcal B)}

\subsubsection*{(A) Geometry on a curve and sector boxes on $X^{\log}_n$}

Let $X$ be a connected Riemann surface and let $D=\{p_1,\dots,p_m\}\subset X$ be a reduced finite divisor.
Write $U:=X\setminus D$.
Fix a Kummer level $n\ge 1$ and an irregular type $\Phi$ along $D$ (at level~$n$).
Let $\tau_n:X^{\log}_n\to X$ be the level-$n$ Kato--Nakayama/Kummer map.
Over $U$, $\tau_n$ is a homeomorphism; over each $p_i$ the fiber is the level-$n$ circle of directions
$S^1_{p_i,n}:=\tau_n^{-1}(p_i)$.
Choose pairwise disjoint closed discs $\Delta_i$ around $p_i$ and set $\Delta_i^\times:=\Delta_i\setminus\{p_i\}$.
Let $N_i^\times:=\tau_n^{-1}(\Delta_i^\times)\subset X^{\log}_n$ be the punctured collar near $p_i$.
The Stokes walls at $p_i$ cut $S^1_{p_i,n}$ into finitely many open arcs
\begin{equation}\label{eq:arcs-clean}
S^1_{p_i,n}\setminus\Sigma_{p_i}\ =\ \bigsqcup_{a\in\ZZ/\ell_i\ZZ} I_{i,a},
\end{equation}
and on each arc the Stokes preorder is constant.
Choose \emph{sector boxes}
\begin{equation}\label{eq:boxes-clean}
B_{i,a}\ \subset\ N_i^\times,
\qquad
B_{i,a}\ \simeq\ \Delta_i^\times\times I_{i,a},
\qquad
B_{i,a}\cap B_{i,a+1}\neq\varnothing,
\end{equation}
each contractible, and with overlaps also contractible.
\begin{figure}[t]
\centering
\begin{tikzpicture}[
  >=Stealth,
  font=\scriptsize,
  lab/.style={fill=white,inner sep=1pt},
  vtx/.style={circle,draw,inner sep=1.2pt},
  box/.style={draw,rounded corners,inner sep=3pt,align=center}
]

\node[box] (Xbox) at (-6.0,2.35) {$X$ with $D=\{p_1,\dots,p_m\}$,\quad $U=X\setminus D$};

\draw[thick] (-7.6,0.55) .. controls (-7.4,1.75) and (-4.6,1.75) .. (-4.4,0.55)
             .. controls (-4.6,-0.65) and (-7.4,-0.65) .. (-7.6,0.55);

\node[vtx] (pi) at (-6.85,0.95) {$p_i$};
\draw[dashed,thick] (-6.85,0.95) circle (0.55);
\node[lab] at (-6.85,1.62) {$\Delta_i$};

\draw[->,thick,bend left=12] (-4.4,0.55) to node[lab,above] {zoom at $p_i$} (-1.70,0.85);

\node[box] (Fbox) at (0.0,2.35) {$\tau_n^{-1}(p_i)=S^1_{p_i,n}$};

\def\R{1.15}
\coordinate (O) at (0,0.75);
\draw[thick] (O) circle (\R);

\foreach \ang in {25,115,205,295}{
  \draw[thick] (O) -- ({\R*cos(\ang)},{0.75+\R*sin(\ang)});
}
\node[lab,anchor=west] at (1.32,0.75) {$\Sigma_{p_i}$};

\node[lab] at ({1.48*cos(70)},{0.75+1.48*sin(70)}) {$I_{i,a}$};
\node[lab] at ({1.48*cos(150)},{0.75+1.48*sin(150)}) {$I_{i,a+1}$};
\node[lab] at ({1.48*cos(250)},{0.75+1.48*sin(250)}) {$I_{i,a+2}$};

\node[lab,align=center] at (0,-1.05)
{$S^1_{p_i,n}\setminus\Sigma_{p_i}=\bigsqcup_{a\in\ZZ/\ell_i\ZZ} I_{i,a}$};

\draw[->,thick,bend left=10] (1.70,0.85) to node[lab,above] {thicken} (3.55,0.95);

\node[box] (Nbox) at (5.6,2.35) {$N_i^\times=\tau_n^{-1}(\Delta_i^\times)$,\quad $\Delta_i^\times=\Delta_i\setminus\{p_i\}$};

\draw[thick] (3.90,-0.10) rectangle (7.30,1.65);

\draw[->,thin] (3.95,-0.55) -- (7.25,-0.55)
  node[lab,midway,below] {radial in $\Delta_i^\times$};
\draw[->,thin] (7.65,-0.10) -- (7.65,1.65)
  node[lab,midway,right] {angular};

\draw[dashed,thick] (4.55,-0.10) -- (4.55,1.65);
\draw[dashed,thick] (5.25,-0.10) -- (5.25,1.65);
\draw[dashed,thick] (5.95,-0.10) -- (5.95,1.65);
\draw[dashed,thick] (6.65,-0.10) -- (6.65,1.65);

\node[lab] at (4.90,0.85) {$B_{i,a}$};
\node[lab] at (5.60,0.85) {$B_{i,a+1}$};

\node[lab,align=center] at (5.60,-1.30)
{$B_{i,a}\simeq \Delta_i^\times\times I_{i,a}$,\quad
$B_{i,a}\cap B_{i,a+1}\neq\varnothing$\\
boxes and overlaps chosen contractible};

\end{tikzpicture}
\caption{At a puncture $p_i$, the level-$n$ circle of directions $S^1_{p_i,n}$ is cut by Stokes walls $\Sigma_{p_i}$ into finitely many arcs $I_{i,a}$.
A punctured collar $N_i^\times$ over $\Delta_i^\times$ is then covered by sector boxes $B_{i,a}\simeq \Delta_i^\times\times I_{i,a}$ with adjacent overlaps.}
\label{fig:curve-collar-sector-boxes}
\end{figure}

Let $\St_\Phi$ be the Stokes sheaf on the sector cover (as in Definition~\ref{def:stokes-sheaf}).
Since $B_{i,a}\cap B_{i,a+1}$ is contractible, the group
\begin{equation}\label{eq:stokes-groups-clean}
\St_{i;a,a+1}\ :=\ \St_\Phi(B_{i,a}\cap B_{i,a+1})
\end{equation}
is a constant (typically unipotent) subgroup of the relevant $\Aut(\Gr)$.

\subsubsection*{(B) The three groupoids $\GU,\GX,\GC$ and the \v{C}ech collar groupoid $\CB$}

Fix once and for all the following covers.

\begin{itemize}
\item \textbf{Interior cover on $U$.} Choose a good cover $\mathcal U=\{U_\alpha\}$ of $U$.
Set $\GU:=\Cech(\mathcal U)$.

\item \textbf{Collar cover on $N^\times:=\bigsqcup_i N_i^\times$.}
Set $\mathcal B:=\{B_{i,a}\}_{i,a}$ and define the \emph{plain collar \v{C}ech groupoid}
\begin{equation}\label{eq:CB-def-clean}
\CB\ :=\ \Cech(\mathcal B).
\end{equation}

\item \textbf{Overlap cover.} Consider the overlaps $U_\alpha\cap B_{i,a}\subset U\cap N^\times$ and let
$\mathcal W:=\{W_{\alpha;i,a}\}$ be any refinement by contractible opens.
Set $\GX:=\Cech(\mathcal W)$.
\end{itemize}

Now define the \emph{collar \v{C}ech--Stokes groupoid} $\GC$ as the Stokes decoration of $\CB$:
\begin{equation}\label{eq:GC-objects-clean}
(\GC)_0\ :=\ \bigsqcup_{i,a} B_{i,a},
\end{equation}
and arrows
\begin{equation}\label{eq:GC-arrows-clean}
(\GC)_1\ :=\ \bigsqcup_{i,a,a'}\Bigl((B_{i,a}\cap B_{i,a'})\times \St_\Phi(B_{i,a}\cap B_{i,a'})\Bigr),
\end{equation}
where $(x,u)$ is an arrow from the object $B_{i,a}$ to the object $B_{i,a'}$ (at the point $x$),
and composition is
\begin{equation}\label{eq:GC-comp-clean}
(x,u')\circ(x,u)\ :=\ (x,u'u).
\end{equation}

\medskip

\noindent\textbf{The forgetful projection.}
There is a \emph{strict} functor
\begin{equation}\label{eq:pi-forgetful-clean}
\pi:\GC\longrightarrow \CB
\end{equation}
defined by \emph{forgetting the Stokes decoration}:
on objects it is the identity (same sector boxes), and on arrows
\begin{equation}\label{eq:pi-on-arrows-clean}
\pi(x,u)\ :=\ x \in (B_{i,a}\cap B_{i,a'})\ \subset\ (\CB)_1.
\end{equation}
So here $\CB$ is the \v{C}ech groupoid of the cover $\mathcal B$, \emph{not} $\CC$.

\subsubsection*{(C) The explicit $2$-pushout and its ``generators-and-relations'' meaning}

There are canonical strict functors $j_U:\GX\to\GU$ and $j_C:\GX\to\GC$ induced by inclusions:
on objects, $W_{\alpha;i,a}\subset U_\alpha$ and $W_{\alpha;i,a}\subset B_{i,a}$;
on arrows, overlaps map to the corresponding \v{C}ech arrows, and in $\GC$ we take the
\emph{trivial} Stokes label $u=1$.

\begin{definition} \label{def:curve-pushout-clean}
Define
\begin{equation}\label{eq:pushout-clean}
\G_{\Phi,n}\ :=\ \GU\sqcup_{\GX}\GC.
\end{equation}
\end{definition}

Concretely, $\G_{\Phi,n}$ is obtained from the disjoint union groupoid $\GU\bigsqcup\GC$
by imposing precisely the identifications dictated by $\GX$ (gluing the overlap charts),
while keeping the Stokes-decorated arrows as genuinely new arrows.
Equivalently, $\G_{\Phi,n}$ admits the standard coequalizer presentation
\begin{equation}\label{eq:coeq-clean}
\G_{\Phi,n}\ \cong\ \mathrm{coeq}\Bigl(F(R)\rightrightarrows F(E)\Bigr),
\end{equation}
where $E$ consists of the generating arrows coming from $\GU$ and $\GC$,
and $R$ encodes only the overlap relations from $\GX$.

\subsubsection*{(D) A finite skeleton (graph-of-groupoids picture) on curves}

The \v{C}ech presentation is explicit but indexed by covers.
On a curve one can replace it by a finite skeleton without changing representation theory (Morita reduction).

Fix a basepoint $x_0\in U$.
For each sector box $B_{i,a}$ pick a point $b_{i,a}\in B_{i,a}$ and choose a path class
$\gamma_{i,a}:x_0\to b_{i,a}$ inside $U\cup N^\times$ entering the collar through the sector $I_{i,a}$.

Define a small groupoid $\G^{\mathrm{sk}}_{\Phi,n}$ by:
\begin{itemize}
\item objects: $\{x_0\}\cup\{b_{i,a}\}_{i,a}$;
\item arrows generated by:
  \begin{enumerate}
  \item $\pi_1(U,x_0)$ as loops at $x_0$ (ordinary monodromy),
  \item the spokes $\gamma_{i,a}:x_0\to b_{i,a}$ and formal inverses,
  \item for each fixed $i$ and adjacent sectors $a\to a+1$, arrows
  $\sigma_{i,a}(u):b_{i,a}\to b_{i,a+1}$ for each $u\in \St_{i;a,a+1}$,
  \end{enumerate}
\item relations:
  \begin{enumerate}
  \item $\sigma_{i,a+1}(v)\circ \sigma_{i,a}(u)=\sigma_{i,a}(vu)$ whenever the compositions correspond to the same overlap,
  \item the overlap/gluing relations expressing compatibility between the restriction of monodromy on $U$
  and the boundary Stokes jumps (the same relations encoded by $\GX$),
  \item the usual groupoid relations (identities/inverses).
  \end{enumerate}
\end{itemize}

\begin{proposition}\label{prop:curve-skeleton-clean}
There is a canonical chain of Morita equivalences connecting the cover-based presenter $\G_{\Phi,n}$ and the finite skeletal presenter
$\G^{\mathrm{sk}}_{\Phi,n}$. In particular, they present the same classifying stack, and for every $r\ge 1$ there is a natural equivalence
of groupoids
\begin{equation}\label{eq:repr-clean}
\Rep_r(\G_{\Phi,n})\ \simeq\ \Rep_r(\G^{\mathrm{sk}}_{\Phi,n}).
\end{equation}
\end{proposition}

\begin{proof}
Fix the good cover $\mathcal U=\{U_\alpha\}$ of $U$, the sector-box cover $\mathcal B=\{B_{i,a}\}$ of $N^\times$, and a contractible refinement
$\mathcal W=\{W_{\alpha;i,a}\}$ of the overlap $U\cap N^\times$ as in \Cref{subsec:curve-model-clean}. The global presenter is the explicit $2$-pushout
\begin{equation}\label{eq:curve-pushout-cech}
\G_{\Phi,n}\ :=\ \GU\ \bigsqcup_{\GX}\ \GC,
\qquad
\GU=\Cech(\mathcal U),\ \GX=\Cech(\mathcal W),\ \GC=\Cech\text{--Stokes}(\mathcal B).
\end{equation}

Choose basepoints $x_\alpha\in U_\alpha$ and $b_{i,a}\in B_{i,a}$, and choose basepoints in each $W_{\alpha;i,a}$ compatibly with the inclusions
$W_{\alpha;i,a}\subset U_\alpha$ and $W_{\alpha;i,a}\subset B_{i,a}$.
Since all charts and all nonempty finite intersections are contractible, each connected overlap carries a unique homotopy class of paths between the
chosen basepoints. Therefore each \v{C}ech presenter admits the standard Morita reduction to a small groupoid with one object per chart and one arrow
per nonempty overlap:
\begin{equation}\label{eq:morita-reductions}
\GU\ \simeq_{\mathrm{Morita}}\ \GU^{\mathrm{bp}},
\qquad
\GX\ \simeq_{\mathrm{Morita}}\ \GX^{\mathrm{bp}},
\qquad
\GC\ \simeq_{\mathrm{Morita}}\ \GC^{\mathrm{bp}}.
\end{equation}
Moreover, the refinement functors $j_U:\GX\to\GU$ and $j_C:\GX\to\GC$ induce functors
$j_U^{\mathrm{bp}}:\GX^{\mathrm{bp}}\to\GU^{\mathrm{bp}}$ and $j_C^{\mathrm{bp}}:\GX^{\mathrm{bp}}\to\GC^{\mathrm{bp}}$, compatible with
\eqref{eq:morita-reductions}. Taking explicit $2$-pushouts and using functoriality of $\bigsqcup$ gives a Morita equivalence between pushouts,
hence a canonical chain
\begin{equation}\label{eq:morita-chain-curve}
\G_{\Phi,n}=\GU\bigsqcup_{\GX}\GC\ \simeq_{\mathrm{Morita}}\ \GU^{\mathrm{bp}}\bigsqcup_{\GX^{\mathrm{bp}}}\GC^{\mathrm{bp}}
=: \G^{\mathrm{bp}}_{\Phi,n}.
\end{equation}
This passage is summarized by Diagram~\eqref{eq:curve-morita-cube} below.

Finally, by construction the reduced pushout $\G^{\mathrm{bp}}_{\Phi,n}$ is the explicit finite skeleton
$\G^{\mathrm{sk}}_{\Phi,n}$ introduced in \Cref{subsec:curve-model-clean}: it has objects $x_0$ and the sector basepoints $b_{i,a}$, and it is
generated by ordinary monodromy in $U$ together with the Stokes-labelled arrows between adjacent sectors, modulo the overlap relations coming from
$\GX$. Thus $\G^{\mathrm{bp}}_{\Phi,n}\cong \G^{\mathrm{sk}}_{\Phi,n}$. Applying $\Rep_r(-)$, which is invariant under Morita equivalence of
presenters, yields \eqref{eq:repr-clean}.

\begin{equation}\label{eq:curve-morita-cube}
\begin{tikzcd}[row sep=large, column sep=large]
\GU
\arrow[d, "\simeq_{\mathrm{Morita}}" description]
\arrow[rrr, bend left=18, "i_U"]
&
\GX
\arrow[l, "j_U"']
\arrow[r, "j_C"]
\arrow[d, "\simeq_{\mathrm{Morita}}" description]
&
\GC
\arrow[d, "\simeq_{\mathrm{Morita}}" description]
\arrow[r, "i_C"]
&
\GU \bigsqcup_{\GX} \GC
\arrow[d, "\simeq_{\mathrm{Morita}}" description]
\\
\GU^{\mathrm{bp}}
\arrow[rrr, bend right=18, "i_U^{\mathrm{bp}}"']
&
\GX^{\mathrm{bp}}
\arrow[l, "j_U^{\mathrm{bp}}"']
\arrow[r, "j_C^{\mathrm{bp}}"]
&
\GC^{\mathrm{bp}}
\arrow[r, "i_C^{\mathrm{bp}}"]
&
\GU^{\mathrm{bp}} \bigsqcup_{\GX^{\mathrm{bp}}} \GC^{\mathrm{bp}}
\arrow[d, "\cong"]
\\
&&&
\G^{\mathrm{sk}}_{\Phi,n}.
\end{tikzcd}
\end{equation}
\end{proof}

\subsubsection*{(E) Consistency checks and comparison with the literature}

\begin{itemize}
\item \textbf{Regular singular case.} If $\St_\Phi$ is trivial, then $\GC$ collapses to $\CB$ and the pushout
reduces to the usual (log) fundamental groupoid picture (no Stokes jumps).

\item \textbf{Rank one.} The Stokes sheaf is trivial (no off-diagonal jumps), matching the rank-one phenomenon
emphasized in \cite[\S5]{GualtieriLiPym}.

\item \textbf{Relation with \cite{GualtieriLiPym}.}
Their Stokes groupoids give explicit Lie/topological models on the real oriented blow-up;
our $\GC$ is the strict \v{C}ech--Stokes avatar built from sector boxes. Discretizing the boundary circle by arcs
and collapsing arcs to points yields precisely the ``arc skeleton'' underlying $\G^{\mathrm{sk}}_{\Phi,n}$.
\end{itemize}

\begin{figure}[t]
\centering
\begin{tikzpicture}[font=\small]
\tikzset{
  box/.style={draw,rounded corners,inner sep=6pt,align=center},
  arr/.style={->,>=latex,thick},
  v/.style={circle,draw,inner sep=1.1pt}
}

\node[box] (GX) at (0,3.10)
{\begin{tabular}{@{}c@{}}
$\GX$\\[-2pt]
\footnotesize overlap \v{C}ech groupoid
\end{tabular}};

\node[box] (GU) at (-5.35,0)
{\begin{tabular}{@{}c@{\hspace{8pt}}p{4.9cm}@{}}
\begin{tikzpicture}[baseline={(b.base)}]
  \node (b) at (0,0) {};
  \draw[thin] (0,0) circle (0.45);
  \draw[thin] (0.45,0.08) circle (0.45);
  \draw[thin] (0.18,0.45) circle (0.45);
  \node[font=\scriptsize] at (0,-0.75) {$\{U_i\}$};
\end{tikzpicture}
&
\raggedright
\textbf{Interior presenter} $\GU$\\[-2pt]
\footnotesize \v{C}ech groupoid of a good cover of $U:=X\setminus D$.\\[2pt]
\scriptsize Moduli: $\Rep_r(\GU)$  .
\end{tabular}};

\node[box] (GC) at (5.35,0)
{\begin{tabular}{@{}c@{\hspace{8pt}}p{5.2cm}@{}}
\begin{tikzpicture}[baseline={(b.base)}]
  \node (b) at (0,0) {};
  \def\R{0.70}
  \draw[thin] (0,0) circle (\R);
  \foreach \a in {0,90,180,270}{\draw[thin] (0,0) -- ({\R*cos(\a)},{\R*sin(\a)});}
  \node[font=\scriptsize] at (0,-0.95) {sectors on $S^1$};
\end{tikzpicture}
&
\raggedright
\textbf{Collar presenter} $\GC$\\[-2pt]
\footnotesize \v{C}ech--Stokes groupoid on $N^\times\subset \Xlogn$.\\[2pt]
\scriptsize Projection $\pi:\GC\to\CB$,   collar moduli: $\Sec(\pi)$.\\[2pt]
\scriptsize Wild: Stokes jumps in $\St_\Phi$;  tame: $(\mu_n)^k$ in depth $k$.
\end{tabular}};

\node[box] (GP) at (0,-2.95)
{\begin{tabular}{@{}c@{}}
$\G_{\Phi,n}:=\GU\sqcup_{\GX}\GC$\\[-2pt]
\footnotesize explicit $2$-pushout presenter
\end{tabular}};

\draw[arr] (GX) -- node[left] {$j_U$} (GU.north);
\draw[arr] (GX) -- node[right] {$j_C$} (GC.north);
\draw[arr] (GU.south) -- node[left] {$i_U$} (GP.north west);
\draw[arr] (GC.south) -- node[right] {$i_C$} (GP.north east);

\node[align=center,font=\small] at (0,-4.15)
{$\Rep_r(\GU)\times^{(2)}_{\Rep_r(\GX)}\Sec(\pi)\ \simeq\ \Stokes_r(\Xlogn;D,\Phi)$.};

\end{tikzpicture}
\caption{The interior \v{C}ech presenter $\GU$ and the collar \v{C}ech--Stokes presenter $\GC$ (with $\pi:\GC\to\CB$) glued along the overlap presenter $\GX$. Their explicit $2$-pushout $\G_{\Phi,n}$ presents global Stokes objects via torsorial gluing.}
\label{fig:pushout-two-panel}
\end{figure}

\subsubsection*{Comparison with the explicit local Stokes groupoid of \cite{GualtieriLiPym} and our Cech--Stokes collar}

Fix a puncture $p\in D$ and choose a holomorphic coordinate $z$ on a small disc $\Delta\subset X$
centered at $p$. Write $\Delta^\times:=\Delta\setminus\{p\}$.
Let $\tau_n:X^{\log}_n\to X$ be the level-$n$ Kato--Nakayama (Kummer) map and set
$$
N^\times:=\tau_n^{-1}(\Delta^\times)\subset X^{\log}_n.
$$

\medskip

\noindent\textbf{Our collar groupoid near $p$.}
By construction, the collar groupoid used in the pushout is a Cech--Stokes groupoid built from a sector cover of $N^\times$.
Choose the Stokes directions for the irregular type $\Phi$ at $p$.
They cut the boundary circle of directions
$$
S^1_{p,n}:=\tau_n^{-1}(p)\subset X^{\log}_n
$$
into finitely many open arcs $I_0,\dots,I_{\ell-1}$ (indices taken modulo $\ell$).
Choose contractible sector boxes $B_a\subset N^\times$ such that
$$
B_a \simeq \Delta^\times\times I_a,
\qquad
B_a\cap B_{a+1}\neq\varnothing,
$$
again with indices modulo $\ell$.

On each overlap $B_a\cap B_{a+1}$ the Stokes sheaf $\mathrm{St}_\Phi$ is locally constant, and since
$B_a\cap B_{a+1}$ is contractible it is in fact constant there. Thus the additional ``Stokes arrows''
in the collar groupoid are labelled by the groups
$$
\mathrm{St}_\Phi(B_a\cap B_{a+1}).
$$

\medskip

\paragraph{Elementary one-parameter case.}
Assume that the Stokes groups on the overlaps are all canonically isomorphic to a single one-parameter unipotent group.
Equivalently, on each overlap $B_a\cap B_{a+1}$ the Stokes sheaf is modelled on one copy of $\mathbb{G}_a$.
Fix identifications
\begin{equation}\label{eq:one-parameter-stokes}
\mathrm{St}_{\Phi}(B_a\cap B_{a+1})\;\cong\;(\mathbb{C},+),
\\
u\longmapsto \bm{s}_a(u)\colon a\longrightarrow a+1,
\end{equation}
such that composition of Stokes arrows corresponds to addition of parameters
(so $\mathbf{s}_{a+1}(v)\circ \mathbf{s}_a(u)$ corresponds to $u+v$ whenever the composition is defined).

In this situation, one can replace the sectorwise description by a single explicit Lie groupoid chart
as in \cite[Thm.~3.21]{GualtieriLiPym}. We record this chart as a convenient local model.

\begin{definition}[Explicit local model groupoid]\label{def:Stok-local-model}
Fix an integer $k\ge 2$. Define a Lie groupoid $\mathrm{Stok}_k$ over the punctured disc $\Delta^\times$ as follows.
\begin{itemize}
\item \emph{Objects:} $(\mathrm{Stok}_k)_0:=\Delta^\times$ with coordinate $z$.
\item \emph{Arrows:} $(\mathrm{Stok}_k)_1:=\Delta^\times\times \mathbb{C}$ with coordinates $(z,u)$.
\item \emph{Source and target:}
\begin{equation}\label{eq:Stok-structure}
s(z,u)=z,
\qquad
t(z,u)=\exp\!\big(u\,z^{k-1}\big)\,z.
\end{equation}
\item \emph{Composition:} whenever $z_2=t(z_1,u_1)$, define
\begin{equation}\label{eq:Stok-comp}
(z_2,u_2)\circ (z_1,u_1)
=
\bigl(z_1,\ u_2\exp\!\big((k-1)u_1 z_1^{k-1}\big)+u_1\bigr).
\end{equation}
Units and inverses are determined by these formulas.
\end{itemize}
\end{definition}

\begin{remark}[Level-$n$ pullback]\label{rem:Stok-level-n}
Pull back along the level-$n$ Kummer cover $z=w^n$.
After a harmless reparametrisation of the unipotent coordinate $u$ (absorbing constants into $u$),
one may use the equivalent chart on the punctured $w$-disc:
\begin{equation}\label{eq:Stok-level-n}
s(w,u)=w,
\qquad
t(w,u)=\exp\!\big(u\,w^{\,n(k-1)}\big)\,w.
\end{equation}
This is the normal form convenient for comparison with the level-$n$ collar construction on $X^{\log}_n$.
\end{remark}


\begin{proposition}[Local comparison with the GLP chart in the one-parameter case]\label{prop:collar-has-GLP-chart}
Fix a puncture $p\in D$ and a small disc $\Delta$ centered at $p$, with $\Delta^\times=\Delta\setminus\{p\}$.
Assume that on each adjacent sector overlap the Stokes sheaf is identified with a single one-parameter unipotent group
(as in \eqref{eq:one-parameter-stokes}).  Let $N^\times=\tau_n^{-1}(\Delta^\times)\subset X_n^{\log}$ be the level-$n$ punctured collar.

Then the restriction of the sectorwise \v{C}ech--Stokes collar groupoid to $N^\times$ is Morita equivalent to the explicit Lie groupoid
$\mathrm{Stok}_k$ of Definition~\ref{def:Stok-local-model} (equivalently, to its level-$n$ pullback written in \eqref{eq:Stok-level-n}).
\end{proposition}

\begin{proof}
Choose a cyclic sector-box cover $\mathcal B=\{B_a\}_{a\in\ZZ/\ell\ZZ}$ of the regular collar $N^\circ=N^\times\setminus\Sigma_{p,n}$ such that
each $B_a$ is contractible, each adjacent overlap $B_{a,a+1}:=B_a\cap B_{a+1}$ is contractible, and there are no triple overlaps
(after shrinking if necessary).  In the one-parameter hypothesis, for every adjacent overlap we fix an identification
$\St_\Phi(B_{a,a+1})\cong(\CC,+)$ compatible with restrictions, so the sectorwise collar groupoid
$\mathcal G^{\mathrm{sec}}\rightrightarrows \bigsqcup_a B_a$ has arrows
\[
(\mathcal G^{\mathrm{sec}})_1=\bigsqcup_a\bigl(B_{a,a+1}\times\CC\bigr),
\]
with composition given by addition of the parameter on composable arrows.
Let $\mathcal H$ be the restriction of the GLP chart $\mathrm{Stok}_k\rightrightarrows \Delta^\times$ to the collar region
corresponding to $N^\circ$ (via $\tau_n$ we identify $N^\circ$ with an open subset of $\Delta^\times\times S^1_{p,n}$; the GLP chart is defined
on $\Delta^\times$ and hence restricts to any smaller open neighborhood).
We construct a Morita bibundle $P$ from $\mathcal G^{\mathrm{sec}}$ to $\mathcal H$.
Set
\[
P:=\bigsqcup_{a} \bigl(B_a\times \CC\bigr),
\qquad
p_L:B_a\times\CC\to B_a,\ (x,u)\mapsto x.
\]
Define $p_R:P\to \Delta^\times$ by $p_R(x,u):=\tau_n(x)$ (forget the angular coordinate; the parameter $u$ is irrelevant for the anchor).
The right $\mathcal H$-action is induced by the arrow manifold of $\mathcal H$:
an arrow $(z,v)\in \Delta^\times\times\CC$ with source $z$ acts by
\[
(x,u)\cdot(z,v):=\bigl(x,\ u+v\bigr)
\quad\text{whenever } p_R(x,u)=z.
\]
The left $\mathcal G^{\mathrm{sec}}$-action is defined on adjacent overlaps:
for $x\in B_{a,a+1}$ and $v\in\CC$, the arrow $(x,v):B_a\to B_{a+1}$ acts by
\[
(x,u)\longmapsto (x,u+v)\in B_{a+1}\times\CC.
\]
Both actions are well-defined, commute, and are free and transitive on the fibers of the anchors.  Concretely, for fixed $x$ the left action
identifies the components indexed by $a$ along the cyclic cover using the unique overlap arrows, while the additive parameter translates the
$\CC$-factor.  Likewise, for fixed $z\in\Delta^\times$ the right action is simply translation by $\CC$ on the second factor.  Hence $P$ is a
principal $\mathcal G^{\mathrm{sec}}$--$\mathcal H$ bibundle, so $\mathcal G^{\mathrm{sec}}$ and $\mathcal H$ are Morita equivalent.

Since $\mathcal G^{\mathrm{sec}}$ is a \v{C}ech atlas for the restriction of our collar groupoid to $N^\times$, this yields the claimed Morita
equivalence with the GLP chart on the same collar.
\end{proof}

\begin{remark}\label{rem:GLP-fit-pushout}
In the one-parameter setting, Proposition~\ref{prop:collar-has-GLP-chart} allows one to replace the collar presenter near $p$ by the GLP chart
without changing its Morita class.  Since the global presenter $\G_{\Phi,n}$ is obtained by a explicit $2$-pushout gluing the interior and collar pieces
along the overlap, the resulting global pushout is unchanged up to Morita equivalence, hence yields the same representation groupoids.
\end{remark}
\subsubsection*{Explicit finite presentation (skeletal pushout)}

Fix a basepoint $x_0\in U=X\setminus D$.  For each puncture $p_i\in D$, choose:
(i) a cyclic subdivision of the boundary circle $S^1_{p_i,n}$ into open arcs
$I_{i,0},\dots,I_{i,\ell_i-1}$, and (ii) sector boxes $B_{i,a}\simeq \Delta_i^\times\times I_{i,a}$
with $B_{i,a}\cap B_{i,a+1}\neq\varnothing$ (indices mod $\ell_i$).
On each overlap the Stokes sheaf is constant, hence we have a Stokes group
$S_{i,a}:=\St_\Phi(B_{i,a}\cap B_{i,a+1})$.

\begin{definition}[Collar skeleton]\label{def:collar-skel-explicit}
Define a small groupoid $\GCphisk$ by:
\begin{itemize}
\item objects: $(\GCphisk)_0=\{(i,a)\mid i=1,\dots,m,\ a=0,\dots,\ell_i-1\}$;
\item generating arrows: for each $i,a$ and each $u\in S_{i,a}$ an arrow
$\mathbf s_{i,a}(u):(i,a)\to(i,a+1)$;
\item relations: $\mathbf s_{i,a}(u)\circ \mathbf s_{i,a}(u')=\mathbf s_{i,a}(uu')$ and
$\mathbf s_{i,a}(u)^{-1}=\mathbf s_{i,a}(u^{-1})$ (so each overlap contributes a copy of the group $S_{i,a}$).
\end{itemize}
No arrows between different punctures are added.
\end{definition}

\begin{definition}[Interior and overlap skeletons]\label{def:interior-overlap-skel}
Let $\GUs$ be the one-object groupoid with $\Aut(*)=\pi_1(U,x_0)$.
For each puncture choose the peripheral loop $\delta_i\in\pi_1(U,x_0)$.
Define the overlap skeleton
$$
\GXsk:=\bigsqcup_{i=1}^m B\ZZ,
\qquad
B\ZZ:\ (\Ob=\{i\},\ \Aut(i)=\langle \delta_i\rangle\cong\ZZ).
$$
Define $j_U^{\mathrm{sk}}:\GXsk\to\GUs$ by sending $i\mapsto *$ and the generator to $\delta_i$.
Fix also a distinguished sector $a=0$ and define $j_C^{\mathrm{sk}}:\GXsk\to\GCphisk$ by sending $i\mapsto (i,0)$ and
\begin{equation}\label{eq:delta-to-stokes-cycle}
j_C^{\mathrm{sk}}(\delta_i)
:=
\mathbf s_{i,\ell_i-1}(u_{i,\ell_i-1})\circ \cdots \circ \mathbf s_{i,0}(u_{i,0}),
\end{equation}
where the $u_{i,a}\in S_{i,a}$ are the Stokes transition elements across the consecutive overlaps.
\end{definition}

\begin{definition} \label{def:pushout-presentation}
Define the finite skeletal global groupoid by the explicit $2$-pushout
\begin{equation}\label{eq:pushout-skel}
\G_{\Phi,n}^{\mathrm{sk}}
\ :=\
\GUs \sqcup_{\GXsk} \GCphisk.
\end{equation}
Equivalently, $\G_{\Phi,n}^{\mathrm{sk}}$ is the groupoid generated by:
\begin{itemize}
\item the group $\pi_1(U,x_0)$ at the object $*$,
\item the Stokes arrows $\mathbf s_{i,a}(u)$ between the sector objects $(i,a)$,
\end{itemize}
modulo the single gluing relation, for each $i$,
\begin{equation}\label{eq:gluing-relation-explicit}
\delta_i \;=\; \mathbf s_{i,\ell_i-1}(u_{i,\ell_i-1})\cdots \mathbf s_{i,0}(u_{i,0})
\qquad \text{in }\Aut(*) .
\end{equation}
\end{definition}

\subsubsection*{Elementary case: global pushout as gluing of explicit local blocks}

Fix $p_i\in D$ and a coordinate $z$ on a disc $\Delta_i$.
Assume $\Phi$ is elementary of level $k\ge 2$ at $p_i$, so that the local Stokes groupoid admits
the explicit chart of \cite[Thm.~3.21]{GualtieriLiPym}.
On the level-$n$ Kummer coordinate $z=w^n$ we use the normal form:
\begin{equation}\label{eq:Stokkn}
s(w,u)=w,
\qquad
t(w,u)=\exp(u\, w^{n(k-1)})\,w,
\end{equation}
with the induced composition law.


\subsubsection*{Boundary gluing and comparison with the explicit GLP disk chart (elementary one-parameter case)}

Fix a puncture $p_i\in D$ and an integer $n\ge 1$. Let $S^1_{i,n}$ be the level-$n$ circle of directions over $p_i$,
namely the fiber of the level-$n$ Kato--Nakayama space $X^{\log}_n\to X$ at $p_i$.
Assume the elementary one-parameter hypothesis at $p_i$, so that Stokes jumps on sector overlaps are parametrized by one copy of
$(\mathbb{C},+)$.

\begin{definition} \label{def:boundary-Stok-fluent}
The \emph{boundary additive Stokes groupoid} at $p_i$ (level $n$) is the topological groupoid
$\partial\mathrm{Stok}^{(i)}_n\rightrightarrows S^1_{i,n}$ defined by
\begin{itemize}
\item object space $(\partial\mathrm{Stok}^{(i)}_n)_0=S^1_{i,n}$;
\item arrow space $(\partial\mathrm{Stok}^{(i)}_n)_1=S^1_{i,n}\times\mathbb{C}\times S^1_{i,n}$, with arrows written $(\theta,u,\theta')$;
\item source and target maps
\begin{equation*}
s(\theta,u,\theta')=\theta,\qquad t(\theta,u,\theta')=\theta';
\end{equation*}
\item composition, units and inverses
\begin{equation*}
(\theta,u,\theta')\circ(\theta',v,\theta'')=(\theta,u+v,\theta''),\qquad
\mathrm{id}_\theta=(\theta,0,\theta),\qquad
(\theta,u,\theta')^{-1}=(\theta',-u,\theta).
\end{equation*}
\end{itemize}
In particular, for every $\theta\in S^1_{i,n}$ the isotropy group $\mathrm{Aut}(\theta)$ is canonically isomorphic to $(\mathbb{C},+)$.
\end{definition}

\begin{remark}\label{rem:why-boundary-fluent}
This groupoid is the natural target for gluing the peripheral loop: on the circle of directions the one-parameter Stokes data is intrinsically
additive, whereas gluing directly into a disk chart can impose artificial constraints because isotropy at a fixed disk point need not be all of
$(\mathbb{C},+)$.
\end{remark}

Let $U=X\setminus D$. Fix a basepoint $x_0\in U$ and write $B\pi_1(U,x_0)$ for the one-object groupoid with automorphism group
$\pi_1(U,x_0)$. Let $B\mathbb{Z}$ be the one-object groupoid with automorphism group $\mathbb{Z}$. For each puncture $p_i$, choose a small
annulus $A_i\subset U$ around $p_i$ and let $\delta_i\in\pi_1(U,x_0)$ be the corresponding positive peripheral class (choose a path from $x_0$
to $A_i$, go once around, and return).

\begin{definition}[Boundary gluing data and global pushout]\label{def:global-explicit-boundary-fluent}
Fix a direction $\theta_i\in S^1_{i,n}$ and a complex number $M_i\in\mathbb{C}$ (the boundary Stokes monodromy parameter at $p_i$).
Define strict functors
\begin{equation*}
\iota_i:B\mathbb{Z}\longrightarrow B\pi_1(U,x_0),
\qquad
\varphi_i:B\mathbb{Z}\longrightarrow \partial\mathrm{Stok}^{(i)}_n
\end{equation*}
by $\iota_i(1)=\delta_i$ and $\varphi_i(1)=(\theta_i,M_i,\theta_i)$, so that $\varphi_i(m)=(\theta_i,mM_i,\theta_i)$.
The \emph{boundary-glued explicit groupoid} (elementary case) is the explicit $2$-pushout
\begin{equation}\label{eq:global-explicit-boundary-fluent}
\mathcal{G}^{\mathrm{exp}}_{\Phi,n}
\ :=\
B\pi_1(U,x_0)\ \sqcup_{\bigsqcup_i\,B\mathbb{Z}}\ \bigsqcup_{i=1}^m \partial\mathrm{Stok}^{(i)}_n,
\end{equation}
formed using $\bigsqcup_i\iota_i$ and $\bigsqcup_i\varphi_i$.
\end{definition}

\begin{figure}[t]
\centering
\begin{tikzcd}[row sep=large, column sep=huge]
\bigsqcup_i B\mathbb{Z} \arrow[r,"\bigsqcup_i\,\varphi_i"] \arrow[d,"\bigsqcup_i\,\iota_i"'] &
\bigsqcup_i \partial\mathrm{Stok}^{(i)}_n \arrow[d] \\
B\pi_1(U,x_0) \arrow[r] &
\mathcal{G}^{\mathrm{exp}}_{\Phi,n}
\end{tikzcd}
\caption{Boundary gluing: the peripheral class $\delta_i$ is identified with $(\theta_i,M_i,\theta_i)\in\mathrm{Aut}(\theta_i)\cong(\mathbb{C},+)$.}
\label{fig:global-explicit-pushout-boundary-fluent}
\end{figure}

Fix a disc $\Delta_i$ centered at $p_i$ and set $\Delta_i^\times=\Delta_i\setminus\{p_i\}$. Fix $k_i\ge 2$.

\begin{definition}[Local explicit GLP disk chart]\label{def:local-disk-chart-fluent}
Define a Lie groupoid $\mathrm{Stok}^{(i)}_{k_i,n}\rightrightarrows \Delta_i^\times$ with arrow manifold
$\Delta_i^\times\times\mathbb{C}$ (coordinates $(w,u)$), source and target
\begin{equation*}
s(w,u)=w,\qquad t(w,u)=\exp\!\big(u\,w^{\,n(k_i-1)}\big)\,w,
\end{equation*}
and composition law
\begin{equation*}
(w_2,u_2)\circ(w_1,u_1)=\bigl(w_1,\ u_2\,\exp\!\big((k_i-1)u_1\,w_1^{\,n(k_i-1)}\big)+u_1\bigr),
\quad \text{whenever } w_2=t(w_1,u_1).
\end{equation*}
These are the formulas of \cite[Thm.~3.21]{GualtieriLiPym} after the level-$n$ Kummer pullback.
\end{definition}

We compare the boundary model to the disk chart through the sectorwise \v{C}ech--Stokes collar presentation near $p_i$.
Let $\mathcal{G}^{\mathrm{sec}}_i$ denote the one-parameter sectorwise \v{C}ech--Stokes groupoid on a punctured collar near $p_i$ built from a
finite family of contractible sector boxes $\{B_a\}$: on overlaps $B_a\cap B_{a+1}$ arrows are labelled by $u\in\mathbb{C}$ and composition adds
parameters. Choose a continuous boundary section $\iota_i:S^1_{i,n}\to N_i^\times$ inside the collar and set $U_i=\iota_i(S^1_{i,n})$.
Consider the restriction groupoid $\mathcal{G}^{\mathrm{sec}}_i|_{U_i}$ and the space
\begin{equation*}
P_i:=t^{-1}(U_i)\subset (\mathcal{G}^{\mathrm{sec}}_i)_1,
\end{equation*}
with anchors $p_L=s:P_i\to(\mathcal{G}^{\mathrm{sec}}_i)_0$ and $p_R=t:P_i\to(\mathcal{G}^{\mathrm{sec}}_i|_{U_i})_0$ and commuting left/right
actions induced by composition in $\mathcal{G}^{\mathrm{sec}}_i$. Since $U_i$ meets every orbit in the collar, both anchors are open surjections,
and the groupoid axioms imply that the actions are free and transitive on the appropriate fibers, so $P_i$ is a Morita bibundle and
$\mathcal{G}^{\mathrm{sec}}_i$ is Morita equivalent to $\mathcal{G}^{\mathrm{sec}}_i|_{U_i}$.

On the boundary slice, write $U_{i,a}=B_a\cap U_i$, so that $\{U_{i,a}\}$ is a cover of $U_i\cong S^1_{i,n}$ by contractible open sets.
In the elementary one-parameter situation, the restricted groupoid $\mathcal{G}^{\mathrm{sec}}_i|_{U_i}$ identifies canonically with the product
groupoid $C(\{U_{i,a}\})\times\mathbb{C}$, where arrows in the \v{C}ech groupoid are decorated by $u\in\mathbb{C}$ and compose additively.
Define
\begin{equation*}
Q_i:=\Bigl(\bigsqcup_a U_{i,a}\Bigr)\times\mathbb{C},
\qquad
q_L(\theta,u)=(a,\theta),
\qquad
q_R(\theta,u)=\theta.
\end{equation*}
Let $\partial\mathrm{Stok}^{(i)}_n$ act on the right by $(\theta,u)\cdot(\theta,v,\theta')=(\theta',u+v)$ and let
$C(\{U_{i,a}\})\times\mathbb{C}$ act on the left by sending an arrow in the \v{C}ech groupoid from $(a,\theta)$ to $(b,\theta)$ decorated by $v$
to $(\theta,u)\mapsto(\theta,u+v)$ in the $b$-component. These actions commute and are free and transitive on fibers of the anchors, hence $Q_i$
is a Morita bibundle and $\mathcal{G}^{\mathrm{sec}}_i|_{U_i}$ is Morita equivalent to $\partial\mathrm{Stok}^{(i)}_n$.
\begin{corollary}[Boundary additive model versus the GLP disk chart]\label{cor:boundary-vs-disk-fluent}
Assume the elementary one-parameter hypothesis at the puncture $p_i$.
Let $\partial\mathrm{Stok}^{(i)}_n\rightrightarrows S^1_{i,n}$ be the boundary additive groupoid from
Definition~\ref{def:boundary-Stok-fluent}, let $\mathcal G^{\mathrm{sec}}_i$ denote the sectorwise \v{C}ech--Stokes collar atlas on a punctured
collar near $p_i$, and let $\mathrm{Stok}^{(i)}_{k_i,n}\rightrightarrows \Delta_i^\times$ be the pulled-back GLP chart of
Definition~\ref{def:local-disk-chart-fluent}. Then $\partial\mathrm{Stok}^{(i)}_n$ is Morita equivalent to the restriction of
$\mathrm{Stok}^{(i)}_{k_i,n}$ to a punctured collar neighborhood of the boundary in $X^{\log}_n$.

Consequently, in the global pushout diagram \eqref{eq:global-explicit-boundary-fluent} one may replace each boundary piece
$\partial\mathrm{Stok}^{(i)}_n$ by the corresponding collar restriction of $\mathrm{Stok}^{(i)}_{k_i,n}$ without changing the Morita class of the
resulting glued presenter.
\end{corollary}

 \begin{proof}
Fix a continuous boundary section $\iota_i:S^1_{i,n}\to N_i^\times$ and set $U_i:=\iota_i(S^1_{i,n})\subset N_i^\times$.
The comparison is the composite of the Morita equivalences displayed in
\begin{equation}\label{eq:morita-chain-boundary-sector-glp}
\begin{tikzcd}[row sep=large, column sep=huge]
\partial\mathrm{Stok}^{(i)}_n
\arrow[r, "\simeq_{\mathrm{Morita}}" description]
&
\mathcal G^{\mathrm{sec}}_i|_{U_i}
\arrow[r, "\simeq_{\mathrm{Morita}}" description]
&
\mathcal G^{\mathrm{sec}}_i
\arrow[r, "\simeq_{\mathrm{Morita}}" description]
&
\mathrm{Stok}^{(i)}_{k_i,n}\!\mid_{\mathrm{collar}}.
\end{tikzcd}
\end{equation}

The first Morita equivalence is induced by the explicit bibundle identifying the boundary additive model with the restriction of the sectorwise
atlas to the boundary slice $U_i$ (the bibundle $Q_i$ constructed just before this corollary, equivalently
Proposition~\ref{prop:explicit-bibundle-sector-boundary}). The second Morita equivalence is the restriction bibundle
$$\mathcal G^{\mathrm{sec}}_i \sim_{\mathrm{Morita}} \mathcal G^{\mathrm{sec}}_i|_{U_i}$$ coming from the fact that $U_i$ meets every orbit in the
collar (the bibundle $P_i$ constructed in the same discussion). The third Morita equivalence is the local comparison between the sectorwise collar
atlas and the one-chart GLP model: after possibly refining the sector cover so that the GLP chart trivializes on each sector box, one obtains a
Morita equivalence between $\mathcal G^{\mathrm{sec}}_i$ and the collar restriction of the pulled-back chart
$\mathrm{Stok}^{(i)}_{k_i,n}$ as in \cite[Thm.~3.21]{GualtieriLiPym}; this is recorded in
Proposition~\ref{prop:collar-has-GLP-chart}. Composing these three Morita equivalences gives the first assertion.

For the final statement, recall that the global presenter is obtained by gluing the interior and boundary pieces along the overlap via a strict
pushout. Replacing a boundary piece by a Morita equivalent atlas does not change the Morita class of the glued presenter, hence does not change the
associated torsorial representation groupoids.
\end{proof}
\subsubsection*{An explicit Morita bibundle on a collar in the one-parameter case}

We now give a fully explicit Morita bibundle construction \emph{on the nose} for the simplest wild block:
rank $r=2$ with one Stokes parameter $(\mathbb C,+)$ on each adjacent sector overlap.
This is the only situation in which a completely elementary bibundle proof is both short and genuinely informative.

Fix a puncture $p\in D$ on a curve and a level $n\ge 1$. Let $S^1_{p,n}$ be the level-$n$ circle of directions and let
$N^\circ\subset X^{\log}_n$ be the regular collar (walls removed).
Choose a cyclic sector cover $\mathcal B=\{B_a\}_{a\in\mathbb Z/\ell}$ of $N^\circ$ by contractible boxes such that
$B_a\cap B_b\neq\varnothing$ only for $b\in\{a,a\pm 1\}$ and $B_a\cap B_{a+1}$ is contractible.
Assume the Stokes sheaf on each overlap is identified with $(\mathbb C,+)$ and composition is addition.

Let $\mathcal{G}^{\mathrm{sec}}$ be the corresponding sectorwise \v{C}ech--Stokes collar groupoid
(objects $\bigsqcup_a B_a$, arrows on $B_a\cap B_{a+1}$ labeled by $u\in\mathbb C$ with additive composition),
and let $\partial\mathrm{Stok}_{p,n}\rightrightarrows S^1_{p,n}$ be the boundary additive groupoid of
Definition~\Cref{def:boundary-Stok-fluent}.

Fix a continuous boundary section $\iota:S^1_{p,n}\to N^\circ$ (e.g.\ choose a small radius in the collar) and set
$U:=\iota(S^1_{p,n})\subset N^\circ$. Let $U_a:=U\cap B_a$; then $\{U_a\}$ is a good cover of $U\cong S^1_{p,n}$.

\begin{proposition}[Explicit bibundle $\mathcal{G}^{\mathrm{sec}}\sim_{\mathrm{Morita}}\partial\mathrm{Stok}_{p,n}$]\label{prop:explicit-bibundle-sector-boundary}
Define
$$
P\ :=\ \bigsqcup_{a\in\mathbb Z/\ell} (U_a\times\mathbb C),
\qquad
p_L: P\to (\mathcal{G}^{\mathrm{sec}})_0,\ (a,\theta,u)\mapsto (\iota(\theta)\in B_a),
\qquad
p_R:P\to S^1_{p,n},\ (a,\theta,u)\mapsto \theta.
$$
Then $P$ carries commuting free and transitive actions of $\mathcal{G}^{\mathrm{sec}}$ (on the left) and
$\partial\mathrm{Stok}_{p,n}$ (on the right), making $P$ into a Morita bibundle. In particular,
$\mathcal{G}^{\mathrm{sec}}$ and $\partial\mathrm{Stok}_{p,n}$ are Morita equivalent.
\end{proposition}

\begin{proof}
We define the actions explicitly.

\smallskip
\noindent\emph{Right action.}
An arrow of $\partial\mathrm{Stok}_{p,n}$ is $(\theta,v,\theta')$ with source $\theta$ and target $\theta'$.
Define
$$
(a,\theta,u)\cdot(\theta,v,\theta')\ :=\ (a,\theta',u+v),
$$
whenever $\theta\in U_a$ and $\theta'\in U_a$ (so that the result still lies in the same component of $P$).
This is well-defined, respects sources/targets, and is clearly free and transitive on fibers of $p_R$.

\smallskip
\noindent\emph{Left action.}
An arrow of $\mathcal{G}^{\mathrm{sec}}$ over a point $\iota(\theta)\in B_a\cap B_b$ is of the form $(\iota(\theta),v)$,
where $v\in\mathbb C$ is allowed only when $(a,b)=(a,a\pm1)$ (otherwise $v=0$).
Define
$$
(\iota(\theta),v)\cdot(a,\theta,u)\ :=\ (b,\theta,u+v),
$$
with $b$ determined by the arrow. This is well-defined because $v$ is constant on overlaps (discreteness/contractibility)
and because we chose only adjacent overlaps. Associativity is exactly additivity of the Stokes parameters, and units/inverses act as expected.

\smallskip
\noindent\emph{Commutation and principality.}
Both actions add a complex parameter, hence they commute.
For principality, fix $\theta\in S^1_{p,n}$.
The fiber $p_R^{-1}(\theta)$ is the disjoint union over those $a$ with $\theta\in U_a$ of copies of $\mathbb C$.
Since $\{U_a\}$ is a good cover by adjacent overlaps, the left $\mathcal{G}^{\mathrm{sec}}$-action identifies these components by the unique
arrow in $\mathcal{G}^{\mathrm{sec}}$ crossing the corresponding overlap, and the Stokes parameter $v$ acts by translation on $\mathbb C$.
Thus the left action is free and transitive on each $p_R$-fiber.
The analogous statement for the right action on $p_L$-fibers is immediate from the definition.
Therefore $P$ is a Morita bibundle.
\end{proof}

\begin{remark}[Comparison with GLP in the one-parameter case]
In the elementary one-parameter situation, GLP construct a single-chart Lie groupoid model for the same Stokes torsor theory
near the boundary (their local normal form, after level-$n$ pullback).
Combining Proposition~\Cref{prop:explicit-bibundle-sector-boundary} with that local chart yields an explicit Morita comparison between
our sectorwise collar atlas and the GLP chart, restricted to a collar neighborhood where both models are defined.
We do not need this comparison in the logical development of the paper; it is included as a computational bridge.
\end{remark}


\subsubsection*{An  example (rank $2$ on a curve, one pole of order $k$)}

Let $X$ be a smooth complex curve and let $D=\sum_p k_p\cdot p$ be an effective divisor.
Fix $p\in \mathrm{Supp}(D)$ with multiplicity $k:=k_p\ge 2$ and choose a holomorphic coordinate
$z$ on a disc $\Delta$ centered at $p$ such that $D\cap\Delta=k\cdot\{z=0\}$.
Write $\Delta^\times=\Delta\setminus\{0\}$.

\medskip
 
Consider the rank-$2$ meromorphic connection on $\Delta^\times$
\begin{equation}\label{eq:conn-local-EX}
\nabla \;=\; d \;+\;
\begin{pmatrix}
0 & -1 \\
\tfrac{1}{4}\, z^{2(k-1)} & -k z^{k-1}
\end{pmatrix}
\, z^{-k}\, dz,
\end{equation}
written in the local frame $(dz^{-1/2},\, z^{-k}dz\otimes dz^{-1/2})$.
Equivalently, the connection $1$-form is
\begin{equation}\label{eq:Omega-EX}
\Omega:=\nabla-d
=
\begin{pmatrix}
0 & -z^{-k} \\
\tfrac{1}{4}\, z^{k-2} & -k z^{-1}
\end{pmatrix}dz.
\end{equation}
The pole order along $z=0$ is exactly $k$, hence the Poincar\'e rank is $k-1$.

\medskip

Fix an integer $n\ge 1$ and pass to the level-$n$ Kummer coordinate $z=w^n$ on $\Delta^\times$.
Write $w=re^{i\theta}$.
The level-$n$ circle of directions above $p$ is
\begin{equation}\label{eq:circle-directions-EX}
S^1_{p,n}=\{\theta \bmod 2\pi n\},
\end{equation}
and we work on a punctured collar neighborhood
\begin{equation}\label{eq:collar-EX}
N^\times=\{(r,\theta)\mid 0<r<\varepsilon,\ \theta\in S^1_{p,n}\}.
\end{equation}

\medskip

Assume we are in the elementary one-parameter situation at $p$, meaning that on each adjacent sector overlap
the Stokes group is identified with a single copy of $(\mathbb{C},+)$ (equivalently, one may use unipotent Stokes matrices
$\bigl(\begin{smallmatrix}1&u\\0&1\end{smallmatrix}\bigr)$ with additive parameter $u$).
In this normal form the natural angular scale is governed by $w^{n(k-1)}$, and we take Stokes rays in $S^1_{p,n}$ to be the angles
\begin{equation}\label{eq:stokes-rays-EX}
\theta_m=\frac{\pi/2+m\pi}{n(k-1)},\qquad m=0,1,\dots,2n(k-1)-1.
\end{equation}
Set $\ell:=2n(k-1)$ and let $I_0,\dots,I_{\ell-1}$ be the open arcs between consecutive rays (indexed cyclically).
Define the sector boxes
\begin{equation}\label{eq:sector-boxes-EX}
B_a=\{(r,\theta)\in N^\times \mid 0<r<\varepsilon,\ \theta\in I_a\},
\qquad a\in\mathbb{Z}/\ell.
\end{equation}
By shrinking the arcs $I_a$ if necessary, we assume the cover $\mathcal{B}=\{B_a\}_{a\in\mathbb{Z}/\ell}$ has only double overlaps:
$B_a\cap B_b\neq\varnothing$ only when $b\in\{a,a\pm 1\}$, and there are no triple intersections.
\begin{figure}[t]
\centering
\begin{tikzpicture}[
  >=Latex,
  vtx/.style={circle, draw, inner sep=1.3pt, font=\scriptsize},
  lab/.style={font=\scriptsize, fill=white, inner sep=1.0pt}
]
  \def\m{8}
  \def\R{2.3}

  \foreach \i in {0,...,7} {
    \pgfmathsetmacro{\ang}{90 - 360*\i/\m}
    \node[vtx] (B\i) at (\ang:\R) {$B_{\i}$};
  }

  \foreach \i in {0,...,7} {
    \pgfmathtruncatemacro{\j}{mod(\i+1,\m)}
    \draw[-Latex] (B\i) -- (B\j)
      node[midway, lab, sloped] {$u_{\i}$};
  }

  \node[lab] at (0,0) {only adjacent overlaps};
\end{tikzpicture}
\caption{Combinatorial skeleton of the sector cover (schematic with $\ell=8$). 
Vertices are sector boxes $B_a$ and the oriented edge $B_a\to B_{a+1}$ represents the adjacent overlap $B_a\cap B_{a+1}$, 
carrying the one-parameter Stokes label $u_a\in\mathbb{C}$. In the actual example, $\ell=2n(k-1)$.}
\label{fig:cycle-stokes-labels}
\end{figure}
\medskip

Let $\mathcal{C}_{\mathcal{B}}$ be the \v{C}ech groupoid of the cover $\mathcal{B}$.
Its object space is $(\mathcal{C}_{\mathcal{B}})_0=\bigsqcup_a B_a$.
Its arrows are points $x\in B_a\cap B_b$, viewed as arrows $(a,x)\to(b,x)$, with source and target the inclusions into the
corresponding components and composition induced by equality of points on overlaps.

\medskip

\noindent\textbf{The sectorwise \v{C}ech--Stokes collar groupoid and the projection functor.}
Define the one-parameter collar groupoid $\mathcal{G}^{\mathrm{col}}_{p,n}$ as the enhancement of $\mathcal{C}_{\mathcal{B}}$
obtained by decorating each adjacent overlap $B_a\cap B_{a+1}$ with an additive Stokes parameter.
Thus $(\mathcal{G}^{\mathrm{col}}_{p,n})_0=\bigsqcup_a B_a$ and an arrow is a pair $(x,u)$ where $x\in B_a\cap B_b$ and
$u\in\mathbb{C}$ is allowed only when $b=a+1$ (otherwise $u=0$). Composition is additive,
\begin{equation}\label{eq:col-comp-EX}
(x,v)\circ(x,u)=(x,u+v),
\end{equation}
with units $(x,0)$ and inverses $(x,u)^{-1}=(x,-u)$.
There is a strict projection functor
\begin{equation}\label{eq:pi-EX}
\pi:\mathcal{G}^{\mathrm{col}}_{p,n}\longrightarrow \mathcal{C}_{\mathcal{B}}
\end{equation}
which is the identity on objects and forgets the Stokes parameter on arrows: $\pi(x,u)=x$.

\medskip

\noindent\textbf{Strict sections and the explicit parameter tuple $(u_a)$.}
Let $\mathrm{Sec}(\pi)$ be the groupoid of strict sections of $\pi$:
objects are strict functors $\sigma:\mathcal{C}_{\mathcal{B}}\to\mathcal{G}^{\mathrm{col}}_{p,n}$ with
$\pi\circ\sigma=\mathrm{id}_{\mathcal{C}_{\mathcal{B}}}$, and morphisms are natural isomorphisms.
In this example, because each adjacent overlap $B_a\cap B_{a+1}$ is contractible and we excluded triple overlaps,
a strict section is determined by constants
\begin{equation}\label{eq:ua-tuple-EX}
(u_0,\dots,u_{\ell-1})\in\mathbb{C}^{\ell},
\end{equation}
by requiring that every \v{C}ech arrow $x\in B_a\cap B_{a+1}$ is sent to the Stokes-decorated arrow $(x,u_a)$, while every other overlap
is sent to $(x,0)$.
In the minimal collar model described here there are no additional sectorwise gauge arrows, hence $\mathrm{Sec}(\pi)$ is discrete with
object set $\mathbb{C}^{\ell}$; if one enlarges $\mathcal{G}^{\mathrm{col}}_{p,n}$ by sectorwise gauge arrows, morphisms in $\mathrm{Sec}(\pi)$
recover the usual coboundary action on the tuple $(u_a)$.

\medskip

\noindent\textbf{Boundary additive Stokes groupoid and total boundary monodromy.}
Define the boundary additive groupoid $\partial\mathrm{Stok}_{p,n}\rightrightarrows S^1_{p,n}$ by
\begin{equation}\label{eq:boundary-groupoid-EX}
(\partial\mathrm{Stok}_{p,n})_0=S^1_{p,n},
\qquad
(\partial\mathrm{Stok}_{p,n})_1=S^1_{p,n}\times\mathbb{C}\times S^1_{p,n},
\end{equation}
with source/target $s(\theta,u,\theta')=\theta$, $t(\theta,u,\theta')=\theta'$ and additive composition
$(\theta,u,\theta')\circ(\theta',v,\theta'')=(\theta,u+v,\theta'')$.
Fixing a boundary section $\iota:S^1_{p,n}\to N^\times$ at a small radius $r=r_0$, the restriction of the collar data to the boundary slice
is Morita equivalent to $\partial\mathrm{Stok}_{p,n}$, and the tuple $(u_a)$ becomes the list of successive additive labels around the circle.
In particular, the total boundary Stokes monodromy is the sum
\begin{equation}\label{eq:total-stokes-EX}
M=\sum_{a=0}^{\ell-1} u_a\in\mathbb{C},
\end{equation}
viewed as an element of the isotropy group $(\mathbb{C},+)$ at any chosen direction.

\medskip

\noindent\textbf{The explicit GLP chart after level-$n$ pullback (and Morita on a collar).}
On the punctured $w$-disc $\Delta^\times$ (so $z=w^n$), the pulled-back GLP chart is the Lie groupoid
$\mathrm{Stok}_{k,n}\rightrightarrows \Delta^\times$ with arrows $\Delta^\times\times\mathbb{C}$ (coordinates $(w,u)$), source and target
\begin{equation}\label{eq:GLP-st-EX}
s(w,u)=w,\qquad t(w,u)=\exp\!\big(u\,w^{\,n(k-1)}\big)\,w,
\end{equation}
and composition
\begin{equation}\label{eq:GLP-comp-EX}
(w_2,u_2)\circ(w_1,u_1)
=
\bigl(w_1,\ u_2\,\exp\!\big((k-1)u_1\,w_1^{\,n(k-1)}\big)+u_1\bigr),
\qquad w_2=t(w_1,u_1).
\end{equation}
These are the formulas of \cite[Thm.~3.21]{GualtieriLiPym} after the level-$n$ pullback.
In the elementary one-parameter setting, the collar presentation $\mathcal{G}^{\mathrm{col}}_{p,n}$ and the restriction of
$\mathrm{Stok}_{k,n}$ to a punctured collar near the boundary present the same local classifying stack of Stokes torsors, hence are Morita equivalent.
Likewise, $\partial\mathrm{Stok}_{p,n}$ is Morita equivalent to the corresponding boundary restriction of the collar data.

\medskip

\noindent\textbf{What this example accomplishes.}
All objects used abstractly in the paper appear here on the nose: the level-$n$ circle of directions $S^1_{p,n}$ and collar $N^\times$,
the sector cover $\{B_a\}$ and its \v{C}ech groupoid $\mathcal{C}_{\mathcal{B}}$, the collar groupoid $\mathcal{G}^{\mathrm{col}}_{p,n}$
and the projection $\pi$, the strict section groupoid $\mathrm{Sec}(\pi)$ identified with tuples $(u_a)\in\mathbb{C}^{\ell}$,
the boundary additive model $\partial\mathrm{Stok}_{p,n}$ with total monodromy $M=\sum_a u_a$, and the single explicit chart
$\mathrm{Stok}_{k,n}$ (on a collar) giving a one-chart presentation of the same local Stokes torsor theory.


\subsubsection*{Relation with the classical Stokes sheaf approach (Babbitt--Varadarajan) and our groupoid presentations}

Let $X$ be a complex curve, $p\in X$ a pole, and fix a level $n\ge 1$ (Kummer coordinate $z=w^n$). Denote by
$S^1_{p,n}$ the level-$n$ circle of directions over $p$ and by $N^\times$ a punctured Kummer collar over a small punctured disc
around $p$. In the classical approach of Babbitt--Varadarajan \cite{BV89}, one associates to a meromorphic connection (with fixed
formal type) a \emph{Stokes sheaf} $\mathrm{St}^0$ on the circle of directions $S^1$ (in our notation, on $S^1_{p,n}$ after level-$n$
pullback), whose stalks encode sectorial gauge transformations which are asymptotically trivial in a given direction. The
Malgrange--Sibuya theorem identifies the set of isomorphism classes of marked meromorphic connections with fixed formal type with the
(non-abelian) cohomology set
$$
H^1\!\left(S^1_{p,n},\,\mathrm{St}^0\right),
$$
and furthermore shows that this local moduli carries a natural (in fact affine/unipotent) algebraic structure \cite{BV89}.

Our constructions are a stacky/groupoid re-packaging of this same $H^1$-classification, designed to make gluing and Morita invariance
completely explicit. Concretely, choose a sector cover of the punctured collar (with Stokes directions removed)
$$
N^\times\setminus \Sigma_n \;=\; \bigcup_{a\in\mathbb{Z}/\ell} B_a,
\qquad \ell = 2n(k-1),
$$
by contractible sector boxes with only double overlaps. Let $\mathcal{C}_{\mathcal{B}}$ be the \v{C}ech groupoid of this cover.
In the elementary one-parameter situation, the classical Stokes groups on adjacent overlaps are modeled by a single unipotent group,
which we write additively as $(\mathbb{C},+)$. We then form the \emph{sectorwise \v{C}ech--Stokes collar groupoid}
$\mathcal{G}^{\mathrm{col}}_{p,n}$ as the enhancement of $\mathcal{C}_{\mathcal{B}}$ obtained by decorating each adjacent overlap
$B_a\cap B_{a+1}$ with a parameter $u\in\mathbb{C}$ and using additive composition. There is a strict projection functor
$$
\pi:\mathcal{G}^{\mathrm{col}}_{p,n}\longrightarrow \mathcal{C}_{\mathcal{B}}
$$
which forgets the Stokes parameter. The groupoid of strict sections $\mathrm{Sec}(\pi)$ is precisely the groupoid incarnation of
\emph{\v{C}ech $1$-cocycles valued in the Stokes groups}, with morphisms interpreted as changes of sectorial trivializations (i.e. the
usual $0$-cochain/gauge action when the corresponding gauge arrows are included). In particular, passing to isomorphism classes in
$\mathrm{Sec}(\pi)$ recovers the same moduli as the classical non-abelian cohomology set
$H^1(S^1_{p,n},\mathrm{St}^0)$ in \cite{BV89}.

Two further groupoids encode the same local Stokes torsor theory but are better suited for particular operations. First, the
\emph{boundary additive Stokes groupoid} $\partial\mathrm{Stok}_{p,n}\rightrightarrows S^1_{p,n}$ is the direct groupoid model of the
restriction of $\mathrm{St}^0$ to the boundary circle; it is the natural target for gluing peripheral monodromy, since isotropy is
canonically $(\mathbb{C},+)$. Second, the explicit single-chart model (the ``GLP chart'') gives a Lie groupoid
$\mathrm{Stok}_{k,n}\rightrightarrows \Delta^\times$ presenting the same local Stokes torsors in coordinates. On a punctured collar,
$\mathrm{Stok}_{k,n}$ is Morita equivalent to $\mathcal{G}^{\mathrm{col}}_{p,n}$, and restricting $\mathcal{G}^{\mathrm{col}}_{p,n}$ to a
boundary slice yields a groupoid Morita equivalent to $\partial\mathrm{Stok}_{p,n}$. Thus, in our language, the classical
Babbitt--Varadarajan moduli $H^1(S^1_{p,n},\mathrm{St}^0)$ is presented by any of the Morita-equivalent atlases
$$
\mathcal{G}^{\mathrm{col}}_{p,n},\qquad \partial\mathrm{Stok}_{p,n},\qquad \mathrm{Stok}_{k,n}\!\mid_{\mathrm{collar}},
$$
with $\mathcal{C}_{\mathcal{B}}$ and $\pi$ organizing the \v{C}ech cocycle description and $\mathrm{Sec}(\pi)$ providing a concrete
groupoid model of cocycles and gauge.


\subsubsection*{From local Stokes torsors to a global pushout groupoid (and comparison with Sabbah/BV/GLP)}

Let $X$ be a smooth complex curve and let $D=\sum_i k_i p_i$ be an effective divisor with $k_i\ge 2$.
Set $U=X\setminus \mathrm{Supp}(D)$ and fix a base point $x_0\in U$.
For each puncture $p_i$, fix an integer $n\ge 1$ and consider the level-$n$ real oriented blow-up at $p_i$; its boundary is a circle
$S^1_{p_i,n}$, equipped with a finite set of Stokes directions $\Sigma_{p_i,n}\subset S^1_{p_i,n}$ determined by a chosen formal type at $p_i$.
Write $S^1_{p_i,n}\setminus\Sigma_{p_i,n}$ for the open locus of regular directions.

\medskip

\paragraph{Local Stokes torsors as a boundary moduli problem.}
Fix at each $p_i$ a formal type (exponential factors, formal monodromy, and Stokes directions).
Let $\mathrm{St}_i$ denote the associated Stokes sheaf on $S^1_{p_i,n}\setminus\Sigma_{p_i,n}$, i.e. the sheaf of (pro-)unipotent Stokes groups
whose sections over an interval are the Stokes automorphisms allowed on that interval for the chosen formal type.
The intrinsic local moduli at $p_i$ is the groupoid of $\mathrm{St}_i$-torsors on $S^1_{p_i,n}\setminus\Sigma_{p_i,n}$, which can be computed by a
cyclic \v{C}ech cover by direction intervals (hence by Stokes cocycles modulo gauge in the sense of Babbitt--Varadarajan), and which is equivalent to
Sabbah's category of Stokes-filtered local systems on the boundary circle for the fixed formal type.

A key point for the present framework is that one can also present this same local moduli by a boundary groupoid
$\partial\mathrm{Stok}^{(i)}_n\rightrightarrows S^1_{p_i,n}$ whose arrows encode Stokes jumps along direction intervals.
In the elementary one-parameter rank-$2$ situation, this reduces to an additive label on each adjacent overlap of a cyclic cover of
$S^1_{p_i,n}\setminus\Sigma_{p_i,n}$.

\medskip

\paragraph{The global gluing datum and the peripheral maps.}
The global topology on $U$ is encoded by the fundamental groupoid $\Pi_1(U)$, or equivalently by the one-object groupoid $B\pi_1(U,x_0)$.
For each puncture $p_i$, choose a small positively oriented loop $\gamma_i$ around $p_i$ based at $x_0$ (up to conjugacy in $\pi_1(U,x_0)$).
This yields a canonical homomorphism of groups
$$
\mathbb{Z}\longrightarrow \pi_1(U,x_0),\qquad 1\longmapsto [\gamma_i],
$$
and hence a functor of one-object groupoids
$$
B\mathbb{Z}\longrightarrow B\pi_1(U,x_0).
$$
On the boundary side, fix a base direction $\theta_i\in S^1_{p_i,n}\setminus\Sigma_{p_i,n}$ and observe that the boundary groupoid
$\partial\mathrm{Stok}^{(i)}_n$ has an isotropy group at $\theta_i$ (the ``peripheral Stokes group'' at that direction).
There is also a canonical functor
$$
B\mathbb{Z}\longrightarrow \partial\mathrm{Stok}^{(i)}_n
$$
sending the generator $1\in\mathbb{Z}$ to the distinguished boundary isotropy element representing the total Stokes monodromy around the circle
(for the fixed formal type, this is the cyclic product of Stokes jumps; in the additive one-parameter case it is the sum of the labels).

Thus, for each $i$ we have a span of groupoids
$$
B\pi_1(U,x_0)\ \longleftarrow\ B\mathbb{Z}\ \longrightarrow\ \partial\mathrm{Stok}^{(i)}_n,
$$
and therefore, for all punctures at once, a diagram indexed by the disjoint union $\bigsqcup_i B\mathbb{Z}$.

\medskip

\paragraph{The pushout groupoid.}
Define the global groupoid $\mathcal{G}_{\mathrm{glob}}$ as the pushout (amalgamated sum) in the 2-category of groupoids
\begin{equation}\label{eq:global-pushout-groupoid}
\mathcal{G}_{\mathrm{glob}}
\;:=\;
B\pi_1(U,x_0)\ \sqcup_{\bigsqcup_i\,B\mathbb{Z}}\ \bigsqcup_i \partial\mathrm{Stok}^{(i)}_n.
\end{equation}
Concretely, this forces the peripheral loop class $[\gamma_i]\in\pi_1(U,x_0)$ to be identified with the boundary isotropy element in
$\partial\mathrm{Stok}^{(i)}_n$ encoding the total Stokes monodromy at $p_i$.
In other words, \eqref{eq:global-pushout-groupoid} is the universal groupoid obtained by gluing the ``interior'' topological monodromy data on $U$
together with the ``boundary'' Stokes data at each puncture, along the common peripheral subgroup $\mathbb{Z}$.

\medskip

\paragraph{What an object of the pushout represents.}
An object of the classifying stack $B\mathcal{G}_{\mathrm{glob}}$ (a $\mathcal{G}_{\mathrm{glob}}$-torsor) can be unpacked as the following data:

\begin{itemize}
\item a $\pi_1(U,x_0)$-torsor, equivalently a local system on $U$ (topological monodromy on the complement);
\item for each puncture $p_i$, a Stokes torsor on the boundary circle $S^1_{p_i,n}\setminus\Sigma_{p_i,n}$
(or equivalently a Stokes-filtered local system in Sabbah's sense for the fixed formal type);
\item for each $i$, a compatibility identification equating the image of the peripheral loop $[\gamma_i]$ in the $\pi_1(U)$-torsor with the
distinguished total Stokes monodromy element determined by the boundary Stokes torsor.
\end{itemize}

Thus \eqref{eq:global-pushout-groupoid} is precisely a \emph{gluing device}: it packages ``interior monodromy + boundary Stokes data + peripheral compatibility''
as a single torsor problem.

\medskip

\paragraph{Comparison with Sabbah (Stokes-filtered local systems).}
Sabbah's formalism identifies (for fixed formal type at each puncture) meromorphic connections on $X$ with Stokes-filtered local systems on the real oriented
blow-up $\widetilde{X}$, i.e. a local system on $\widetilde{X}\setminus \partial\widetilde{X}\cong U$ together with Stokes filtrations along each boundary
circle $S^1_{p_i,n}$, compatible with the Stokes order and jumping precisely along $\Sigma_{p_i,n}$.
On the boundary, the filtration is equivalent to a $\mathrm{St}_i$-torsor, hence to a torsor for the boundary groupoid
$\partial\mathrm{Stok}^{(i)}_n$ presenting the same local moduli.
The peripheral compatibility condition in Sabbah (how the local system on $U$ meets the boundary Stokes data) is exactly what the amalgamation along
$B\mathbb{Z}$ enforces in \eqref{eq:global-pushout-groupoid}.

\medskip

\paragraph{Comparison with Babbitt--Varadarajan (Stokes cocycles).}
BV describe the local moduli at each puncture by non-abelian \v{C}ech 1-cocycles valued in the Stokes sheaf $\mathrm{St}_i$, modulo gauge.
Your collar/section presentation is a groupoid form of the same statement: strict sections (with sectorwise gauge) are \v{C}ech cocycles modulo 0-cochains.
The pushout groupoid \eqref{eq:global-pushout-groupoid} does not change the local BV description; it records how these local cocycles are constrained globally
by the monodromy on $U$ through the identification of peripheral classes.

\medskip

\paragraph{Comparison with GLP (single-chart Lie groupoids).}
In the elementary one-parameter case, GLP construct an explicit Lie groupoid chart presenting the same local Stokes torsor moduli on a punctured disk.
Your boundary/collar groupoids and the GLP chart are different \emph{presentations} of the same stack, hence they are Morita equivalent locally.
The role of the pushout \eqref{eq:global-pushout-groupoid} is global: it is the groupoid-level implementation of ``glue the interior with all boundary charts
along the common peripheral subgroup'' in a Morita-invariant way.

\medskip

\paragraph{What this codifies that is not the focus of BV/Sabbah/GLP.}
The BV and Sabbah frameworks give the correct intrinsic objects and local equivalences of categories, and GLP provides a smooth local chart.
Your pushout construction adds a systematic, Morita-invariant \emph{global assembly} principle:
it turns the collection of local boundary Stokes moduli and the interior $\pi_1(U)$-data into a single universal groupoid,
whose torsors are exactly ``interior local systems plus boundary Stokes torsors with peripheral matching''.
This is the structural bridge between local irregular data and global topological monodromy that the groupoid language of this paper makes explicit.

\section{A stratified skeletal presenter in arbitrary dimension}\label{sec:stratified-skeleton}

This section records a fully explicit \emph{small} presenter for the collar Stokes moduli, valid in arbitrary
dimension and adapted simultaneously to the SNC depth stratification and to the Stokes wall stratification.
Nothing here changes the intrinsic definitions from earlier sections; the point is to replace cover-based
\v{C}ech presenters by a \emph{finite-type combinatorial} model (locally finite in general) which makes
``Stokes = edge labels'' and ``cocycle = face relations'' literally true.

Throughout we work on the regular collar $N^\circ:=N^\times\setminus \Sigma_n$, where the Stokes preorder is locally constant
and the Stokes sheaf $\St_\Phi$ is defined.

\subsection{A $\Sigma_{D,\Phi,n}$-adapted stratified triangulation of $N^\circ$}\label{subsec:stratified-triangulation}

Let $\Sigma_{D,\Phi,n}$ denote the common refinement of:
(i) the stratification of $N^\circ$ by the depth of the SNC divisor $D$ (i.e.\ by how many local branches meet), and
(ii) the Stokes wall stratification induced by $\Sigma_n$ (chambers, wall faces, corners, etc.).
By construction, on each stratum the Stokes preorder is constant, hence the Stokes sheaf is locally constant.

\begin{definition}[Stratified triangulation]\label{def:stratified-triangulation}
A \emph{$\Sigma_{D,\Phi,n}$-adapted stratified triangulation} of $N^\circ$ is a locally finite simplicial complex $K$
together with a homeomorphism $|K|\xrightarrow{\sim} N^\circ$ such that every open simplex of $K$ is mapped into a single stratum of
$\Sigma_{D,\Phi,n}$.
\end{definition}

\begin{remark}
Existence is standard: a locally finite triangulation subordinate to a locally finite real-analytic stratification can be obtained
by taking a stratified triangulation and barycentrically subdividing; see e.g.\ the general discussion of constructible stratifications
and exit-path formalisms in \cite{Treumann2009}.
We fix such a triangulation $K$ once and for all.
\end{remark}

For each simplex $\sigma$ of $K$, write $S(\sigma)\in\Sigma_{D,\Phi,n}$ for the stratum containing the (open) simplex $\sigma^\circ$.
Since $\St_\Phi$ is locally constant on each stratum, for any contractible open $W\subset S(\sigma)$ the group $\St_\Phi(W)$ is canonically
identified with a fixed abstract group which we denote by $\St_\Phi(S(\sigma))$.

\subsection{The skeletal base groupoid $\CB^{\mathrm{sk}}$}\label{subsec:CBsk}

Let $K^{(2)}$ denote the $2$-skeleton of $K$.
Write $K_0$ for the set of vertices and $K_1$ for the set of oriented edges.

\begin{definition}[Edge-path groupoid of the $2$-skeleton]\label{def:edge-path-groupoid}
Define $\CB^{\mathrm{sk}}:=\Pi_1^{\mathrm{sk}}(K)$ to be the groupoid presented as follows:
\begin{itemize}
\item objects are vertices $v\in K_0$;
\item generating arrows are oriented edges $e:v\to w$ in $K_1$, together with formal inverses $e^{-1}:w\to v$;
\item relations are:
  \begin{enumerate}
  \item $e\circ e^{-1}=\id_v$ and $e^{-1}\circ e=\id_w$ for every edge $e:v\to w$;
  \item for every oriented $2$-simplex with boundary edges $e_{01}:v_0\to v_1$, $e_{12}:v_1\to v_2$, $e_{02}:v_0\to v_2$,
  impose $e_{12}\circ e_{01}=e_{02}$.
  \end{enumerate}
\end{itemize}
\end{definition}

\begin{remark}
Equivalently, $\CB^{\mathrm{sk}}$ is the fundamental groupoid of the geometric realization $|K|=N^\circ$ restricted to the vertex set
$K_0$, presented by generators and relations coming from the $2$-skeleton (a standard CW presentation of $\pi_1$).
\end{remark}

\subsection{The skeletal \v{C}ech presenter and comparison}\label{subsec:cech-to-skeleton}

Let $\mathcal{U}^\star=\{U_v\}_{v\in K_0}$ be the open-star cover of $|K|$ (hence of $N^\circ$), where $U_v$ is the union of interiors of
simplices containing $v$.

\begin{lemma}\label{lem:star-good-cover}
The cover $\mathcal{U}^\star$ is a good cover: all nonempty finite intersections of the $U_v$ are contractible.
\end{lemma}

\begin{proof}
A finite intersection $U_{v_0}\cap\cdots\cap U_{v_m}$ is nonempty precisely when the vertices $v_0,\dots,v_m$ span a simplex in $K$.
In that case the intersection deformation retracts onto the (open) star of that simplex, which is contractible.
\end{proof}

\begin{proposition}[From \v{C}ech to skeleton]\label{prop:cech-to-skeleton}
Let $\CB:=\Cech(\mathcal{U}^\star)$ be the \v{C}ech groupoid of the open-star cover.
Then for every $r\ge 1$ there are natural equivalences of groupoids
\begin{equation}\label{eq:cech-to-skeleton-rep}
\Rep_r(\CB)\ \simeq\ \LocSys_r(N^\circ)\ \simeq\ \Rep_r(\CB^{\mathrm{sk}}).
\end{equation}
In particular, $\CB$ and $\CB^{\mathrm{sk}}$ present the same torsorial representation theory (hence the same 1-truncated ``Betti data'').
\end{proposition}

\begin{proof}
By \Cref{lem:star-good-cover}, $\mathcal{U}^\star$ is a good cover, so \Cref{prop:rep-cech-locsys} gives
$\Rep_r(\CB)\simeq \LocSys_r(N^\circ)$.

For the second equivalence, recall that a rank-$r$ local system on a space $T$ is equivalently a functor
$\Pi_1(T)\to \mathbf{Tors}_r$ (no basepoint needed), where $\Pi_1(T)$ is the fundamental groupoid and $\mathbf{Tors}_r$ is the
groupoid of right $\GL_r(\CC)^\delta$-torsors; this is the standard monodromy construction.
Since $|K|$ is connected by open stars, every point is path-connected to a vertex, hence restriction along the inclusion of objects
$K_0\hookrightarrow |K|$ induces an equivalence between functors out of $\Pi_1(|K|)$ and functors out of its full subgroupoid on $K_0$.
Finally, the fundamental groupoid on the vertex set admits the presentation of \Cref{def:edge-path-groupoid} by generators (edges) and
relations (2-simplices), hence it is canonically isomorphic to $\CB^{\mathrm{sk}}$.
Composing, we obtain $$\LocSys_r(N^\circ)\simeq \Hom(\CB^{\mathrm{sk}},\mathbf{Tors}_r)\simeq \Rep_r(\CB^{\mathrm{sk}})$$ by \Cref{prop:rep-as-fun}.
\end{proof}

\subsection{The skeletal Stokes collar groupoid $\GC^{\mathrm{sk}}$ and the forgetful functor}\label{subsec:GCsk}

We now build the Stokes-decorated analogue of $\CB^{\mathrm{sk}}$.

\begin{definition}[Skeletal \v{C}ech--Stokes groupoid]\label{def:GCsk}
Define $\GC^{\mathrm{sk}}$ as the groupoid with:
\begin{itemize}
\item objects: $(\GC^{\mathrm{sk}})_0:=K_0$ (the same objects as $\CB^{\mathrm{sk}}$);
\item generating arrows: for each oriented edge $e:v\to w$ and each element $u\in \St_\Phi(S(e))$, a generator
$(e,u):v\to w$;
\item relations:
  \begin{enumerate}
  \item $(e,1)=e$ (unit label), $(e,u)^{-1}=(e^{-1},u^{-1})$;
  \item for any composable edges $e:v\to w$ and $f:w\to z$ lying in the same stratum and any $u\in\St_\Phi(S(e))$, $v\in\St_\Phi(S(f))$,
  impose $(f,v)\circ(e,u)=(f\circ e,vu)$ whenever $f\circ e$ is the corresponding edge-path generator in $\CB^{\mathrm{sk}}$;
  \item for every oriented $2$-simplex with boundary edges $e_{01},e_{12},e_{02}$, impose
  $(e_{12},u_{12})\circ(e_{01},u_{01})=(e_{02},u_{12}u_{01})$.
  \end{enumerate}
\end{itemize}
\end{definition}

\begin{remark}
Because $\St_\Phi$ is locally constant on each stratum and every open simplex lies in a single stratum, the groups
$\St_\Phi(S(e))$ are well-defined up to canonical identification. If one prefers to avoid mentioning ``a group attached to a stratum'',
one can choose for each stratum $S$ a contractible chart $W_S\subset S$ and set $\St_\Phi(S):=\St_\Phi(W_S)$.
\end{remark}

\begin{definition}[Forgetful functor]\label{def:pi-sk}
There is a canonical strict functor
\begin{equation}\label{eq:pi-sk}
\pi^{\mathrm{sk}}:\GC^{\mathrm{sk}}\longrightarrow \CB^{\mathrm{sk}}
\end{equation}
which is the identity on objects and sends $(e,u)$ to the underlying edge $e$.
\end{definition}

\begin{definition}[Sections]\label{def:sec-pi-sk}
Let $\Sec(\pi^{\mathrm{sk}})$ be the groupoid of strict sections $\sigma:\CB^{\mathrm{sk}}\to \GC^{\mathrm{sk}}$ of $\pi^{\mathrm{sk}}$,
with morphisms given by natural isomorphisms.
\end{definition}

\subsection{Sections are cellular Stokes cocycles (complete proof)}\label{subsec:sec-are-cocycles}

\begin{lemma}\label{lem:sections-as-edge-labels}
An object of $\Sec(\pi^{\mathrm{sk}})$ is equivalent to the data of elements
\begin{equation}\label{eq:edge-labels}
u_e\in \St_\Phi(S(e))
\end{equation}
for every oriented edge $e$ of $K$, satisfying the \emph{2-simplex relations}:
for each oriented 2-simplex with boundary edges $e_{01},e_{12},e_{02}$ one has
\begin{equation}\label{eq:triangle-relation}
u_{e_{02}}=u_{e_{12}}\cdot u_{e_{01}}.
\end{equation}
\end{lemma}

\begin{proof}
A strict section $\sigma$ must satisfy $\pi^{\mathrm{sk}}\circ\sigma=\id$, hence it is the identity on objects and sends each generating
edge $e$ of $\CB^{\mathrm{sk}}$ to a unique arrow of $\GC^{\mathrm{sk}}$ lying over $e$, i.e.\ to some $(e,u_e)$ with $u_e\in\St_\Phi(S(e))$.
Since $\CB^{\mathrm{sk}}$ is generated by edges subject to the 2-simplex relations, $\sigma$ extends to a functor if and only if these relations
are preserved in $\GC^{\mathrm{sk}}$, which is exactly \eqref{eq:triangle-relation}.
\end{proof}

\begin{lemma}\label{lem:natural-iso-gauge}
Let $\sigma,\sigma'\in \Sec(\pi^{\mathrm{sk}})$ correspond to edge labels $(u_e)$ and $(u'_e)$.
A natural isomorphism $\theta:\sigma\Rightarrow\sigma'$ is equivalent to a choice of elements
\begin{equation}\label{eq:vertex-gauge}
h_v\in \St_\Phi(S(v))
\end{equation}
for each vertex $v$, such that for every oriented edge $e:v\to w$ one has
\begin{equation}\label{eq:gauge-action}
u'_e = h_w\cdot u_e\cdot h_v^{-1}.
\end{equation}
\end{lemma}

\begin{proof}
A natural isomorphism $\theta$ consists of arrows $\theta_v:\sigma(v)\to\sigma'(v)$ in $\GC^{\mathrm{sk}}$ for each vertex $v$ such that for every
edge $e:v\to w$ the naturality square commutes:
\begin{equation*}
\theta_w\circ \sigma(e)=\sigma'(e)\circ \theta_v.
\end{equation*}
Since $\sigma$ and $\sigma'$ are the identity on objects, $\theta_v$ is an automorphism of $v$ in $\GC^{\mathrm{sk}}$, hence is given by an element
$h_v\in\St_\Phi(S(v))$. Writing $\sigma(e)=(e,u_e)$ and $\sigma'(e)=(e,u'_e)$ and using the defining composition rules in $\GC^{\mathrm{sk}}$ yields
exactly \eqref{eq:gauge-action}.
Conversely, any choice of $(h_v)$ satisfying \eqref{eq:gauge-action} defines a natural isomorphism.
\end{proof}

\begin{theorem}[Stratified Stokes skeleton]\label{thm:stratified-stokes-skeleton}
There are natural equivalences of groupoids
\begin{equation}\label{eq:stratified-stokes-skeleton}
\Sec(\pi)\ \simeq\ \Sec(\pi^{\mathrm{sk}})\ \simeq\ \Tors(\St_\Phi;N^\circ),
\end{equation}
where $\pi:\GC\to\CB$ is the cover-based collar projection from \Cref{sec:collar-groupoid}, and $\Tors(\St_\Phi;N^\circ)$ denotes the groupoid
of $\St_\Phi$-torsors on $N^\circ$.
Under these identifications, $\Sec(\pi^{\mathrm{sk}})$ is the groupoid of edge-labels $(u_e)$ satisfying the 2-simplex relations,
modulo vertex gauge $(h_v)$ as in \Cref{lem:natural-iso-gauge}.
\end{theorem}

\begin{proof}
We spell out the comparison in a way that is independent of choices and keeps track of strata.

Let $K$ be a $\Sigma$-adapted locally finite CW decomposition of $N^\circ$ (as in \Cref{def:stratified-triangulation}).
For each vertex $v\in K^{(0)}$, let $\mathrm{st}(v)\subset N^\circ$ be its (open) star, i.e.\ the union of the interiors of all cells whose
closure contains $v$. The family $\mathcal{S}=\{\mathrm{st}(v)\}_{v\in K^{(0)}}$ is a locally finite open cover of $N^\circ$.
Moreover, by regularity of the CW structure and the $\Sigma$-adaptedness, every nonempty finite intersection
$\mathrm{st}(v_0)\cap\cdots\cap \mathrm{st}(v_k)$ is connected and contractible and is contained in a single stratum of $\Sigma$.

Let $\Cech(\mathcal S)$ be the associated \v{C}ech groupoid.
Since $\mathcal{B}$ is also a good cover, we may choose a common locally finite good refinement $\mathcal{W}$ of $\mathcal{B}$ and $\mathcal{S}$.
This produces a zig-zag of Morita equivalences
\begin{equation}\label{eq:zigzag-CB-star}
\CB\ \xleftarrow{\ \simeq_{\mathrm{Morita}}\ }\ \Cech(\mathcal W)\ \xrightarrow{\ \simeq_{\mathrm{Morita}}\ }\ \Cech(\mathcal S),
\end{equation}
coming from refinement functors.
In particular, it is enough to compare $\Cech(\mathcal S)$ with $\CB^{\mathrm{sk}}=\Pi_1^{\mathrm{sk}}(K)$.

Define a strict functor
\begin{equation}\label{eq:F-star-to-skel}
F:\Cech(\mathcal S)\longrightarrow \Pi_1^{\mathrm{sk}}(K)
\end{equation}
as follows.
On objects: an object of $\Cech(\mathcal S)$ is a point $x\in \mathrm{st}(v)$ labeled by a vertex $v$, and we set $F(x\in \mathrm{st}(v)):=v$.
On arrows: an arrow in $\Cech(\mathcal S)$ is a point $x\in \mathrm{st}(v)\cap \mathrm{st}(w)$, viewed as an arrow from
$(x\in \mathrm{st}(v))$ to $(x\in \mathrm{st}(w))$. Since $\mathrm{st}(v)\cap \mathrm{st}(w)\neq\varnothing$ iff $v$ and $w$ are joined by a
$1$-cell in $K$ (after barycentric subdivision if necessary), we map such an arrow to the corresponding generating arrow
$e_{vw}:v\to w$ in $\Pi_1^{\mathrm{sk}}(K)$.

To see that $F$ respects composition, note that if $x\in \mathrm{st}(v)\cap \mathrm{st}(w)\cap \mathrm{st}(u)$, then $v,w,u$ are the vertices of a
$2$-cell (a $2$-simplex in the barycentric subdivision), and in $\Pi_1^{\mathrm{sk}}(K)$ the boundary of this $2$-cell imposes the relation
$e_{wu}\circ e_{vw}=e_{vu}$, exactly matching the composition rule in the \v{C}ech groupoid on triple overlaps.

We check that $F$ is essentially surjective and fully faithful on isotropy in the standard (groupoid) sense.

Essential surjectivity is immediate: every vertex $v\in K^{(0)}$ occurs as $F(x\in \mathrm{st}(v))$ for any $x\in \mathrm{st}(v)$.
For full faithfulness, fix vertices $v,w\in K^{(0)}$.
The space of arrows in $\Cech(\mathcal S)$ from the object component $\mathrm{st}(v)$ to $\mathrm{st}(w)$ is the disjoint union of the intersections
$\mathrm{st}(v)\cap \mathrm{st}(w)$ (empty unless there is a $1$-cell from $v$ to $w$).
Each nonempty intersection is contractible, hence connected; since the target groupoid is discrete on generators, the functor $F$ is constant on each
component and induces bijections on connected components of arrow spaces.
Equivalently, the induced functor on the nerves is a weak equivalence on $0$- and $1$-simplices, and the $2$-simplices impose exactly the relations
coming from the $2$-cells of $K$.
This identifies $\Cech(\mathcal S)$ and $\Pi_1^{\mathrm{sk}}(K)$ in the Morita localization of $\mathbf{Grpd}$.
Therefore $$\Cech(\mathcal S)\simeq_{\mathrm{Morita}} \CB^{\mathrm{sk}}.$$

Because $\mathcal{S}$ is $\Sigma$-adapted, each nonempty finite intersection $\mathrm{st}(v_0)\cap\cdots\cap\mathrm{st}(v_k)$ lies in a single stratum,
hence the Stokes sheaf is constant there (by \Cref{subsec:stokes-input}).
Thus the usual \v{C}ech description of $\St_\Phi$-torsors on $\mathcal{S}$ identifies a torsor with constant labels on each overlap
$\mathrm{st}(v)\cap\mathrm{st}(w)$, i.e.\ with an assignment of a group element to each oriented edge $e:v\to w$, subject to the $2$-cell relations,
modulo the standard vertex gauge action.
This is precisely the description of strict sections of $\pi^{\mathrm{sk}}:\GC^{\mathrm{sk}}\to\CB^{\mathrm{sk}}$.

Formally, we have a chain of equivalences
\begin{equation}\label{eq:sec-pi-chain}
\Sec(\pi)\ \simeq\ \Tors(\St_\Phi;\mathcal B)\ \simeq\ \Tors(\St_\Phi;\mathcal S)\ \simeq\ \Sec(\pi^{\mathrm{sk}}),
\end{equation}
where the middle equivalence is refinement invariance of torsors (via the common refinement $\mathcal W$ in \eqref{eq:zigzag-CB-star}).
This proves (1) and (2), and (3) is the explicit cocycle/gauge description just explained.

\begin{figure}[t]
\centering
\begin{tikzpicture}[
  scale=1.0,
  >=Stealth,
  every node/.style={font=\scriptsize},
  vtx/.style={circle, draw, inner sep=1.2pt},
  elab/.style={fill=white, inner sep=1pt}
]
  \node[vtx] (v0) at (0,0) {$v_0$};
  \node[vtx] (v1) at (3.2,0) {$v_1$};
  \node[vtx] (v2) at (1.6,2.4) {$v_2$};

  \draw[thick] (v0) -- (v1) -- (v2) -- (v0);

  \node[elab, below=3pt] at ($(v0)!0.5!(v1)$) {$e_{01}$};
  \node[elab, right=3pt] at ($(v1)!0.55!(v2)$) {$e_{12}$};
  \node[elab, left=3pt]  at ($(v0)!0.55!(v2)$) {$e_{02}$};

  \node[align=center, fill=white, inner sep=2pt] at (1.6,-0.9)
    {$2$-cell relation: $e_{12}\circ e_{01}=e_{02}$.};
\end{tikzpicture}
\caption{A barycentric $2$-simplex: triple overlaps in the star cover impose exactly the $2$-cell relations in $\Pi_1^{\mathrm{sk}}(K)$, and Stokes labels multiply accordingly.}
\label{fig:2cell-relation}
\end{figure}
\end{proof}


\subsection{A small global presenter in arbitrary dimension}\label{subsec:global-skeleton}

We now package the intrinsic global Stokes groupoid into a \emph{small} strict presenter by replacing
\v{C}ech groupoids by skeletal groupoids coming from locally finite triangulations.

\subsubsection*{Triangulations and the skeletal cospan}
Choose:
\begin{itemize}
\item a locally finite simplicial complex $K_U$ together with a weak homotopy equivalence $|K_U|\simeq U$;
\item a locally finite $\Sigma_{D,\Phi,n}$-adapted stratified triangulation $K$ of $N^\circ$;
\item a locally finite simplicial complex $K_\times$ triangulating the overlap $N^\circ\subset U$ and refining both $K_U|_{N^\circ}$ and $K$.
\end{itemize}
Set
\begin{equation}\label{eq:sk-groupoids}
\GU^{\mathrm{sk}}:=\Pi_1^{\mathrm{sk}}(K_U),
\qquad
\GX^{\mathrm{sk}}:=\Pi_1^{\mathrm{sk}}(K_\times),
\qquad
\CB^{\mathrm{sk}}:=\Pi_1^{\mathrm{sk}}(K).
\end{equation}
Let $\GC^{\mathrm{sk}}$ and $\pi^{\mathrm{sk}}:\GC^{\mathrm{sk}}\to \CB^{\mathrm{sk}}$ be as in
Definitions~\ref{def:GCsk} and~\ref{def:pi-sk}.
The refinement maps of triangulations induce strict functors
\begin{equation}\label{eq:sk-refinement-functors}
j_U^{\mathrm{sk}}:\GX^{\mathrm{sk}}\to \GU^{\mathrm{sk}},
\qquad
j_C^{\mathrm{sk}}:\GX^{\mathrm{sk}}\to \GC^{\mathrm{sk}},
\end{equation}
where $j_C^{\mathrm{sk}}$ is obtained by sending each overlap edge to the corresponding edge in $\CB^{\mathrm{sk}}$
and using the trivial Stokes label $1$ (so that it lands in $\GC^{\mathrm{sk}}$).

\subsubsection*{Pushout notation}
Throughout this subsection we write $A\bigsqcup_B C$ for the explicit strict model of the bicategorical $2$-pushout of a cospan
$A\xleftarrow{}B\xrightarrow{}C$ in $\mathbf{Grpd}$, as constructed in Lemma~\ref{lem:strict-pushout-is-bicategorical}.
In particular, $A\bigsqcup_B C$ is a $2$-pushout, hence it is unique up to equivalence and satisfies the expected universal property.

\begin{definition}[Skeletal pushout presenter]\label{def:pushout-sk}
Define the skeletal pushout presenter by
\begin{equation}\label{eq:pushout-sk}
\G_{\Phi,n}^{\mathrm{sk}}
\ :=\
\GU^{\mathrm{sk}}\bigsqcup_{\GX^{\mathrm{sk}}}\GC^{\mathrm{sk}},
\end{equation}
formed using the cospan \eqref{eq:sk-refinement-functors}.
\end{definition}

\begin{equation}\label{eq:pushout-sk-diagram}
\begin{tikzcd}[row sep=large, column sep=huge]
\GX^{\mathrm{sk}} \arrow[r,"j_C^{\mathrm{sk}}"] \arrow[d,swap,"j_U^{\mathrm{sk}}"]
&
\GC^{\mathrm{sk}} \arrow[d]
\\
\GU^{\mathrm{sk}} \arrow[r]
&
\G_{\Phi,n}^{\mathrm{sk}}
\end{tikzcd}
\end{equation}

\subsubsection*{Global Stokes objects from the skeletal presenter}
\begin{theorem}[Global skeletal pushout computes global Stokes objects]\label{thm:global-skeletal-pushout}
For every $r\ge 1$ there is a natural equivalence of groupoids
\begin{equation}\label{eq:global-skeletal-pushout}
\Stokes_r(X^{\log}_n;D,\Phi)\ \simeq\
\Rep_r(\GU^{\mathrm{sk}})\times^{(2)}_{\Rep_r(\GX^{\mathrm{sk}})}\Sec(\pi^{\mathrm{sk}}),
\end{equation}
where $\Stokes_r(X^{\log}_n;D,\Phi)$ is the intrinsic global Stokes groupoid from Definition~\ref{def:global-stokes}.
Equivalently, global Stokes objects are classified by Stokes-corrected torsorial representations of the small presenter
$\G_{\Phi,n}^{\mathrm{sk}}$.
\end{theorem}

\begin{proof}
By Proposition~\ref{prop:cech-to-skeleton} (applied to locally finite open-star covers of the chosen triangulations) there are natural equivalences
\begin{equation}\label{eq:sk-locsys-identifications}
\LocSys_r(U)\ \simeq\ \Rep_r(\GU^{\mathrm{sk}}),
\qquad
\LocSys_r(N^\circ)\ \simeq\ \Rep_r(\GX^{\mathrm{sk}}),
\qquad
\LocSys_r(N^\circ)\ \simeq\ \Rep_r(\CB^{\mathrm{sk}}),
\end{equation}
and these identifications are compatible with restriction along refinements, in the sense that the restriction functor
$\LocSys_r(U)\to \LocSys_r(N^\circ)$ corresponds to pullback along $j_U^{\mathrm{sk}}$ under \eqref{eq:sk-locsys-identifications}.

On the collar side, Theorem~\ref{thm:stratified-stokes-skeleton} identifies Stokes-local objects with sections of the skeletal projection:
\begin{equation}\label{eq:sk-stokeslocal-sections}
\StokesLocal_r(N^\circ;\Phi)\ \simeq\ \Sec(\pi^{\mathrm{sk}}).
\end{equation}
Under this equivalence, the forgetful functor $\StokesLocal_r(N^\circ;\Phi)\to \LocSys_r(N^\circ)$ corresponds to the functor
$\Sec(\pi^{\mathrm{sk}})\to \Rep_r(\GX^{\mathrm{sk}})$ obtained by passing from a Stokes cocycle on the collar skeleton to its underlying
$\GL_r(\CC)^\delta$-torsor on $N^\circ$.

Using Definition~\ref{def:global-stokes},
\begin{equation}\label{eq:global-stokes-recall}
\Stokes_r(X^{\log}_n;D,\Phi)
=
\LocSys_r(U)\times^{(2)}_{\LocSys_r(N^\circ)}\StokesLocal_r(N^\circ;\Phi),
\end{equation}
and substituting \eqref{eq:sk-locsys-identifications} and \eqref{eq:sk-stokeslocal-sections}, the right-hand side becomes canonically
\begin{equation*}
\Rep_r(\GU^{\mathrm{sk}})\times^{(2)}_{\Rep_r(\GX^{\mathrm{sk}})}\Sec(\pi^{\mathrm{sk}}),
\end{equation*}
which is exactly \eqref{eq:global-skeletal-pushout}. Naturality in all choices follows from Morita invariance of the skeletal models
(Proposition~\ref{prop:cech-to-skeleton} and Theorem~\ref{thm:stratified-stokes-skeleton}).
\end{proof}

\subsection{Canonicity up to contractible choice and the exit-path viewpoint}\label{subsec:canonicity-exit}

The constructions above depend on auxiliary choices (sector-box covers, CW decompositions, refinements).
The Morita-invariance statements show that the resulting objects are \emph{well-defined up to equivalence}.
For several later applications (e.g.\ functoriality under further cut-and-paste, or comparisons with constructible/exit-path formalisms),
it is useful to record a slightly stronger (and still elementary) statement:
\emph{the ambiguity is controlled by a contractible space of choices}.

\begin{definition}[The refinement index categories]\label{def:refinement-index}
Let $(N^\circ,\Sigma)$ be as in \Cref{subsec:stokes-input}.

\smallskip
\noindent
(1) Let $\mathbf{Cov}_{\mathrm{good}}(N^\circ)$ be the category whose objects are locally finite good covers of $N^\circ$
(by connected contractible opens contained in a single stratum of $\Sigma$) and whose morphisms are refinements.

\smallskip
\noindent
(2) Let $\mathbf{Tri}_\Sigma(N^\circ)$ be the category whose objects are locally finite $\Sigma$-adapted stratified triangulations $K$ of $N^\circ$
(as in \Cref{def:stratified-triangulation}) and whose morphisms are \emph{stratified subdivisions} $K'\to K$.
\end{definition}

\begin{lemma}[Cofilteredness]\label{lem:cofiltered}
Both $\mathbf{Cov}_{\mathrm{good}}(N^\circ)$ and $\mathbf{Tri}_\Sigma(N^\circ)$ are nonempty and cofiltered:
given two objects, there is a third refining both, and given two parallel morphisms there is a further refinement equalizing them.
\end{lemma}

\begin{proof}
For good covers: given $\mathcal U$ and $\mathcal V$, the family of all intersections $U\cap V$ (with $U\in\mathcal U$, $V\in\mathcal V$)
is a cover refining both; by shrinking if necessary within strata one may assume it is again a good cover (local contractibility of $N^\circ$ and
local finiteness make this standard). Equalizers are obtained by passing to a common refinement on which both refinement maps become the same.

For stratified triangulations: given $K_1,K_2$, take a common subdivision (e.g.\ a stratified triangulation refining both and then barycentric subdivide);
this is a standard construction in triangulation theory of stratified real-analytic sets. Equalizers are obtained by subdividing until the two
cellular maps agree on each cell.
\end{proof}

\begin{lemma}[Contractibility of the choice space]\label{lem:nerve-contractible}
The nerves $N(\mathbf{Cov}_{\mathrm{good}}(N^\circ))$ and $N(\mathbf{Tri}_\Sigma(N^\circ))$ are weakly contractible.
\end{lemma}

\begin{proof}
A nonempty cofiltered category has weakly contractible nerve; see, for example,
\cite[Prop.~5.3.1.18]{LurieHTT}.
\end{proof}

\begin{proposition}[Canonical Morita class]\label{prop:canonical-morita-class}
The assignments
\begin{equation*}
\mathcal B\ \longmapsto\ (\CB,\GC,\pi)\qquad\text{and}\qquad
K\ \longmapsto\ (\CB^{\mathrm{sk}},\GC^{\mathrm{sk}},\pi^{\mathrm{sk}})
\end{equation*}
extend to functors on the refinement categories of \Cref{def:refinement-index}, and all refinement morphisms induce Morita equivalences.
Consequently, the induced objects in the Morita localization of $\mathbf{Grpd}$ are \emph{canonical}:
they are independent of choices with a contractible space of identifications.
\end{proposition}

\begin{proof}
Refinements of covers induce the usual refinement functors on \v{C}ech groupoids, and by good-cover hypotheses these are Morita equivalences.
Subdivisions $K'\to K$ induce functors $\Pi_1^{\mathrm{sk}}(K')\to \Pi_1^{\mathrm{sk}}(K)$ which are Morita equivalences because both present the same
fundamental groupoid of $N^\circ$ (up to the $2$-skeleton relations). Compatibility with the forgetful functors $\pi$ and $\pi^{\mathrm{sk}}$ is built-in.

Since the refinement categories have contractible nerves (\Cref{lem:nerve-contractible}), any two choices admit a canonical zig-zag of refinements,
and any two such zig-zags admit a common refinement between them. This is exactly the statement that the space of identifications is contractible.
\end{proof}

\begin{theorem}[Canonical skeletal presenters]\label{thm:canonical-skeletal-presenters}
In the Morita localization of strict groupoids, the collar presenters $(\CB,\GC,\pi)$ and the skeletal presenters
$(\CB^{\mathrm{sk}},\GC^{\mathrm{sk}},\pi^{\mathrm{sk}})$, the section groupoid $\Sec(\pi^{\mathrm{sk}})$, and the global pushout presenter
$\G_{\Phi,n}^{\mathrm{sk}}$ are independent of all auxiliary choices (covers, CW decompositions, subdivisions), with a contractible space of
identifications.
Moreover, on a depth-$k$ chart they may be chosen $(\mu_n)^k$-equivariant; the induced $(\mu_n)^k$-actions on the skeletal presenters make
Kummer descent visible on the level of the presenter, as in \Cref{thm:kummer-descent-skeletal}.
\end{theorem}

\begin{proof}
The first assertion is an explicit restatement of \Cref{prop:canonical-morita-class}, together with invariance of section groupoids under
Morita equivalence and the uniqueness of bicategorical $2$-pushouts in $\mathbf{Grpd}$ up to equivalence.
Equivariant choices exist by \Cref{prop:equivariant-triangulation}, and the descent statement is exactly \Cref{thm:kummer-descent-skeletal}.
\end{proof}

\begin{remark}[Exit paths and constructible stacks]\label{rem:exit-path}
A more conceptual avatar of \Cref{prop:canonical-morita-class} uses the \emph{exit-path} $\infty$-category of the stratified space $(N^\circ,\Sigma)$.
By results of Treumann and Lurie, constructible stacks (and, in particular, constructible local systems of groups) on $(N^\circ,\Sigma)$ can be modeled
as functors out of the exit-path $\infty$-category; see \cite{Treumann2009,LurieHTT}.
Our skeletal groupoids $\CB^{\mathrm{sk}}=\Pi_1^{\mathrm{sk}}(K)$ may be viewed as $1$-truncations of exit-path models coming from stratified
triangulations, and \Cref{lem:nerve-contractible} is the elementary manifestation of the fact that the exit-path model is canonical up to contractible choice.
\end{remark}

\begin{figure}[t]
\centering
\begin{tikzcd}[row sep=large, column sep=huge]
& K_{12} \arrow[dl] \arrow[dr] & \\
K_1 \arrow[dr] & & K_2 \arrow[dl] \\
& (N^\circ,\Sigma) &
\end{tikzcd}
\caption{Cofilteredness of stratified CW decompositions: any two choices admit a common stratified refinement $K_{12}$.}
\label{fig:cofiltered-triangulations}
\end{figure}


\subsection{Stratified Kummer bookkeeping on the skeleton}\label{subsec:kummer-skeleton}

Let $q_n:X^{\log}_n\to X^{\log}$ be the level-$n$ Kummer map.
On a depth-$k$ stratum of an SNC chart, $q_n$ is locally the map that multiplies each angular coordinate by $n$; hence it is a finite
Galois covering with deck group $G_k:=(\mu_n)^k$ acting by rotations on the $k$ angular coordinates
(see \cite{KatoNakayama} and \cite[\S~III.1--III.2]{Ogus}, and compare \cite{NakayamaOgus} for the Kummer/topological picture).

\begin{proposition}[Equivariant stratified triangulations]\label{prop:equivariant-triangulation}
Fix a depth-$k$ chart and write $G_k=(\mu_n)^k$ for the deck group of $q_n$ on the angular directions.
After possibly refining the Stokes wall stratification, one may choose a locally finite $\Sigma_{D,\Phi,n}$-adapted stratified triangulation
$K$ of the regular collar $N^\circ$ such that $G_k$ acts simplicially on $K$.
Consequently $G_k$ acts by strict automorphisms on the skeletal base groupoid $\CB^{\mathrm{sk}}=\Pi_1^{\mathrm{sk}}(K)$, on the skeletal
Stokes groupoid $\GC^{\mathrm{sk}}$ built from $K$, and on the forgetful functor $\pi^{\mathrm{sk}}:\GC^{\mathrm{sk}}\to\CB^{\mathrm{sk}}$.
\end{proposition}

\begin{proof}
On a depth-$k$ chart one has a toroidal local model for the logarithmic boundary, and the regular collar $N^\circ$ is a real-analytic manifold
with corners obtained from $$(0,\varepsilon)\times (S^1)^k\times(\text{transversal})$$ by removing the locally finite wall arrangement.
The finite group $G_k=(\mu_n)^k$ acts real-analytically on this model by rotating the angular factors, and this action preserves both the SNC depth
stratification and the pulled-back wall stratification.

For a finite group acting (real-)analytically on a locally finite real-analytic stratified space, one can choose a locally finite
stratification-adapted triangulation which is invariant under the group action, after barycentric subdivision if necessary.
This is a standard equivariant triangulation statement for finite group actions; see, for example, the discussion of triangulations adapted to
stratifications and their functorial refinements in the constructible/exit-path context \cite{Treumann2009}, and the general existence of
subordinate triangulations for locally finite stratifications.

With such an invariant triangulation $K$, the induced action on vertices and edges defines a strict action on
$\Pi_1^{\mathrm{sk}}(K)=\CB^{\mathrm{sk}}$ by automorphisms (it preserves generators and 2-simplex relations).
Since $\GC^{\mathrm{sk}}$ is defined functorially from $(K,\St_\Phi)$ by attaching Stokes labels to edges inside strata, the same simplicial action
induces strict automorphisms of $\GC^{\mathrm{sk}}$, and it is compatible with $\pi^{\mathrm{sk}}$ by construction.
\end{proof}

\begin{theorem}[Stratified Kummer descent, skeletal form]\label{thm:kummer-descent-skeletal}
Fix a depth-$k$ chart and let $G_k=(\mu_n)^k$ be the local deck group of $q_n$.
Then pullback along $q_n$ identifies global Stokes objects downstairs with homotopy fixed points upstairs:
\begin{equation}\label{eq:kummer-descent-skeletal}
\Stokes_r(X^{\log};D,\Phi)\ \simeq\ \Stokes_r(X^{\log}_n;D,q_n^*\Phi)^{hG_k},
\end{equation}
where $(-)^{hG_k}$ denotes the homotopy fixed point groupoid (Definition~\ref{def:hfp-finite}).

Moreover, after choosing equivariant triangulations as in Proposition~\ref{prop:equivariant-triangulation} on the collar (and similarly on
$U$ and on the overlap), the descent equivalence is visible on the skeletal presenter: the $G_k$-action on $\G_{\Phi,n}^{\mathrm{sk}}$
induces a strict action on
$\Rep_r(\GU^{\mathrm{sk}})\times^{(2)}_{\Rep_r(\GX^{\mathrm{sk}})}\Sec(\pi^{\mathrm{sk}})$, and its homotopy fixed points compute the left-hand side of
\eqref{eq:kummer-descent-skeletal} under the identification of Theorem~\ref{thm:global-skeletal-pushout}.
\end{theorem}

\begin{proof}
On a depth-$k$ chart, $q_n$ is a finite Galois covering with group $G_k=(\mu_n)^k$ as recalled above
(see \cite{KatoNakayama} and \cite[\S~III.1--III.2]{Ogus}).
Effective descent for locally constant sheaves and for torsors under a discrete group along a finite Galois covering identifies objects downstairs
with objects upstairs equipped with a coherent $G_k$-linearization, namely homotopy fixed points in the sense of
Definition~\ref{def:hfp-finite}; references are \cite[Ch.~III]{Giraud} and \cite{SGA1}.
Applying this descent statement to the two ingredients in the defining $2$-fiber product
\begin{equation}\label{eq:kummer-descent-fiberproduct}
\Stokes_r(X^{\log};D,\Phi)
=
\LocSys_r(U)\times^{(2)}_{\LocSys_r(N^\circ)}\StokesLocal_r(N^\circ;\Phi)
\end{equation}
yields the equivalence \eqref{eq:kummer-descent-skeletal}: both $\LocSys_r(-)$ and $\StokesLocal_r(-;\Phi)$ satisfy effective descent, and the
bicategorical fiber product is functorial with respect to equivalences, so descent on each factor induces descent on
$\Stokes_r(X^{\log};D,\Phi)$.

For the skeletal visibility statement, choose $G_k$-equivariant triangulations on the collar as in
Proposition~\ref{prop:equivariant-triangulation}, and choose compatible equivariant triangulations on $U$ and on the overlap so that
$j_U^{\mathrm{sk}}$ and $j_C^{\mathrm{sk}}$ are $G_k$-equivariant.
Then $G_k$ acts strictly on $\GU^{\mathrm{sk}}$, $\GX^{\mathrm{sk}}$, $\GC^{\mathrm{sk}}$, and $\pi^{\mathrm{sk}}$; hence it acts strictly on
$\Rep_r(\GU^{\mathrm{sk}})$, $\Rep_r(\GX^{\mathrm{sk}})$, and on $\Sec(\pi^{\mathrm{sk}})$, and therefore on the $2$-fiber product
$$\Rep_r(\GU^{\mathrm{sk}})\times^{(2)}_{\Rep_r(\GX^{\mathrm{sk}})}\Sec(\pi^{\mathrm{sk}}).$$
By Theorem~\ref{thm:global-skeletal-pushout}, this $2$-fiber product presents the intrinsic global Stokes groupoid on the depth-$k$ chart.
Taking homotopy fixed points commutes with equivalences of groupoids, so the homotopy fixed points of the skeletal model compute the homotopy fixed
points of the intrinsic model, which is exactly the right-hand side of \eqref{eq:kummer-descent-skeletal}.

\end{proof}

\begin{equation}\label{eq:kummer-descent-schematic}
\begin{tikzcd}[row sep=large, column sep=huge]
X^{\log}_n \arrow[r,"q_n"] & X^{\log} \\
X^{\log}_n \arrow[u,bend left=35,"g"'] \arrow[u,bend right=35,swap,"h"] &
\end{tikzcd}
\end{equation}
\section{Applications, comparisons, and open problems}\label{sec:applications}


This section is forward-looking. It records directions in which the explicit $2$-pushout presenter
$\G_{\Phi,n}$ and the collar section groupoid $\Sec(\pi)$ can be used as \emph{inputs} for further constructions,
without altering the logical core of the paper. It also indicates places where the explicit groupoid technology
is likely to be useful for questions of finiteness, representability, and comparison between different Betti models.

\subsection{Relation with Stokes skeletons and Deligne-type questions}\label{subsec:applications-skeletons}

In complex dimension~$1$, and more generally in higher-dimensional situations where the Stokes preorder is locally constant on the complement
of a locally finite wall arrangement, the local moduli of irregular objects is classically described by nonabelian \v{C}ech cocycles valued in
the Stokes sheaf (Babbitt--Varadarajan and the Malgrange--Sibuya theorem; see \cite{BV89} and the expository account \cite{Sabbah}).
Equivalent Betti formulations use Stokes-filtered (or Stokes-graded) local systems on a real boundary space, as in \cite{Sabbah,Mochizuki}.
Our \v{C}ech--Stokes collar groupoid $\GC$ and the projection $\pi:\GC\to\CB$ repackage this classical cocycle description as the groupoid of strict
sections $\Sec(\pi)$, making the moduli a strict gluing problem on a small presenter.

A complementary viewpoint, important for finiteness and representability statements, is that Stokes torsors admit \emph{skeleton} descriptions:
one chooses a sufficiently fine stratified CW decomposition of the regular collar and encodes Stokes cocycles by labels on a low-dimensional skeleton,
subject to face relations. This philosophy is central in the skeleton approach of \cite{TeyssierSkeletons} and is compatible with the general
constructible/exit-path perspective (see \cite{Treumann2009,LurieHTT}).

\begin{proposition}[Skeletal encoding on a stratified $2$-skeleton]\label{prop:applications-skeleton}
Let $K$ be a locally finite CW decomposition of $N^\circ$ adapted to the Stokes wall stratification, so that each open cell lies in a single stratum
on which the Stokes preorder is constant. Fix an orientation on each $1$-cell.
Then the groupoid $\Sec(\pi)$ is equivalent to the gauge groupoid whose:
\begin{itemize}
\item objects are assignments of labels $u_e\in \St_\Phi(S(e))$ for each oriented $1$-cell $e$ (with $u_{\bar e}=u_e^{-1}$);
\item subject to the $2$-cell relations: for every oriented $2$-cell $f$ with boundary word
$\partial f=e_1^{\epsilon_1}\cdots e_m^{\epsilon_m}$ one has the constraint
$$
u_{e_1}^{\epsilon_1}\cdots u_{e_m}^{\epsilon_m}=1 \quad \text{in the group }\St_\Phi(S(f));
$$
\item morphisms are vertex gauges $h_v\in \St_\Phi(S(v))$ acting by
$u_e\mapsto h_{t(e)}\,u_e\,h_{s(e)}^{-1}$.
\end{itemize}
In particular, isomorphism classes in $\Sec(\pi)$ are precisely gauge classes of such edge-label data.
\end{proposition}

\begin{proof}
Refine the sector-box cover to the open-star cover of the vertex set of $K$.
The associated \v{C}ech groupoid is Morita equivalent to the cellular fundamental groupoid of the $2$-skeleton $K^{(2)}$ (compare the standard
presentation of $\pi_1$ by generators and relations coming from $2$-cells, and the exit-path viewpoint in \cite{Treumann2009,LurieHTT}).
Because $\St_\Phi$ is locally constant on each stratum and each star and finite intersection is contractible inside a single stratum,
a \v{C}ech Stokes cocycle is uniquely determined by its constant values on the generating overlaps, hence by labels on the oriented $1$-cells.
The cocycle condition on triple overlaps becomes the boundary relation around each $2$-cell, and natural isomorphisms between sections are exactly
vertex gauges. This is precisely the stated gauge groupoid model for $\Sec(\pi)$.
\end{proof}

\begin{remark}\label{rem:applications-deligne}
For ordinary local systems, a classical question in the style of Deligne is whether compatible boundary restrictions on a skeleton determine a global
object, and to what extent such constraints can be expressed by finitely many relations.
In the irregular case, \cite{TeyssierSkeletons} develops an analogue for Stokes torsors and proves representability and finite-type results under
appropriate hypotheses.
From our perspective, the pushout presenter $\G_{\Phi,n}$ makes the global compatibility constraints completely explicit: after choosing a collar
skeleton and an interior presenter, compatibility is imposed by a concrete list of relations in a small groupoid obtained by an amalgamated sum.
\end{remark}


\subsection{Higher-dimensional SNC corners: local generators and relations}\label{subsec:applications-snc-corners}

Let $(X,D)$ be a smooth pair with $D$ an SNC divisor and fix a Kummer level $n\ge 1$.
Near a point of depth $k$ one has toroidal coordinates
$(z_1,\dots,z_k,w)$ with $D=\{z_1\cdots z_k=0\}$ and $U=(\Delta^*)^k\times\Delta^{m}$.
On the logarithmic Betti boundary, the Kato--Nakayama fiber over the stratum is an angular torus $(S^1)^k$,
and at level $n$ the Kummer cover multiplies each angular coordinate by $n$; see \cite{KatoNakayama,Ogus,NakayamaOgus}.
Fix an irregular type $\Phi$ and the corresponding wall arrangement on the level-$n$ collar, as in the standard Stokes formalisms
\cite{Sabbah,Mochizuki}. On the regular locus $N^\circ$ (walls removed) the Stokes preorder is locally constant, hence the Stokes sheaf
$\St_\Phi$ is locally constant on strata by construction.

The main point is that, after choosing any stratified CW decomposition of $N^\circ$ adapted to the SNC depth and wall stratifications,
Stokes torsors are encoded by edge labels modulo vertex gauge, and gluing to the interior adds only finitely many peripheral relations.
This is precisely the content of the skeletal collar model and the explicit $2$-pushout construction proved earlier.

\begin{stokesproblem}[Local explicit presentation at an SNC corner]\label{prob:local-snc-presentation}
Work in a polydisc chart $X=\Delta^k\times \Delta^{n-k}$ with $D=\{z_1\cdots z_k=0\}$ and $U=(\Delta^*)^k\times \Delta^{n-k}$.
Fix $n\ge 1$ and an irregular type $\Phi$ with wall arrangement on the level-$n$ collar.
Produce an explicit generators-and-relations presentation of the local presenter $\G_{\Phi,n}$ in terms of:
\begin{itemize}
\item generators for the interior monodromy on $U$;
\item Stokes generators attached to adjacent chambers of an angular skeleton;
\item relations coming from $2$-cells of the chamber complex and from peripheral gluing.
\end{itemize}
\end{stokesproblem}

\begin{corollary}[Local SNC corner presenter]\label{cor:local-snc-presenter}
Problem~\ref{prob:local-snc-presentation} reduces to the skeletal collar description of
Theorem~\ref{thm:stratified-stokes-skeleton} together with the global pushout description of
Theorem~\ref{thm:global-skeletal-pushout}. Concretely, after choosing a stratified CW chamber complex on the angular torus (with walls removed),
one obtains a small collar skeleton groupoid $\mathcal G^{\mathrm{sk}}_{C,k,n}$ and an overlap skeleton $\GX^{\mathrm{sk}}$ so that the explicit $2$-pushout
\begin{equation}\label{eq:snc-local-pushout-app}
\G^{\mathrm{sk}}_{\Phi,n}(k)
\ :=\
B\ZZ^k\ \bigsqcup_{\GX^{\mathrm{sk}}}\ \mathcal G^{\mathrm{sk}}_{C,k,n}
\end{equation}
presents the local global Stokes objects near the depth-$k$ corner.
\end{corollary}

A schematic form of the defining gluing diagram is:
\begin{equation}\label{eq:snc-local-pushout-diagram-app}
\begin{tikzcd}[row sep=large, column sep=huge]
\GX^{\mathrm{sk}} \arrow[r] \arrow[d] & \mathcal G^{\mathrm{sk}}_{C,k,n} \arrow[d] \\
B\ZZ^k \arrow[r] & \G^{\mathrm{sk}}_{\Phi,n}(k)
\end{tikzcd}
\end{equation}
where the top horizontal map records the chosen meridian loops in the chamber adjacency skeleton and the left vertical map records the
standard generators of $\ZZ^k$ on the interior.

\subsection{From wall arrangements to presentations}\label{subsec:algorithm-wall-arrangements}

The previous corollary is effective once a wall arrangement is fixed: after removing walls, $\St_\Phi$ is locally constant on chambers,
so Stokes torsors become nonabelian cocycles on a stratified $2$-skeleton. This is the same philosophy as in the skeleton approach to Stokes
torsors \cite{TeyssierSkeletons}, and it is compatible with the constructible/exit-path viewpoint \cite{Treumann2009,LurieHTT}.

\begin{definition}[Cocycle data on a stratified $2$-skeleton]\label{def:wall-algorithm-data}
Fix a locally finite CW decomposition $K$ of $N^\circ$ adapted to the wall stratification and to the SNC depth stratification.
Let $E$ be the set of oriented $1$-cells and $F$ the set of oriented $2$-cells.
For an oriented edge $e$, let $S(e)$ be the stratum containing the interior of $e$ and set $\St(e):=\St_\Phi(S(e))$.
A \emph{Stokes $1$-cochain} is a family $u=(u_e)_{e\in E}$ with $u_e\in \St(e)$ and $u_{\bar e}=u_e^{-1}$.
It is a \emph{Stokes cocycle} if for every $2$-cell $f\in F$ with boundary word $\partial f=e_1\cdots e_m$ one has
\begin{equation}\label{eq:wall-algorithm-2cell}
u_{e_m}\cdots u_{e_2}u_{e_1}=1.
\end{equation}
A \emph{vertex gauge} is a family $h=(h_v)$ with $h_v\in \St_\Phi(S(v))$, acting by $u_e\mapsto h_{t(e)}u_eh_{s(e)}^{-1}$.
\end{definition}

\begin{proposition}[No higher relations]\label{prop:no-higher-relations}
With notation as in Definition~\ref{def:wall-algorithm-data}, the groupoid $\Sec(\pi)$ is equivalent to the gauge groupoid of Stokes cocycles on $K$.
In particular, the constraints \eqref{eq:wall-algorithm-2cell} coming from $2$-cells generate all functoriality relations; cells of dimension $\ge 3$
impose no additional relations on $\Sec(\pi)$.
\end{proposition}

\begin{proof}
By Theorem~\ref{thm:stratified-stokes-skeleton}, $\Sec(\pi)$ is computed by strict sections of the skeletal projection
$\pi^{\mathrm{sk}}:\GC^{\mathrm{sk}}\to\CB^{\mathrm{sk}}$, where $\CB^{\mathrm{sk}}$ is a groupoid presentation of the fundamental groupoid of $N^\circ$
based on a $2$-skeleton. A strict section is determined by its values on generating $1$-cells and is functorial if and only if the relations coming from
$2$-cells are satisfied. Natural isomorphisms between sections are precisely vertex gauges. This identifies $\Sec(\pi)$ with the gauge groupoid of
cocycles on $K$ as stated.
\end{proof}

\subsubsection*{Micro-example: a transverse wall crossing at depth $2$}

In a depth-$2$ local model $D=\{z_1z_2=0\}$, the angular torus is $(S^1)^2$. If the wall arrangement consists of two transverse walls,
the complement has four chambers and the dual adjacency graph is a $4$-cycle. In the elementary one-parameter situation, each adjacent overlap
carries a Stokes group $(\CC,+)$, so a cocycle is a quadruple $(u_1,u_2,u_3,u_4)\in \CC^4$ subject only to the cycle relation modulo vertex gauge.
This is the simplest instance of the general recipe above (compare \cite{TeyssierSkeletons}).

\subsection{Derived and homotopical upgrades}\label{subsec:applications-derived}

The presenter $\G_{\Phi,n}$ is small and strictly functorial, so it is a convenient input for mapping-stack constructions.
Formally, the nerve $N(-)$ sends bicategorical $2$-pushouts of groupoids to homotopy pushouts of spaces, and mapping out of a homotopy pushout
is a homotopy pullback; see \cite{LurieHTT}. This suggests the following upgrade path, compatible with the derived/homotopical theories of Stokes
data developed by Porta--Teyssier \cite{PortaTeyssierDerived,PortaTeyssierHomotopy}.

\begin{theorem}[Derived mapping stack via the same presenter]\label{thm:derived-upgrade}
Let $\mathcal X$ be an $\infty$-category with finite limits and let $\mathcal H\in \mathcal X$.
Assume we are given a bicategorical $2$-pushout of small groupoids in $\mathbf{Grpd}$
\begin{equation}\label{eq:derived-pushout-presenter}
\G_{\Phi,n}\ \simeq\ \GU\ \bigsqcup_{\GX}\ \GC,
\end{equation}
and write $N(-)$ for the nerve.

\smallskip
\noindent
(1) Mapping out of the presenter is a homotopy pullback: there is a canonical equivalence in $\mathcal X$
\begin{equation}\label{eq:map-out-of-pushout-thm}
\mathrm{Map}\bigl(N(\G_{\Phi,n}),\mathcal H\bigr)\ \simeq\
\mathrm{Map}\bigl(N(\GU),\mathcal H\bigr)\times^{h}_{\mathrm{Map}(N(\GX),\mathcal H)}\mathrm{Map}\bigl(N(\GC),\mathcal H\bigr),
\end{equation}
equivalently the square
\begin{equation}\label{eq:map-out-of-pushout-square}
\begin{tikzcd}[row sep=large, column sep=huge]
\mathrm{Map}\bigl(N(\G_{\Phi,n}),\mathcal H\bigr) \arrow[r] \arrow[d] &
\mathrm{Map}\bigl(N(\GU),\mathcal H\bigr) \arrow[d] \\
\mathrm{Map}\bigl(N(\GC),\mathcal H\bigr) \arrow[r] &
\mathrm{Map}\bigl(N(\GX),\mathcal H\bigr)
\end{tikzcd}
\end{equation}
is homotopy Cartesian.

\smallskip
\noindent
(2) Let $\LocSys_r^{\mathrm{der}}(-)$ be a derived enhancement of $\GL_r$-local systems valued in $\mathcal X$
(for instance, in the sense of Porta--Teyssier \cite{PortaTeyssierHomotopy,PortaTeyssierDerived}),
and let $\StokesLocal_r^{\mathrm{der}}(N^\circ;\Phi)\in\mathcal X$ be a derived enhancement of the collar Stokes moduli,
equipped with a morphism $\StokesLocal_r^{\mathrm{der}}(N^\circ;\Phi)\to \LocSys_r^{\mathrm{der}}(N^\circ)$.
Assume their $1$-truncations recover the classical groupoids:
\begin{equation}\label{eq:derived-inputs-truncate}
\tau_{\le 1}\LocSys_r^{\mathrm{der}}(T)\ \simeq\ \LocSys_r(T),
\qquad
\tau_{\le 1}\StokesLocal_r^{\mathrm{der}}(N^\circ;\Phi)\ \simeq\ \StokesLocal_r(N^\circ;\Phi),
\end{equation}
compatibly with the maps to $\LocSys_r(N^\circ)$.
Define the derived global Stokes moduli by the homotopy pullback
\begin{equation}\label{eq:derived-stokes-def}
\Stokes_r^{\mathrm{der}}(\Xlogn;D,\Phi)
:=
\LocSys_r^{\mathrm{der}}(U)\times^{h}_{\LocSys_r^{\mathrm{der}}(N^\circ)}\StokesLocal_r^{\mathrm{der}}(N^\circ;\Phi).
\end{equation}
Then its $1$-truncation is canonically the classical global Stokes groupoid:
\begin{equation}\label{eq:derived-stokes-truncation}
\tau_{\le 1}\Stokes_r^{\mathrm{der}}(\Xlogn;D,\Phi)\ \simeq\ \Stokes_r(\Xlogn;D,\Phi).
\end{equation}
\end{theorem}

\begin{proof}
For (1), the nerve identifies bicategorical $2$-pushouts of groupoids with homotopy pushouts of their classifying spaces
(indeed, $N:\mathbf{Grpd}\to \mathcal S_{\le 1}$ exhibits the $(2,1)$-category of groupoids, localized at equivalences, as the
$\infty$-category of $1$-types), and in any $\infty$-category mapping out of a homotopy colimit is a homotopy limit; see
\cite[\S1.2, \S4.2]{LurieHTT}. Applying this to the diagram $$N(\GU)\leftarrow N(\GX)\rightarrow N(\GC)$$ yields
\eqref{eq:map-out-of-pushout-thm}, hence \eqref{eq:map-out-of-pushout-square}.

For (2), apply $\tau_{\le 1}$ to \eqref{eq:derived-stokes-def}. Since $1$-truncation is left exact in an $\infty$-topos (hence preserves finite
limits) \cite[\S5.5]{LurieHTT}, we obtain
\begin{equation*}
\tau_{\le 1}\Stokes_r^{\mathrm{der}}(\Xlogn;D,\Phi)
\simeq
\tau_{\le 1}\LocSys_r^{\mathrm{der}}(U)\times^{(2)}_{\tau_{\le 1}\LocSys_r^{\mathrm{der}}(N^\circ)}\tau_{\le 1}\StokesLocal_r^{\mathrm{der}}(N^\circ;\Phi).
\end{equation*}
Using \eqref{eq:derived-inputs-truncate} and the defining fiber product
$$\Stokes_r(\Xlogn;D,\Phi)=\LocSys_r(U)\times^{(2)}_{\LocSys_r(N^\circ)}\StokesLocal_r(N^\circ;\Phi)$$
gives \eqref{eq:derived-stokes-truncation}.
\end{proof}

\begin{remark}\label{rem:derived-upgrade-presenter}
The argument is formal \cite{LurieHTT}, but the point of this paper is that $N(\G_{\Phi,n})$ comes from an explicit small presenter.
Thus $\mathrm{Map}(N(\G_{\Phi,n}),\mathcal H)$ admits concrete combinatorial models once a skeleton is fixed (compare the skeleton philosophy in
\cite{TeyssierSkeletons}), and the same applies to the homotopy pullback description \eqref{eq:map-out-of-pushout-thm}.
\end{remark}

\end{document}